\documentclass[a4paper,12pt]{amsart}
\usepackage{amssymb,amsfonts,amsxtra,amscd,mathrsfs,dsfont}
\usepackage[all]{xy}
\usepackage{fullpage}
\usepackage{graphicx}
\usepackage{yfonts}
\usepackage{setspace}
\theoremstyle{plain}
\newtheorem{theorem}{Theorem}[section]
\newtheorem{lemma}[theorem]{Lemma}
\newtheorem{cor}[theorem]{Corollary}
\newtheorem{prop}[theorem]{Proposition}
\theoremstyle{definition}
\newtheorem{defi}[theorem]{Definition}
\newtheorem{example}[theorem]{Example}
\theoremstyle{remark}
\newtheorem{rem}[theorem]{Remark}
\numberwithin{equation}{section}
\newcommand{\gf}{\ensuremath{\mathbb{K}}}
\newcommand{\sg}[1]{\ensuremath{\mathbb{S}_{#1}}}
\newcommand{\ai}{\ensuremath{A_{\infty}}}
\newcommand{\cyc}[1]{\ensuremath{\mathbb{Z}/{#1}\mathbb{Z}}}
\newcommand{\cycsum}[1]{\ensuremath{\Theta_{#1}}}
\newcommand{\mat}[2]{\ensuremath{\mathrm{M}_{#1}\left(#2\right)}}
\newcommand{\cotimes}{\ensuremath{\hat{\otimes}}}
\newcommand{\innprod}[1][-,-]{\ensuremath{\langle #1 \rangle}}
\newcommand{\innprodloc}[1][-,-]{\ensuremath{\langle #1 \rangle_{\mathrm{loc}}}}
\newcommand{\dblinnprod}[1][-,-]{\ensuremath{\langle\!\langle #1 \rangle\!\rangle}}
\newcommand{\OTFT}[2]{\ensuremath{\mathbf{T}^{#1}_{#2}}}
\newcommand{\Morita}[1][\textgoth{A}]{\ensuremath{\mathcal{M}_{\gamma,\nu}^{#1}}}
\newcommand{\den}[1]{\ensuremath{|\Lambda_{#1}|}}

\newcommand{\Derdelim}[1]{\ensuremath{\mathrm{Der}#1}}

\newcommand{\DerLdelim}[1]{\ensuremath{\mathrm{Der}_{\mathrm{Loc}}#1}}
\newcommand{\csalg}[1]{\ensuremath{\widehat{S}\left(#1\right)}}
\newcommand{\csalgdelim}[1]{\ensuremath{\widehat{S} #1}}
\newcommand{\csalgP}[1]{\ensuremath{\widehat{S}_+\left(#1\right)}}
\newcommand{\csalgPdelim}[1]{\ensuremath{\widehat{S}_+ #1}}
\newcommand{\ctalg}[1]{\ensuremath{\widehat{T}\left(#1\right)}}
\newcommand{\ctalgdelim}[1]{\ensuremath{\widehat{T}#1}}
\newcommand{\ctalgP}[1]{\ensuremath{\widehat{T}_+\left(#1\right)}}
\newcommand{\smooth}[1]{\ensuremath{C^{\infty}(#1)}}
\newcommand{\smoothlimit}[1]{\ensuremath{C^{\infty}_{\geq 0}(#1)}}
\newcommand{\smoothsing}[1]{\ensuremath{C^{\infty}_{<0}(#1)}}
\newcommand{\dRham}[2][\bullet]{\ensuremath{\Omega^{#1}\left(#2\right)}}

\newcommand{\Locdelim}[1]{\ensuremath{\mathrm{Loc}#1}}
\newcommand{\ptree}[1][p]{\ensuremath{\mathscr{T}_{#1}}}
\newcommand{\intcomm}[1]{\ensuremath{\mathscr{O}_{\hbar}\left(#1\right)}}
\newcommand{\intcommdelim}[1]{\ensuremath{\mathscr{O}_{\hbar}#1}}
\newcommand{\intcommP}[1]{\ensuremath{\mathscr{O}_{\hbar}^+\left(#1\right)}}
\newcommand{\intcommI}[1]{\ensuremath{\mathscr{O}_{\hbar}^{\mathrm{Int}}\left(#1\right)}}
\newcommand{\intcommL}[1]{\ensuremath{\mathscr{O}_{\hbar}^{\mathrm{Loc}}\left(#1\right)}}
\newcommand{\intcommLI}[1]{\ensuremath{\mathscr{O}_{\hbar}^{\mathrm{LocInt}}\left(#1\right)}}
\newcommand{\intnuP}[1]{\ensuremath{\mathscr{N}_{\nu}\left(#1\right)}}
\newcommand{\intnoncomm}[1]{\ensuremath{\mathscr{N}_{\gamma,\nu}\left(#1\right)}}
\newcommand{\intnoncommdelim}[1]{\ensuremath{\mathscr{N}_{\gamma,\nu}#1}}
\newcommand{\intnoncommP}[1]{\ensuremath{\mathscr{N}_{\gamma,\nu}^+\left(#1\right)}}
\newcommand{\intnoncommI}[1]{\ensuremath{\mathscr{N}_{\gamma,\nu}^{\mathrm{Int}}\left(#1\right)}}
\newcommand{\intnoncommL}[1]{\ensuremath{\mathscr{N}_{\gamma,\nu}^{\mathrm{Loc}}\left(#1\right)}}
\newcommand{\intnoncommLP}[1]{\ensuremath{\mathscr{N}_{\gamma,\nu}^{\mathrm{Loc}+}\left(#1\right)}}
\newcommand{\intnoncommLI}[1]{\ensuremath{\mathscr{N}_{\gamma,\nu}^{\mathrm{LocInt}}\left(#1\right)}}
\newcommand{\intnoncommItree}[1]{\ensuremath{\mathscr{N}_{\gamma,\nu}^{\mathrm{Tree}}\left(#1\right)}}
\newcommand{\intnoncommLItree}[1]{\ensuremath{\mathscr{N}_{\gamma,\nu}^{\mathrm{LocTree}}\left(#1\right)}}
\newcommand{\intnoncommconstants}{\ensuremath{\mathscr{N}_{\gamma,\nu}}}
\newcommand{\intnoncommIconstants}{\ensuremath{\mathscr{N}_{\gamma,\nu}^{\mathrm{Int}}}}
\newcommand{\KontHam}[1]{\ensuremath{\mathscr{H}\left(#1\right)}}
\newcommand{\KontHamP}[1]{\ensuremath{\mathscr{H}_+\left(#1\right)}}
\newcommand{\KontHamL}[1]{\ensuremath{\mathscr{H}^{\mathrm{Loc}}\left(#1\right)}}
\newcommand{\KontHamPL}[1]{\ensuremath{\mathscr{H}_+^{\mathrm{Loc}}\left(#1\right)}}
\newcommand{\NCPreThy}[1]{\ensuremath{\mathbf{NCPreThy}\left(#1\right)}}
\newcommand{\NCPreThyL}[2]{\ensuremath{\mathbf{NCPreThy}_{#1}\left(#2\right)}}
\newcommand{\PreThy}[1]{\ensuremath{\mathbf{PreThy}\left(#1\right)}}
\newcommand{\NCThy}[1]{\ensuremath{\mathbf{NCThy}\left(#1\right)}}
\newcommand{\NCThyL}[2]{\ensuremath{\mathbf{NCThy}_{#1}\left(#2\right)}}
\newcommand{\Thy}[1]{\ensuremath{\mathbf{Thy}\left(#1\right)}}
\DeclareMathOperator{\Hom}{Hom}
\DeclareMathOperator{\Aut}{Aut}
\DeclareMathOperator{\Bij}{Bij}
\DeclareMathOperator{\Ker}{Ker}
\DeclareMathOperator{\Image}{Im}
\DeclareMathOperator{\Tr}{Tr}

\begin{document}
\title{The Batalin-Vilkovisky formalism in noncommutative effective field theory}
\author{Alastair Hamilton}
\address{Department of Mathematics and Statistics, Texas Tech University, Lubbock, TX 79409-1042. USA.} \email{alastair.hamilton@ttu.edu}
\begin{abstract}
We address the treatment of gauge theories within the framework that is formed from combining the machinery of noncommutative symplectic geometry, as introduced by Kontsevich, with Costello's approach to effective gauge field theories within the Batalin-Vilkovisky formalism; discussing the problem of quantization in this context, and identifying the relevant cohomology theory controlling this process. We explain how the resulting noncommutative effective gauge field theories produce classes in a compactification of the moduli space of Riemann surfaces, when we pass to the large length scale limit. Within this setting, the large $N$ correspondence of 't~Hooft---describing a connection between open string theories and gauge theories---appears as a relation between the noncommutative and commutative geometries. We use this correspondence to investigate and ultimately quantize a noncommutative analogue of Chern-Simons theory.
\end{abstract}
\makeatletter
\@namedef{subjclassname@2020}{%
  \textup{2020} Mathematics Subject Classification}
\makeatother
\maketitle
\footnotesize
\begin{spacing}{0.86}
\tableofcontents
\end{spacing}
\normalsize

\section{Introduction}

This paper continues the work that was started in \cite{NCRGF}, in which we integrated the machinery of noncommutative geometry that was introduced by Kontsevich in \cite{KontSympGeom}, with the framework of effective field theory that was developed by Costello in \cite{CosEffThy}. For the object that was produced from this marriage, we introduced the term `noncommutative effective field theory'---a term that is potentially somewhat at odds with the other types of structures that use the same title, as we will later explain. In this paper, we apply this framework of noncommutative effective field theory to the study of gauge theories in the Batalin-Vilkovisky (BV) formalism, following Costello's treatment in \cite{CosBVrenormalization, CosEffThy} of the commutative case.

The appearance of noncommutative geometry within the framework of effective field theory is a phenomena that may in this instance be explained as one which emerges as a consequence of replacing the usual world-\emph{line} picture of Feynman diagrams, for one in which the Feynman diagrams are formed by open \emph{strings}. Such diagrams are represented by \emph{ribbon graphs} \cite[\S 5.5]{OpAlgTopPhys}, which exhibit a form of \emph{cyclic} symmetry at each vertex, rather than the full group of symmetries provided by the symmetric group. Ultimately, this leads to a type of noncommutative geometry that is based on taking quotients by actions of the cyclic groups. As we will eventually observe, this phenomena will trace a path from noncommutative effective field theories, to classes in the moduli space of Riemann surfaces. Separately, we will also come to see that it provides us with a natural formalism in which to investigate large $N$ phenomena.

\subsection{Background}

\subsubsection{Noncommutative effective field theory}

The definition of an effective field theory that we will use for the rest of this paper was introduced by Costello in \cite{CosEffThy}, and was based on the work of Kadanoff \cite{KadScaling}, Wilson \cite{WilRGcrit, WilRenormScalar}, Polchinski \cite{PolRenormLag} and others. It consists of a family of (suitably local) interactions parameterized by a length (equivalently, energy) scale cut-off that are connected by the renormalization group flow. The definition of a \emph{noncommutative} effective field theory introduced in \cite{NCRGF}, is essentially the same, except that now our interactions live in a space based on the noncommutative geometry of \cite{KontSympGeom}, and the renormalization group flow is now defined using (stable) ribbon graphs. Many of the results proved by Costello in the commutative framework \cite{CosEffThy}, also hold in the noncommutative framework \cite{NCRGF}.

We should be careful to emphasize that the manner in which the noncommutative geometry that we study in this paper arises, is from replacing ordinary graphs with ribbon graphs---or equivalently, points with open strings---and not, per se, from some deformation quantization of the spacetime itself. Since the latter interpretation of the term `noncommutative effective field theory' appears to be quite common \cite{NoteonNCCSThy}, we do not wish to be misunderstood on this point; nevertheless, we feel that adopting this nomenclature for the structures that we study in this paper is appropriate, and is justified by the resulting efficiency of exposition.

\subsubsection{Batalin-Vilkovisky formalism}

The `Batalin-Vilkovisky formalism' is the eponymous term for the framework introduced in \cite{BaViGauge} for the quantization of gauge theories. In this framework, the space of fields is extended to include the (infinitesimal) generators of the gauge transformations (known as ghosts). The \emph{classical master equation} for the extended action is then the statement that the action functional on the space of fields is invariant under the symmetries defined by the gauge group.

In order for the path integral to have a well-defined value that does not depend upon the choice of a gauge-fixing condition, a constraint must be imposed on the action known as the \emph{quantum master equation}. Typically, constructing a solution to this equation involves deforming the original action. A solution that is constructed according to this protocol is called a \emph{quantization} of the original theory.

The treatment of the BV formalism within effective field theory---particularly the subject of the quantum master equation---is quite technical. A comprehensive discussion of the problem was provided by Costello in \cite[\S 5]{CosEffThy}. Here, he specified a well-defined version of the quantum master equation at every length scale, and proved that the renormalization group flow connects the solutions at these different length scales. This allowed him to extend his definition of an effective field theory to the setting of gauge theories.

In this paper we establish noncommutative analogues of these results. We provide a definition for the classical master equation in this framework based on the noncommutative symplectic geometry of Kontsevich \cite{KontSympGeom}. The constraint that we arrive at in this manner turns out to be slightly stronger than the naive condition of gauge invariance that we might otherwise impose, cf. Proposition \ref{prop_CMEgauge}. Likewise, we formulate a version of the quantum master equation at every length scale, and prove that the renormalization group flow transforms one solution into another. The formulation for the quantum master equation that we employ comes from studying the degeneration of Riemann surfaces upon approaching the boundary of the moduli space \cite{baran, HamCompact, KontAiry, LooCompact}.

Quantization, as mentioned above, is the process of deforming a solution to the classical master equation into a solution to the quantum master equation. In general, such a deformation problem is typically controlled by some underlying cohomology theory, cf. \cite{GerDeform}. In this paper we identify the cohomology theory that controls this process in noncommutative effective field theory, applying the methods used by Costello in \cite{CosEffThy}. One immediate application that we demonstrate is the independence of the notion of a theory in this setting upon the choice of a gauge-fixing condition. This is expressed formally by the statement that the simplicial set of theories forms a Kan fibration over the set of gauge-fixing conditions, cf. \cite[\S 5.11.2]{CosEffThy}.

\subsubsection{The large $N$ string/gauge theory correspondence}

The connection between string theories and the large $N$ phenomena of gauge theories goes back to the work of 't~Hooft \cite{tHooftplanar}, in which he showed that when calculations in $U(N)$ gauge theories are expanded in powers of the rank $N$, the terms in the perturbation series are indexed by ribbon graphs---which describe diagrams formed by open strings---with the planar diagrams dominating the expansion.

This general phenomena is captured in the notion of a noncommutative effective field theory \cite[\S 5.1]{NCRGF}. Here, the renormalization group flow is defined using ribbon graphs, and the noncommutative geometry is intimately connected to the structure of these graphs. To explain the link in more detail, we must describe two transformations that may be applied to noncommutative effective field theories.

The first transformation takes any noncommutative theory and produces a commutative theory in the sense of Costello \cite{CosEffThy}. This transformation has a fairly simple description, which basically amounts to forgetting the extra structure present on the noncommutative side.

The second transformation is derived from the structure maps of a two-dimensional Open Topological Field Theory (OTFT), which---as is well-known \cite{AtiyahTFT}---is the same thing as a Frobenius algebra. The space of fields is transformed by forming the tensor product with this Frobenius algebra. Hence, when we apply this transformation to a space of fields based on the de Rham algebra, and the Frobenius algebra that we use is the algebra of $N$-by-$N$ matrices, the space of fields that we will produce will describe connections---the principal objects of gauge theories.

Both of the two transformations described above are compatible with the structures that facilitate the machinery of the BV formalism, so that when we combine them, we have a way to produce a large $N$ family of gauge theories from a single noncommutative theory; see Diagram \eqref{dig_largeN}.

In principle, this allows us to study the large $N$ phenomena of gauge theories (or at least, those based on the gauge group $U(N)$) by studying the single noncommutative theory that generates them via this correspondence. A concrete example is provided in Section \ref{sec_powerseries}; a powerseries may be associated to any theory modulo constants, and the powerseries for the gauge theories may be computed from the powerseries for the noncommutative theory by making a substitution that depends upon the rank $N$.

Furthermore, there is a way to go backwards, and deduce properties of the noncommutative theory from the corresponding gauge theories. This mechanism is provided (somewhat indirectly) by the Loday-Quillen-Tsygan Theorem \cite{LodayQuillen, Tsygan}. It allows us, for instance, to deduce that the interaction for a noncommutative theory satisfies the quantum master equation if and only if the same is true for the gauge theories that it generates; see Section \ref{sec_vanishcriteria}. This will be applied later when we come to consider a noncommutative analogue of Chern-Simons theory.

\subsubsection{Algebraic structures on the cohomology of the fields---links to classes in the moduli space of Riemann surfaces}

In \cite{CosEffThy}, Costello described an algebraic structure that is produced by an effective gauge field theory on the cohomology of the fields in the length scale limit $L\to \infty$. This structure is a solution to the quantum master equation at the infinite length scale. (This quantum master equation, it should be mentioned, is in fact easier to define, as the cohomology will be finite-dimensional, in contrast to the space of fields itself.) The algebraic structures so produced may be described as an enhancement of the notion of a (cyclic) $L_{\infty}$-algebra \cite{ZwiClosedstring}. From such algebraic structures, it is possible to produce a class in a chain complex that is spanned by oriented graphs, cf. \cite{KontFeyn}. In general, graph homology classes are known to be related, through Chern-Simons theory \cite{KontVass, KontFeyn}, to invariants of knots and links.

This entire story has a noncommutative counterpart. In this paper, we demonstrate that a noncommutative theory will produce a solution to the (noncommutative) quantum master equation on the cohomology of the fields at the infinite length scale limit. The algebraic structures that are produced in this way may be described as an enhancement of the notion of a (cyclic) $A_{\infty}$-algebra \cite{HiroCyclic}. From such algebraic structures, it is possible to produce classes in chain complexes spanned by ribbon graphs, cf. \cite{baran, HamQME, KontFeyn}; or equivalently, by the results of Harer \cite{HarOrbicells}, Mumford, Penner \cite{PenOrbicells} and Thurston, in certain compactifications of the moduli space of Riemann surfaces.

The two stories told above are compatible with the preceding picture of the large $N$ correspondence that we have just sketched. There is a natural map from the ribbon graph complex that is used by the noncommutative theories, to the graph complex used by the commutative theories. If a large $N$ family of gauge theories is generated by a single noncommutative theory under the large $N$ correspondence described above, then the same can be said for the respective classes in the graph complexes that we have just defined, with a diagram similar to Diagram \eqref{dig_largeN} describing the relation. If we consider for a moment, the relationship between Gromov-Witten invariants and  Reshethikin-Turaev link invariants that was proposed by Gopakumar-Vafa in \cite{GopVafa} (see \cite{AuKoGVLargeN} and Equation (26) of \cite{BriGOVconjecture}), it seems natural to view this whole picture that we have just described as a more general incarnation of this sort of phenomena.

\subsubsection{Noncommutative Chern-Simons theory}

As an application of the ideas and methods that we develop in this paper, we examine a noncommutative analogue of Chern-Simons theory. This noncommutative analogue was introduced in \cite{NCRGF}, where it arises as a fairly simple-minded generalization of ordinary Chern-Simons theory in which world-lines are replaced with open strings---a picture consistent, for instance, with that described in~\cite{CosTCFT}. In \cite{NCRGF} we demonstrated that the effective gauge field theories that this noncommutative Chern-Simons theory generates under the large $N$ correspondence, were precisely the $U(N)$ Chern-Simons theories.

We consider the problem of quantizing this noncommutative Chern-Simons theory within our formalism, and we show that in the presence of a flat Riemannian metric, such a canonical quantization exists (modulo constants); in fact, there is no need to deform the tree-level Chern-Simons action, which provides by itself the required solution to the quantum master equation.

This particular result was shown in \cite[\S 15]{CosBVrenormalization} to hold in the commutative case for the gauge groups $U(N)$ (amongst others). The proof in that case relies to a great degree on the work of Kontsevich in \cite{KontFeyn}, and his analysis of certain integrals defined over compactifications of configuration spaces of points. It is reasonable to presume that a similar analysis may be carried out for our noncommutative Chern-Simons theory, and that this would yield the same desired result---in this context, we mention the recent work of Cieliebak-Volkov \cite{CiVoCSStrTop}. However, we prefer to deduce this result as a simple consequence of the large $N$ correspondence using the vanishing criteria that may be deduced from the Loday-Quillen-Tsygan Theorem.

It should be mentioned that in the current paper, we limit ourselves to working over a compact manifold, and that there are not many flat compact three-manifolds (in fact, there are precisely six orientable ones). In principle, dealing with curved spacetimes should not present a problem. In \cite{CosEffThy} Costello proves that the set of theories forms a sheaf over the spacetime manifold. Since flat metrics are always available locally, this implies the existence of a quantization over any curved spacetime. The same should hold true in the noncommutative case; however, such a treatment, and the extra technology it would require us to introduce, would add considerably to the length of the current article. For this reason, we plan to take up this problem in a separate paper.

Adding to the list of problems that lie outside the scope of the current paper, it is worth mentioning that the same techniques applied above to Chern-Simons theory, should also be available for Yang-Mills theory; that is, there should be a noncommutative analogue of Yang-Mills theory that generates the usual $U(N)$ Yang-Mills gauge theories under the large $N$ correspondence, and with which we might hope to be able to analyze some of their large $N$ phenomena.

\subsection{Acknowledgements}

The author would like to thank Owen Gwilliam, who first introduced the author to the subject of large $N$ limits some years ago. The author would also like to thank Dmitri Pavlov and Mahmoud Zeinalian for a number of helpful conversations.

\subsection{Layout of the paper}

We begin in Section \ref{sec_freethy} by setting up some of the basic framework for what will follow. We provide the definition for a free theory in the Batalin-Vilkovisky formalism, including the definition for a gauge-fixing condition,  following those definitions provided in \cite{CosEffThy}. Additionally, we collect some important identities for the heat kernel of a generalized Laplacian.

In Section \ref{sec_ClassicBV}, we define and investigate the classical master equation in the BV formalism from the perspective of Kontsevich's noncommutative symplectic geometry \cite{KontSympGeom}. We give an interpretation in noncommutative geometry---which may be of some independent interest---of the action of the ghosts on the space of fields, and of the gauge symmetry constraint on the action functional that is determined by the classical master equation. We then explain how to construct a noncommutative analogue of Chern-Simons theory using the protocol and methods described in this section.

In Section \ref{sec_RGFlow} we recall the significant body of material that we will need from \cite{NCRGF}. We begin with the definition and properties of the renormalization group flow, which in turn leads to the definition of a noncommutative effective field theory (which we refer to as a pretheory in this article). We then recall the main theorem stating that there is a one-to-one correspondence between pretheories and local functionals.

In Section \ref{sec_QuantumBV}, we bring in Costello's perspective from \cite[\S 5.9]{CosEffThy} and define the BV quantum master equation in the framework of noncommutative geometry by matching it to the length scale regularization parameter of the theory. We then provide a description of this quantum master equation in terms of Feynman diagrams and operations that contract the edges of a ribbon graph. We subsequently use this formulation to prove one of our main theorems, Theorem \ref{thm_QMEflow}, a formula describing the compatibility of the quantum master equation with the renormalization group flow. Finally, we make use of those results from \cite{NCRGF} that arise from an application of the Loday-Quillen-Tsygan Theorem, to give an equivalent characterization of the noncommutative quantum master equation in terms of the commutative quantum master equation. This description will be employed later when we come to discuss the large $N$ correspondence.

In Section \ref{sec_NCEffThyBV}, we provide the definition of a noncommutative effective field theory within the framework of the BV formalism, and use this to formulate a version of the large $N$ correspondence. Separately, we give a sufficient condition for a noncommutative theory to produce a solution to the infinite length scale quantum master equation in the limit $L\to \infty$, and hence to produce a class in (a compactification of) the moduli space of Riemann surfaces. Next, we identify and develop the obstruction theory that describes the quantization process, and use this to prove the independence of a theory on the choice of a gauge-fixing condition.

Lastly, in Section \ref{sec_NCCSTheory} we introduce and study our noncommutative analogue of Chern-Simons theory, and prove that it has a natural quantization in a flat metric.

At the end of the paper, we include a short appendix on Topological Vector Spaces in which we collect a number of basic definitions and results.

\subsection{Notation and conventions}

Throughout the paper we work over a ground field $\mathbb{K}$ which is either the real numbers $\mathbb{R}$ or the complex numbers $\mathbb{C}$, and with $\mathbb{Z}$-graded locally convex Hausdorff topological vector spaces over $\gf$. We will denote the algebra of $N$-by-$N$ matrices with entries in $\gf$ by $\mat{N}{\gf}$. Our convention will be to work with cohomologically graded spaces. Consequently, we define the suspension of a graded topological vector space $\mathcal{V}$ by $\Sigma\mathcal{V}^i:=\mathcal{V}^{i+1}$.

Given topological vector spaces $\mathcal{V}$ and $\mathcal{W}$, we will denote the space of \emph{continuous} $\gf$-linear maps by
\[ \Hom_{\gf}(\mathcal{V},\mathcal{W}). \]
We denote the continuous $\gf$-linear dual of $\mathcal{V}$ by $\mathcal{V}^{\dag}$. These spaces always carry the strong topology of uniform convergence on bounded sets.

We will denote the completed projective tensor product of two locally convex Hausdorff topological vector spaces $\mathcal{V}$ and $\mathcal{W}$ by
\[ \mathcal{V}\cotimes\mathcal{W}. \]
We will reserve the notation $\otimes$ for the ordinary algebraic tensor product. The permutation
\[ \mathcal{V}\cotimes\mathcal{W}\to\mathcal{W}\cotimes\mathcal{V}, \qquad v\otimes w \mapsto (-1)^{|v||w|}w\otimes v \]
will be denoted by $\tau_{\mathcal{V},\mathcal{W}}$ (or frequently, just by $\tau$).

Given a locally convex Hausdorff topological vector space $\mathcal{V}$, we define the completed symmetric and tensor algebras of $\mathcal{V}$---along with their nonunital counterparts---by:
\begin{align}
\label{eqn_symmalg}
\csalg{\mathcal{V}} &:= \prod_{n=0}^\infty \bigl[\mathcal{V}^{\cotimes n}\bigr]_{\sg{n}} & \csalgP{\mathcal{V}} &:= \prod_{n=1}^\infty \bigl[\mathcal{V}^{\cotimes n}\bigr]_{\sg{n}} \\
\label{eqn_tensalg}
\ctalg{\mathcal{V}} &:= \prod_{n=0}^\infty \mathcal{V}^{\cotimes n} & \ctalgP{\mathcal{V}} &:= \prod_{n=1}^\infty \mathcal{V}^{\cotimes n}.
\end{align}
Here the symmetric group is denoted by $\sg{n}$. Throughout the paper we adopt the convention of denoting coinvariants using a subscript and invariants using a superscript.

We will denote the algebra of smooth $\gf$-valued functions on a smooth manifold $M$ by $\smooth{M,\gf}$, with smooth real-valued functions being denoted simply by $\smooth{M}$. More generally, if $E$ is a vector bundle over $M$ then $\mathcal{E}:=\Gamma(M,E)$ will denote the space of smooth sections of this bundle. We equip $\mathcal{E}$ with the $C^{\infty}$-topology of uniform convergence of sections and their derivatives on compact sets that is defined by Definition \ref{def_Cinftopology}. With this topology, $\mathcal{E}$ is a nuclear Fr\'echet space. The de Rham algebra of $\gf$-valued forms on $M$ will be denoted by $\dRham{M,\gf}$, with real forms simply denoted by $\dRham{M}$.

We will denote the number of elements in a finite set $X$ by $|X|$. In formulas, we will use $\mathds{1}$ to denote the identity map on a set.

\section{Free theories} \label{sec_freethy}

Here we will introduce, following \cite{CosEffThy}, the definition of a free theory in the framework of the Batalin-Vilkovisky (BV) formalism. This will provide the initial framework necessary to begin describing interacting theories and their quantization.

\subsection{Free BV-theories and gauge-fixing conditions}

\subsubsection{Basic definitions}

\begin{defi}
A \emph{free BV-theory} consists of the following:
\begin{enumerate}
\item
A smooth compact manifold $M$.
\item
A $\mathbb{Z}$-graded vector bundle $E$ over $M$. The space of smooth sections of this bundle will be denoted by $\mathcal{E}:=\Gamma(M,E)$.
\item
A local pairing on $E$, of degree minus-one. This is a map of vector bundles
\[ \innprodloc:E\otimes E \to \den{M}\underset{\mathbb{R}}{\otimes}\gf. \]
Here $\den{M}$ denotes the real density bundle concentrated in degree zero. The pairing over each fiber must be nondegenerate and graded skew-symmetric.
\item
A differential operator $Q:\mathcal{E}\to\mathcal{E}$ of degree one satisfying $Q^2=0$.
\end{enumerate}

The local pairing $\innprodloc$ defines an integration pairing
\[ \innprod:\mathcal{E}\otimes\mathcal{E}\to\gf, \qquad  \innprod[s_1,s_2]:=\int_M \innprodloc[s_1,s_2]. \]
We impose the compatibility requirement
\[ \innprod[Qs_1,s_2] +(-1)^{|s_1|}\innprod[s_1,Qs_2] = 0, \quad\text{for all }s_1,s_2\in\mathcal{E}. \]
\end{defi}

Note that the pairing $\innprod$ is nondegenerate in the sense that the map
\[ \mathcal{E}\to\mathcal{E}^{\dag}, \qquad s\mapsto\innprod[s,-] \]
is injective.

The kinetic term for a free BV-theory will be given by
\begin{equation} \label{eqn_kineticterm}
\frac{1}{2}\innprod[Qs,s], \quad s\in\mathcal{E}^0.
\end{equation}
To make sense of this requires a choice of gauge.

\begin{defi} \label{def_gaugefixing}
A \emph{gauge-fixing operator} for a free BV-theory $\mathcal{E}$ is a differential operator
\[ Q^{\mathrm{GF}}:\mathcal{E}\to\mathcal{E} \]
of degree minus-one such that:
\begin{itemize}
\item
$(\mathcal{E},Q^{\mathrm{GF}})$ is a chain complex, that is $(Q^{\mathrm{GF}})^2=0$.
\item
$Q^{\mathrm{GF}}$ is self-adjoint, that is
\[ \innprod[Q^{\mathrm{GF}}s_1,s_2] = (-1)^{|s_1|}\innprod[s_1,Q^{\mathrm{GF}}s_2], \quad\text{for all }s_1,s_2\in\mathcal{E}. \]
\item
The operator
\[ H:=[Q,Q^{\mathrm{GF}}]=QQ^{\mathrm{GF}}+Q^{\mathrm{GF}}Q \]
on $\mathcal{E}$ is a generalized Laplacian; that is, a second order differential operator whose principal symbol $\sigma_2(H)$ is multiplication by a Riemannian metric $g$ on $M$, see \cite[\S 2.1]{BerGetVer}.
\end{itemize}
\end{defi}

In order for such a gauge-fixing to exist, it is necessary for the complex $(\mathcal{E},Q)$ to be elliptic.

\begin{rem}
If $\mathcal{E}$ is a free BV-theory and $Q^{\mathrm{GF}}$ is a gauge-fixing operator for it, then the generalized Laplacian $H$ defined above completes the list of structures required for $\mathcal{E}$ to form a free theory in the sense of Definition 2.1 of \cite{NCRGF}. This allows us to apply the results of \cite{NCRGF} throughout the rest of the paper.
\end{rem}

One example which we will be of particular interest in this paper is Chern-Simons theory.

\begin{example} \label{exm_CStheory}
We start with a Lie algebra $\mathfrak{g}$ with a symmetric nondegenerate pairing $\innprod_{\mathfrak{g}}$ and consider the graded vector bundle
\[E:=\Sigma\Lambda^{\bullet} T^*M\underset{\mathbb{R}}{\otimes}\mathfrak{g} \]
on a compact oriented manifold $M$ of dimension three. The local pairing on $E$ is defined using the pairing $\innprod_{\mathfrak{g}}$ on $\mathfrak{g}$ along with the exterior multiplication on $T^*M$ and orientation of $M$; note that this becomes skew-symmetric following the shift in parity. The space of fields is then
\[ \mathcal{E}=\Sigma\Omega^{\bullet}(M)\underset{\mathbb{R}}{\otimes}\mathfrak{g} \]
and taking the differential operator $Q$ to be the exterior derivative completes the list of structures required for a free BV-theory.

A gauge-fixing operator may be defined by choosing a Riemannian metric $g$ on $M$. We may then define $Q^{\mathrm{GF}}$ to be the usual Hodge adjoint of the exterior derivative $Q$. Since the dimension of $M$ is odd, $Q^{\mathrm{GF}}$ will be self-adjoint with respect to the pairing on $\mathcal{E}$.

Note that everything works just as above over a manifold of arbitrary odd-dimension, so long as we are willing to give up our $\mathbb{Z}$-grading for a $\mathbb{Z}/2\mathbb{Z}$-grading.
\end{example}

\subsubsection{Families of gauge-fixing conditions} \label{sec_gaugefixingfamily}

Since there will be more than one choice of gauge-fixing operator for any given free BV-theory, we will ultimately have to account for the effects of varying the gauge-fixing condition. This leads us to consider parameterized families of gauge-fixing operators.

Our parameter space $X$ will be a smooth manifold with corners. For our purposes, it suffices to assume that $X$ is an $n$-simplex $\Delta^n$. We then consider the graded commutative topological $\gf$-algebra $\mathcal{A}:=\Gamma(X,A)$ formed by taking the global sections of some sheaf $A$ of graded commutative $\gf$-algebras. We will denote the commutative multiplication on $\mathcal{A}$ by $\mu_{\mathcal{A}}$. Again, for our purposes we may limit our consideration to the case when $\mathcal{A}$ is either the algebra of smooth functions $\smooth{X,\gf}$ or the de Rham algebra $\dRham{X,\gf}$.

Those uninterested in considering families of gauge-fixing conditions may take $X$ to be a point, in which case $\mathcal{A}=\gf$ and hence $\mathcal{A}$ consequently disappears from all formulas. More generally, any point $p\in X$ determines a morphism of algebras from $\mathcal{A}$ to $\gf$. This allows us to pick out the gauge-fixing condition assigned to that point.

\begin{defi} \label{def_gaugefixingfamily}
Let $\mathcal{E}$ be a free BV-theory and $\mathcal{A}$ be as above. Suppose that
\[ Q^{\mathrm{GF}}:\mathcal{E}\to\mathcal{E}\cotimes\mathcal{A} \]
is an operator of degree minus-one and, given any point $p\in X$, denote the operator formed by composing $Q^{\mathrm{GF}}$ with the map from $\mathcal{A}$ to $\gf$ determined by the point $p$ by
\[ Q^{\mathrm{GF}}_p:\mathcal{E}\to\mathcal{E}. \]

We say that $Q^{\mathrm{GF}}$ is a \emph{family of gauge-fixing operators} if:
\begin{itemize}
\item
The canonical $\mathcal{A}$-linear extension
\[ (\mathds{1}\cotimes\mu_{\mathcal{A}}) (Q^{\mathrm{GF}}\cotimes\mathds{1}) :\mathcal{E}\cotimes\mathcal{A}\to\mathcal{E}\cotimes\mathcal{A}, \]
which we also denote by $Q^{\mathrm{GF}}$, is a differential operator.
\item
$(\mathcal{E}\cotimes\mathcal{A},Q^{\mathrm{GF}})$ is a chain complex, that is $(Q^{\mathrm{GF}})^2=0$.
\item
$Q^{\mathrm{GF}}$ is self-adjoint with respect to the pairing
\begin{equation} \label{eqn_innprodA}
\innprod_{\mathcal{A}} := \bigl(\innprod\cotimes\mu_{\mathcal{A}}\bigr) \bigl(\mathds{1}\cotimes\tau\cotimes\mathds{1}\bigr) :\bigl(\mathcal{E}\cotimes\mathcal{A}\bigr)\cotimes\bigl(\mathcal{E}\cotimes\mathcal{A}\bigr) \to \mathcal{A}.
\end{equation}
\item
For every point $p\in X$,
\[ H_p:=[Q^{\mathrm{GF}}_p,Q]:\mathcal{E}\to\mathcal{E} \]
is a generalized Laplacian.
\end{itemize}
\end{defi}

Consequently, if $Q^{\mathrm{GF}}$ is a family of gauge-fixing operators then $Q^{\mathrm{GF}}_p$ will be a gauge-fixing operator in the sense of Definition \ref{def_gaugefixing} for every point $p\in X$.

\begin{example} \label{exm_gaugefixing}
Suppose that $X:=\Delta^n$ and that we have a degree minus-one operator
\[ Q^{\mathrm{GF}}:\mathcal{E}\to\mathcal{E}\cotimes\smooth{X,\gf}. \]
Set $\mathcal{A}:=\dRham{X,\gf}$ and consider $Q^{\mathrm{GF}}$ as an operator from $\mathcal{E}$ to $\mathcal{E}\cotimes\mathcal{A}$, as in Definition \ref{def_gaugefixingfamily} above. Then $Q^{\mathrm{GF}}$ forms a family of gauge-fixing operators if and only if:
\begin{itemize}
\item
the canonical $\smooth{X,\gf}$-linear extension
\[ (\mathds{1}\cotimes\mu) (Q^{\mathrm{GF}}\cotimes\mathds{1}) :\mathcal{E}\cotimes\smooth{X,\gf}\to\mathcal{E}\cotimes\smooth{X,\gf} \]
is a differential operator, and
\item
for every point $p\in X$, the operator $Q^{\mathrm{GF}}_p$ is a gauge-fixing operator in the sense of Definition \ref{def_gaugefixing}.
\end{itemize}
\end{example}

In practice, Example \ref{exm_gaugefixing} will be how we generate our families of gauge-fixing operators.

\begin{example} \label{exm_metricgaugefixing}
Suppose that $g_p$, $p\in X$ is a family of Riemannian metrics on a compact oriented Riemannian manifold $M$ of dimension $n$; that is, a metric
\[ g\in\Gamma\bigl(M\times X,(T^*M)^{\otimes 2}\bigr) \]
on the pullback of the tangent bundle along the projection map. Using the orientation and the metric on this bundle gives us a star operator
\[ \ast:\Gamma(M\times X,\Lambda^k T^* M)\longrightarrow\Gamma(M\times X,\Lambda^{n-k} T^* M) \]
in the usual way. Note that we may identify
\[ \Gamma(M\times X,\Lambda^{\bullet} T^* M) = \dRham{M}\cotimes\smooth{X} \]
using Proposition \ref{prop_tensorsections}, and that $\ast$ is obviously $\smooth{X}$-linear as it is a map of vector bundles. We may then define
\[ Q^{\mathrm{GF}}:=(-1)^{n(k+1)+1}\ast(Q\cotimes\mathds{1})\ast :\dRham[k]{M}\cotimes\smooth{X}\longrightarrow\dRham[k-1]{M}\cotimes\smooth{X}, \]
where $Q$ is the exterior derivative. This will define a family of gauge-fixing conditions for Chern-Simons theory, as described in Example \ref{exm_CStheory}, using the process described in Example \ref{exm_gaugefixing}.
\end{example}

\begin{rem}
Suppose that $\mathcal{A}=\dRham{X,\gf}$ and denote the de Rham differential by $d_{\mathrm{DR}}$. If $Q^{\mathrm{GF}}$ is a family of gauge-fixing operators for a free BV-theory $\mathcal{E}$ then the operator
\begin{equation} \label{eqn_canonicalLaplacian}
H:=\bigl[Q\cotimes\mathds{1}+\mathds{1}\cotimes d_{\mathrm{DR}},Q^{\mathrm{GF}}\bigr]: \mathcal{E}\cotimes\mathcal{A}\to\mathcal{E}\cotimes\mathcal{A}
\end{equation}
completes the list of structures required for $\mathcal{E}$ to form a family of free theories over $\mathcal{A}$, in the sense of Definition 2.9 of \cite{NCRGF}. This fact will allow us to make extensive use of the results of \cite{NCRGF} throughout the paper.
\end{rem}

\subsubsection{Heat kernels and convolution operators}

For the sake of brevity in what follows, we will formulate everything henceforth and for the rest of the paper, for families of gauge-fixing operators; reminding the reader who may wish to temporarily avoid this level of abstraction that this may be achieved by taking the parameter space $X$ to be a point. We begin with the definition of the convolution operator.

\begin{defi}
Let $\mathcal{E}$ be a free BV-theory and $\mathcal{A}$ be as above. The convolution map
\begin{equation} \label{eqn_convolutionmap}
\star:\mathcal{E}\cotimes\mathcal{E}\cotimes\mathcal{A} \longrightarrow \Hom_{\gf}(\mathcal{E}\cotimes\mathcal{A},\mathcal{E}\cotimes\mathcal{A})
\end{equation}
is defined as follows. Given $K\in\mathcal{E}\cotimes\mathcal{E}\cotimes\mathcal{A}$ and $s\in\mathcal{E}\cotimes\mathcal{A}$ we define
\[ K\star s := (-1)^{|K|}\bigl(\mathds{1}\cotimes\innprod_{\mathcal{A}}\bigr)[K\otimes s]\in\mathcal{E}\cotimes\mathcal{A}, \]
where $\innprod_{\mathcal{A}}$ is defined by \eqref{eqn_innprodA}.
\end{defi}

The following facts are easily verified using the (skew) self-adjoint properties of the operators $Q$ and $Q^{\mathrm{GF}}$.

\begin{lemma} \label{lem_convolutionidentities}
Let $\mathcal{E}$ be a free BV-theory and $\mathcal{A}$ be as above:
\begin{enumerate}
\item
The convolution map \eqref{eqn_convolutionmap} is injective.
\item
The map \eqref{eqn_convolutionmap} is a map of chain complexes, that is;
\[ \bigl[(Q\cotimes\mathds{1}\cotimes\mathds{1} + \mathds{1}\cotimes Q\cotimes\mathds{1} + \mathds{1}\cotimes\mathds{1}\cotimes d_{\mathrm{DR}})K\bigr]\star = \bigl[Q\cotimes\mathds{1} + \mathds{1}\cotimes d_{\mathrm{DR}},K\star\bigr]. \]
\item
Let $Q^{\mathrm{GF}}$ be a family of gauge-fixing operators for $\mathcal{E}$ over $\mathcal{A}$, then
\[ \bigl[(\mathds{1}\cotimes\mathds{1}\cotimes\mu)\bigl((\mathds{1}\cotimes\tau\cotimes\mathds{1})(Q^{\mathrm{GF}}\cotimes\mathds{1}\cotimes\mathds{1}) - (\mathds{1}\cotimes Q^{\mathrm{GF}}\cotimes\mathds{1})\bigr)K\bigr]\star = [Q^{\mathrm{GF}},K\star]. \]
\end{enumerate}
\end{lemma}

Set $\mathcal{A}:=\dRham{X,\gf}$ and suppose that $Q^{\mathrm{GF}}$ is a family of gauge-fixing operators for a free BV-theory $\mathcal{E}$. Consider the operator $H$ defined by Equation \eqref{eqn_canonicalLaplacian}. The results from the appendix to Chapter 9 of \cite{BerGetVer} assert the existence of a heat kernel
\[ K \in \smooth{0,\infty}\underset{\mathbb{R}}{\cotimes}\mathcal{E}\cotimes\mathcal{E}\cotimes\mathcal{A} \]
for this operator.

This means the following. If we denote by $K_t\in\mathcal{E}\cotimes\mathcal{E}\cotimes\mathcal{A}$, the value of the heat kernel at a point $t>0$, then the heat kernel is characterized by the following two properties:
\begin{equation} \label{eqn_heatkernelproperty}
\frac{\mathrm{d}}{\mathrm{d}t}\left(K_t\star s\right) = -H\left(K_t\star s\right) \quad\text{and}\quad \lim_{t\to 0}\left(K_t\star s\right) = s.
\end{equation}

Note that the heat kernel $K$ must have total degree one, since the convolution operator $K_t \star$ must have degree zero. Since $H$ is self-adjoint with respect to the pairing $\innprod_{\mathcal{A}}$, the convolution operator $K_t \star$ will be self-adjoint too, from which it follows that the heat kernel $K_t$ will be symmetric in $\mathcal{E}$.

The following identities for the heat kernel $K$ are also easily verified.

\begin{lemma} \label{lem_heatkernelidentities}
Set $\mathcal{A}:=\dRham{X,\gf}$ and suppose that $Q^{\mathrm{GF}}:\mathcal{E}\to\mathcal{E}\cotimes\mathcal{A}$ is a family of gauge-fixing operators for a free BV-theory $\mathcal{E}$. Consider the heat kernel $K$ for the operator $H$ defined by \eqref{eqn_canonicalLaplacian}:
\begin{enumerate}
\item \label{itm_heatkernelidentities1}
The heat kernel $K$ is a cycle; that is for all $t>0$,
\[ (Q\cotimes\mathds{1}\cotimes\mathds{1} + \mathds{1}\cotimes Q\cotimes\mathds{1} + \mathds{1}\cotimes\mathds{1}\cotimes d_{\mathrm{DR}})K_t = 0. \]
\item \label{itm_heatkernelidentities2}
For all $t>0$,
\[ (\mathds{1}\cotimes\mathds{1}\cotimes\mu)\bigl((\mathds{1}\cotimes\tau\cotimes\mathds{1})(Q^{\mathrm{GF}}\cotimes\mathds{1}\cotimes\mathds{1}) - (\mathds{1}\cotimes Q^{\mathrm{GF}}\cotimes\mathds{1})\bigr)K_t = 0. \]
\end{enumerate}
\end{lemma}

\begin{proof}
In both cases the logic is similar.
\begin{enumerate}
\item
Since $Q^2$ and $d_{\mathrm{DR}}^2$ both vanish, it follows that
\[ [Q\cotimes\mathds{1} + \mathds{1}\cotimes d_{\mathrm{DR}},H] = 0. \]
It then follows from \eqref{eqn_heatkernelproperty} that the commutator
\[ [Q\cotimes\mathds{1} + \mathds{1}\cotimes d_{\mathrm{DR}},K_t \star] \]
satisfies the heat equation and converges to zero as $t\to 0$. It must therefore vanish, cf. Lemma 2.16 of \cite{BerGetVer}. The claimed identity now follows as a consequence of Lemma \ref{lem_convolutionidentities}.
\item
Again, since $(Q^{\mathrm{GF}})^2=0$, it follows that $[Q^{\mathrm{GF}},H]$ vanishes. As before, it follows that $[Q^{\mathrm{GF}},K_t \star]$ satisfies the heat equation and converges to zero, and hence must vanish. Now apply Lemma \ref{lem_convolutionidentities}.
\end{enumerate}
\end{proof}

\subsubsection{Propagators}

Given a free BV-theory $\mathcal{E}$, a \emph{propagator} is a symmetric degree zero tensor
\[ P\in\bigl[\mathcal{E}\cotimes\mathcal{E}\bigr]^{\sg{2}}. \]
More generally, a \emph{family of propagators} over $\mathcal{A}$ is a tensor
\[ P\in\mathcal{E}\cotimes\mathcal{E}\cotimes\mathcal{A} \]
of total degree zero that is symmetric in $\mathcal{E}$; that is, it satisfies $(\tau\cotimes\mathds{1})[P]=P$.

There is a canonical propagator---which we will now define and use throughout the rest of the paper---associated to any gauge-fixing operator for a free BV-theory. More generally, any family of gauge-fixing operators will define a family of propagators.

\begin{defi} \label{def_canonicalpropagator}
Set $\mathcal{A}:=\dRham{X,\gf}$ and let $Q^{\mathrm{GF}}:\mathcal{E}\to\mathcal{E}\cotimes\mathcal{A}$ be a family of gauge-fixing operators for a free BV-theory $\mathcal{E}$. Let $K$ be the heat kernel for the operator $H$ defined by \eqref{eqn_canonicalLaplacian}. The family of propagators associated to this family of gauge-fixing conditions is defined by
\begin{equation} \label{eqn_canonicalpropagator}
P(\varepsilon,L) := \bigl(\mathds{1}_{\mathcal{E}\cotimes\mathcal{E}}\cotimes\mu\bigr) \bigl(\mathds{1}_{\mathcal{E}}\cotimes\tau\cotimes\mathds{1}_{\mathcal{A}}\bigr) \left(\int_{t=\varepsilon}^L\mathrm{d}t\cotimes Q^{\mathrm{GF}}\cotimes\mathds{1}_{\mathcal{E}\cotimes\mathcal{A}}\right)[K]; \quad 0<\varepsilon,L<\infty.
\end{equation}
By definition, $P(L,\varepsilon)=-P(\varepsilon,L)$.
\end{defi}

The fact that Equation \eqref{eqn_canonicalpropagator} defines a propagator---that is, the propagator obeys the relevant symmetry condition---follows from the fact that the commutator $[Q^{\mathrm{GF}},H]$ vanishes and the self-adjoint property of the gauge-fixing operator, see Example 2.11 of \cite{NCRGF}.

This propagator has the important property that it interpolates for the heat flow.

\begin{lemma} \label{lem_kernelhomotopy}
Set $\mathcal{A}:=\dRham{X,\gf}$ and let $P(\varepsilon,L)\in\mathcal{E}\cotimes\mathcal{E}\cotimes\mathcal{A}$ be the family of propagators defined by Equation \eqref{eqn_canonicalpropagator} that is assigned to a family of gauge-fixing operators $Q^{\mathrm{GF}}$ for a free BV-theory $\mathcal{E}$. Then $P(\varepsilon,L)$ is a chain homotopy for the heat flow; that is, for all~$\varepsilon,L>0$,
\[ (Q\cotimes\mathds{1}\cotimes\mathds{1} + \mathds{1}\cotimes Q\cotimes\mathds{1} + \mathds{1}\cotimes\mathds{1}\cotimes d_{\mathrm{DR}})P(\varepsilon,L) = K_{\varepsilon} - K_L. \]
\end{lemma}

\begin{proof}
Note that it follows from the heat equation \eqref{eqn_heatkernelproperty} that
\[ \frac{\mathrm{d}}{\mathrm{d}t} K_t = -(\mathds{1}_{\mathcal{E}\cotimes\mathcal{E}}\cotimes\mu) (\mathds{1}\cotimes\tau\cotimes\mathds{1}) (H\cotimes\mathds{1}_{\mathcal{E}\cotimes\mathcal{A}})K_t. \]
Now we calculate
\begin{displaymath}
\begin{split}
(Q\cotimes\mathds{1}\cotimes\mathds{1} &+ \mathds{1}\cotimes Q\cotimes\mathds{1} + \mathds{1}\cotimes\mathds{1}\cotimes d_{\mathrm{DR}})P(\varepsilon,L) \\
=& \int_{t=\varepsilon}^L (\mathds{1}_{\mathcal{E}\cotimes\mathcal{E}}\cotimes\mu) (\mathds{1}\cotimes\tau\cotimes\mathds{1})\bigl( (Q\cotimes\mathds{1}+\mathds{1}\cotimes d_{\mathrm{DR}})Q^{\mathrm{GF}}\cotimes\mathds{1}_{\mathcal{E}\cotimes\mathcal{A}}\bigr)K_t\,\mathrm{d}t \\
&- \int_{t=\varepsilon}^L (\mathds{1}_{\mathcal{E}\cotimes\mathcal{E}}\cotimes\mu) (\mathds{1}\cotimes\tau\cotimes\mathds{1}) (Q^{\mathrm{GF}}\cotimes\mathds{1}_{\mathcal{E}\cotimes\mathcal{A}}) (\mathds{1}\cotimes Q\cotimes\mathds{1} + \mathds{1}\cotimes\mathds{1}\cotimes d_{\mathrm{DR}})K_t\,\mathrm{d}t \\
=& \int_{t=\varepsilon}^L (\mathds{1}_{\mathcal{E}\cotimes\mathcal{E}}\cotimes\mu) (\mathds{1}\cotimes\tau\cotimes\mathds{1})\bigl( \bigl((Q\cotimes\mathds{1}+\mathds{1}\cotimes d_{\mathrm{DR}})Q^{\mathrm{GF}}+Q^{\mathrm{GF}}Q\bigr)\cotimes\mathds{1}_{\mathcal{E}\cotimes\mathcal{A}}\bigr)K_t\,\mathrm{d}t \\
=& \int_{t=\varepsilon}^L (\mathds{1}_{\mathcal{E}\cotimes\mathcal{E}}\cotimes\mu) (\mathds{1}\cotimes\tau\cotimes\mathds{1}) (H\cotimes\mathds{1}_{\mathcal{E}\cotimes\mathcal{A}})K_t\,\mathrm{d}t = -\int_{t=\varepsilon}^L \frac{\mathrm{d}}{\mathrm{d}t} K_t \,\mathrm{d}t = K_{\varepsilon} - K_L,
\end{split}
\end{displaymath}
where on the first line we use the Leibniz identity for $d_{\mathrm{DR}}$, on the second line we use Lemma~\ref{lem_heatkernelidentities}\eqref{itm_heatkernelidentities1} and on the last line we use the heat equation.
\end{proof}

\subsection{Local distributions and operators}

In what follows we will need the notion of locality, which we introduce here following \cite{CosEffThy}.

\subsubsection{Definition of locality}

If $\mathcal{E}$ is the space of sections of a vector bundle $E$, then $\mathcal{E}^{\cotimes l}$ will be the space of sections of its external tensor product $E^{\boxtimes l}$; hence, the dual space will be a space of distributions.

\begin{defi} \label{def_localdistribution}
Let $\mathcal{E}:=\Gamma(M,E)$ be the space of sections of a vector bundle $E$ over a compact manifold $M$. We say that a distribution in $(\mathcal{E}^{\cotimes l})^{\dag}$ is a \emph{local distribution} if it can be written as a finite sum of distributions of the form
\begin{equation} \label{eqn_localdistribution}
s_1,\ldots,s_l \longmapsto \int_M D_1 s_1\cdots D_l s_l\,\mathrm{d}\varrho,
\end{equation}
where $D_1,\ldots D_l:\mathcal{E}\to\smooth{M,\gf}$ are differential operators and $\mathrm{d}\varrho$ is a density.

More generally, if $\mathcal{A}$ is a commutative topological algebra of the form described in Section \ref{sec_gaugefixingfamily} then we will say that an operator from $\mathcal{E}^{\cotimes l}$ to $\mathcal{A}$ is a \emph{local $\mathcal{A}$-valued distribution} if it can be written as a finite sum of operators of the form \eqref{eqn_localdistribution}, where now each $D_i$ is a map from $\mathcal{E}$ to $\smooth{M,\gf}\cotimes\mathcal{A}$ whose $\mathcal{A}$-linear extension to $\mathcal{E}\cotimes\mathcal{A}$ is a differential operator.
\end{defi}

\begin{example}
The kinetic term \eqref{eqn_kineticterm} for a free BV-theory provides an example of a local distribution.
\end{example}

\begin{defi}
Let $\mathcal{E}:=\Gamma(M,E)$ be the space of smooth sections of a vector bundle $E$ over a compact manifold $M$. We say that a map from $\mathcal{E}^{\cotimes l}$ to $\mathcal{E}$ is a \emph{local operator} if it can be written as a finite sum of operators of the form
\begin{equation} \label{eqn_localoperator}
s_1,\ldots,s_l \longmapsto (D_1 s_1\cdots D_l s_l)s,
\end{equation}
where $D_1,\ldots D_l:\mathcal{E}\to\smooth{M,\gf}$ are differential operators and $s\in\mathcal{E}$.

More generally, given a commutative topological algebra $\mathcal{A}$ of the form described in Section \ref{sec_gaugefixingfamily}, we say that an operator from $\mathcal{E}^{\cotimes l}$ to $\mathcal{E}\cotimes\mathcal{A}$ is a \emph{$\mathcal{A}$-valued local operator} if it can be written as a finite sum of operators of the form \eqref{eqn_localoperator}, where now each $D_i$ is a map from $\mathcal{E}$ to $\smooth{M,\gf}\cotimes\mathcal{A}$ whose $\mathcal{A}$-linear extension to $\mathcal{E}\cotimes\mathcal{A}$ is a differential operator and $s\in\mathcal{E}\cotimes\mathcal{A}$. We will denote the space of $\mathcal{A}$-valued local operators by $\Locdelim{(\mathcal{E}^{\cotimes l},\mathcal{E}\cotimes\mathcal{A})}$.
\end{defi}

\begin{rem} \label{rem_localdistribution}
Using integration by parts and the nondegeneracy of the local pairing $\innprodloc$, we may write any local $\mathcal{A}$-valued distribution \eqref{eqn_localdistribution} as
\[ s_1,\ldots,s_l \longmapsto {\innprod[X(s_1,\ldots,s_{l-1}),s_l]}_{\mathcal{A}} \]
for some $\mathcal{A}$-valued local operator $X$. Conversely, any operator of the above form will be a local $\mathcal{A}$-valued distribution.
\end{rem}

\subsubsection{Local vector fields} \label{sec_LocVecfields}

Let $\mathcal{V}$ be a nuclear Fr\'echet space and consider the completed tensor algebra $\ctalgdelim{(\mathcal{V}^{\dag})}$ defined by \eqref{eqn_tensalg}. Denote the Lie algebra consisting of all those (continuous) derivations on $\ctalgdelim{(\mathcal{V}^{\dag})}$ by
\[ \Derdelim{\bigl(\ctalgdelim{(\mathcal{V}^{\dag})}\bigr)}. \]
We will call any such a derivation a \emph{vector field} on $\mathcal{V}$. Since any vector field is completely determined by its values on the generators $\mathcal{V}^{\dag}$, we may identify the underlying vector space of this Lie algebra as
\[ \Derdelim{\bigl(\ctalgdelim{(\mathcal{V}^{\dag})}\bigr)} = \Hom_{\gf}\bigl(\mathcal{V}^{\dag},\ctalgdelim{(\mathcal{V}^{\dag})}\bigr) = \prod_{n=0}^{\infty} \Hom_{\gf}\bigl(\mathcal{V}^{\dag},(\mathcal{V}^{\dag})^{\cotimes n}\bigr)
= \prod_{n=0}^{\infty} \Hom_{\gf}\bigl(\mathcal{V}^{\dag},(\mathcal{V}^{\cotimes n})^{\dag}\bigr), \]
where at the last step we have used Proposition \ref{prop_tensordual}. Now, if $\mathcal{E}$ is the space of smooth sections of a vector bundle $E$ over a compact manifold $M$, then we may identify the last vector space above as
\begin{equation} \label{eqn_DerequalsHom}
\Derdelim{\bigl(\ctalgdelim{(\mathcal{E}^{\dag})}\bigr)} = \prod_{n=0}^{\infty} \Hom_{\gf}\bigl(\mathcal{E}^{\cotimes n},\mathcal{E}\bigr);
\end{equation}
where we have used the fact that both $\mathcal{E}$ and $\mathcal{E}^{\cotimes n}$ are reflexive. More generally, if $\mathcal{A}$ is a commutative topological algebra of the form described in Section \ref{sec_gaugefixingfamily}, then we may consider the space of \emph{$\mathcal{A}$-valued vector fields}
\begin{equation} \label{eqn_DerAequalsHomA}
\Derdelim^{\mathcal{A}}{\bigl(\ctalgdelim{(\mathcal{E}^{\dag})}\bigr)} := \prod_{n=0}^{\infty} \Hom_{\gf}\bigl(\mathcal{E}^{\cotimes n},\mathcal{E}\bigr)\cotimes\mathcal{A} = \prod_{n=0}^{\infty} \Hom_{\gf}\bigl(\mathcal{E}^{\cotimes n},\mathcal{E}\cotimes\mathcal{A}\bigr).
\end{equation}
This will again be a Lie algebra, as we will shortly see. If $X$ is an $\mathcal{A}$-valued vector field on $\mathcal{E}$ then we will denote the corresponding linear operators by
\[ X_n:\mathcal{E}^{\cotimes n}\to\mathcal{E}\cotimes\mathcal{A}, \quad n\geq 0. \]

It is well-known that the commutator bracket on the left-hand side of \eqref{eqn_DerequalsHom} arises from the Gerstenhaber pre-Lie bracket on the right-hand side of the same, see \cite{GerPreLie, StaBracket}. More generally, the Gerstenhaber bracket may be extended to $\mathcal{A}$-valued vector fields, providing them with the structure of a Lie algebra. Given
\[ X_n:\mathcal{E}^{\cotimes n}\to\mathcal{E}\cotimes\mathcal{A} \quad\text{and}\quad Y_m:\mathcal{E}^{\cotimes m}\to\mathcal{E}\cotimes\mathcal{A}, \]
we define $X_n\circ_i Y_m$ for $i=1,\ldots, m$ as the composite
\begin{equation} \label{eqn_Gerstbracket}
\xymatrix{
\mathcal{E}^{\cotimes n+m-1} \ar^<<<<<<<<<<<<<<{\mathds{1}_{\mathcal{E}^{\cotimes i-1}}\cotimes X_n\cotimes\mathds{1}_{\mathcal{E}^{\cotimes m-i}}}[rrr] &&&  \mathcal{E}^{\cotimes i}\cotimes\mathcal{A}\cotimes\mathcal{E}^{\cotimes m-i} \ar@{=}[r] & \mathcal{E}^{\cotimes m}\cotimes\mathcal{A} \ar^{Y_m\cotimes\mathds{1}}[r] & \mathcal{E}\cotimes\mathcal{A}^{\cotimes 2} \ar^{\mathds{1}\cotimes\mu}[r] & \mathcal{E}\cotimes\mathcal{A}.
}
\end{equation}
The Lie and pre-Lie bracket are then defined by:
\begin{displaymath}
\begin{split}
X_n\circ Y_m &:= \sum_{i=1}^m X_n\circ_i Y_m, \\
[X_n,Y_m] &:= X_n\circ Y_m - (-1)^{|X_n||Y_m|}Y_m\circ X_n.
\end{split}
\end{displaymath}

We may consider the Lie subalgebra of $\Derdelim^{\mathcal{A}}{\bigl(\ctalgdelim{(\mathcal{E}^{\dag})}\bigr)}$ that arises from taking just the local operators on the right-hand side of \eqref{eqn_DerAequalsHomA}, which we denote by
\[ \DerLdelim^{\mathcal{A}}{\bigl(\ctalgdelim{(\mathcal{E}^{\dag})}\bigr)} = \prod_{n=0}^{\infty} \Locdelim{\bigl(\mathcal{E}^{\cotimes n},\mathcal{E}\cotimes\mathcal{A}\bigr)}. \]
We will refer to these as \emph{local ($\mathcal{A}$-valued) vector fields}. That they form a Lie subalgebra follows tautologically from the definition of a local operator and the Gerstenhaber bracket.

The Lie algebra $\Derdelim{\bigl(\ctalgdelim{(\mathcal{E}^{\dag})}\bigr)}$ acts canonically on $\ctalgdelim{(\mathcal{E}^{\dag})}$. More generally, the Lie algebra $\Derdelim^{\mathcal{A}}{\bigl(\ctalgdelim{(\mathcal{E}^{\dag})}\bigr)}$ acts on the space of $\mathcal{A}$-valued distributions
\[ \ctalgdelim{(\mathcal{E}^{\dag})}\cotimes\mathcal{A} = \prod_{n=0}^{\infty}\Hom_{\gf}\bigl(\mathcal{E}^{\cotimes n},\mathcal{A}\bigr). \]
The latter action is defined using the same formula \eqref{eqn_Gerstbracket} that was used to define the Gerstenhaber pre-Lie bracket, except that now $Y_m$ is an $\mathcal{A}$-valued distribution. The Lie subalgebra of \emph{local} $\mathcal{A}$-valued vector fields $\DerLdelim^{\mathcal{A}}{\bigl(\ctalgdelim{(\mathcal{E}^{\dag})}\bigr)}$ then acts on the subspace formed by the \emph{local} $\mathcal{A}$-valued distributions.

A story similar to that described above also holds in the commutative context, for the completed symmetric algebra $\csalg{\mathcal{E}^{\dag}}$; but for the sake of brevity, we will not elaborate further.

\section{Classical BV-geometry} \label{sec_ClassicBV}

In this section we will explain how to adapt the ideas and methods from \cite[\S 5.3]{CosEffThy} and \cite{SchwarzBV}---which describe the basic framework for the Batalin-Vilkovisky formalism in the context of commutative geometry---to the framework of noncommutative geometry which was introduced by Kontsevich in \cite{KontSympGeom}. We will work in the general framework of $\mathcal{A}$-valued distributions and vector fields, where $\mathcal{A}$ is a commutative topological algebra of the form described in Section \ref{sec_gaugefixingfamily}.

\subsection{Symplectic structures}

Our starting point in this section will be a free BV-theory, which includes as part of the data a skew-symmetric pairing on the space of fields that plays the role of a symplectic structure. We will explain how this symplectic structure leads to a Lie bracket on the space of Hamiltonian functionals.

\subsubsection{Noncommutative geometry}

Given a free BV-theory $\mathcal{E}$, we wish to describe the space in which our functionals will live, and on which we will build our BV-structures. In the framework of commutative geometry, the correct space is the space of polynomial functions on the fields; that is, the completed symmetric algebra $\csalgdelim{(\mathcal{E}^{\dag})}$ defined by \eqref{eqn_symmalg}. In the context of noncommutative geometry, the correct space was introduced by Kontsevich in \cite{KontSympGeom, KontFeyn}.

\begin{defi} \label{def_KontHam}
Given a free BV-theory $\mathcal{E}$, define
\[ \KontHam{\mathcal{E}}:=\prod_{n=0}^{\infty}\bigl[(\mathcal{E}^{\dag})^{\cotimes n}\bigr]_{\cyc{n}} \quad\text{and}\quad \KontHamP{\mathcal{E}}:=\prod_{n=1}^{\infty}\bigl[(\mathcal{E}^{\dag})^{\cotimes n}\bigr]_{\cyc{n}}, \]
where in the above we have taken the \emph{cyclic} coinvariants. More generally, given $\mathcal{A}$ as above, we define
\[ \KontHam{\mathcal{E},\mathcal{A}} := \KontHam{\mathcal{E}}\cotimes\mathcal{A} \quad\text{and}\quad \KontHamP{\mathcal{E},\mathcal{A}} := \KontHamP{\mathcal{E}}\cotimes\mathcal{A}. \]
\end{defi}

\subsubsection{Symplectic vector fields and Hamiltonian functionals} \label{sec_SympandHam}

Recall from Section \ref{sec_LocVecfields} that vector fields were defined to be derivations of the completed tensor algebra. We now define, following \cite[\S 5.3.2]{CosEffThy}, what it means for such a vector field to be symplectic.

\begin{defi}
Let $\mathcal{E}$ be a free BV-theory. We say that an $\mathcal{A}$-valued vector field $X$ on $\mathcal{E}$ is a \emph{symplectic vector field} if for all $n\geq 0$, the map
\begin{equation} \label{eqn_cyclicmap}
\mathcal{E}^{\cotimes n+1}\to\mathcal{A}, \qquad s_1,\ldots,s_n,s_0\mapsto{\innprod[X_n(s_1,\ldots,s_n),s_0]}_{\mathcal{A}}
\end{equation}
is cyclically symmetric, where $X_n:\mathcal{E}^{\cotimes n}\to\mathcal{E}\cotimes\mathcal{A}$ denotes the operator associated to $X$ by \eqref{eqn_DerAequalsHomA} and $\innprod_{\mathcal{A}}$ was defined by \eqref{eqn_innprodA}.
\end{defi}

\begin{rem}
This terminology derives from the observation in Theorem 10.6 of \cite{HtopNC} that this cyclic symmetry condition (when $\mathcal{A}=\gf$) is equivalent to the condition that this vector field annihilates a certain symplectic form in the framework of noncommutative geometry that was introduced and described in \cite{GinzNC, KontSympGeom}. It seems reasonable to assume that this broad framework of noncommutative geometry could ultimately be transported to our current setting, which uses topological vector spaces, and that in such a framework symplectic vector fields would have an equivalent characterization. However, to do so would presently take us too far afield.
\end{rem}

The symplectic vector fields form a Lie subalgebra of the Lie algebra of all vector fields, as we will presently see. Given a locally convex Hausdorff topological vector space $\mathcal{V}$, consider the cyclic symmetrization map
\begin{equation} \label{eqn_cyclicsum}
\cycsum{n}:\bigl[\mathcal{V}^{\cotimes n}\bigr]_{\cyc{n}}\longrightarrow\bigl[\mathcal{V}^{\cotimes n}\bigr]^{\cyc{n}}, \qquad x\mapsto\sum_{\varsigma\in\cyc{n}}\varsigma\cdot x.
\end{equation}

\begin{lemma} \label{lem_sympvfieldbracket}
Let $\mathcal{E}$ be a free BV-theory and let
\[ X_n:\mathcal{E}^{\cotimes n}\to\mathcal{E}\cotimes\mathcal{A} \quad\text{and}\quad Y_m:\mathcal{E}^{\cotimes m}\to\mathcal{E}\cotimes\mathcal{A} \]
be two symplectic $\mathcal{A}$-valued vector fields, where $n+m>0$.

For all $s_1,\ldots,s_{n+m}\in\mathcal{E}$ we have,
\[ {\innprod[{[X_n,Y_m](s_1,\ldots,s_{n+m-1}),s_{n+m}}]}_{\mathcal{A}} =  {\innprod}_{\mathcal{A}} (Y_m\cotimes X_n)\cycsum{n+m}(s_1,\ldots,s_{n+m}). \]
In particular, $[X_n,Y_m]$ is a symplectic vector field.
\end{lemma}

\begin{proof}
This is a straightforward calculation using the definition of the Gerstenhaber bracket, the cyclic symmetry of the map \eqref{eqn_cyclicmap} and the skew symmetry of the pairing $\innprod$.
\end{proof}

We can now introduce, following \cite[\S 5.3.2]{CosEffThy}, the definition of a Hamiltonian functional. Given a free BV-theory $\mathcal{E}$, consider the isomorphism
\begin{equation} \label{eqn_HamtoCycinvariants}
\KontHam{\mathcal{E},\mathcal{A}} \longrightarrow \prod_{l=0}^{\infty}\Hom_{\gf}\bigl(\mathcal{E}^{\cotimes l},\mathcal{A}\bigr)^{\cyc{l}}
\end{equation}
which takes
\[ h_l\in\bigl[(\mathcal{E}^{\dag})^{\cotimes l}\bigr]_{\cyc{l}}\cotimes\mathcal{A} = \Hom_{\gf}\bigl(\mathcal{E}^{\cotimes l},\mathcal{A}\bigr)_{\cyc{l}} \]
to $h_l\cycsum{l}\in\Hom_{\gf}(\mathcal{E}^{\cotimes l},\mathcal{A})^{\cyc{l}}$.

\begin{defi} \label{def_Hamfunctional}
Let $\mathcal{E}$ be a free BV-theory and $h\in\KontHam{\mathcal{E},\mathcal{A}}$. Write $h$ as
\[ h=\sum_{l=0}^{\infty}h_l, \quad h_l\in\bigl[(\mathcal{E}^{\dag})^{\cotimes l}\bigr]_{\cyc{l}}\cotimes\mathcal{A}. \]
We will say that $h$ is an \emph{$\mathcal{A}$-valued Hamiltonian functional} if there is an $\mathcal{A}$-valued vector field $X$ on $\mathcal{E}$ such that for all $l\geq 1$,
\begin{equation} \label{eqn_HamiltonianFunctional}
h_l\cycsum{l}(s_1,\ldots,s_l) = {\innprod[X_{l-1}(s_1,\ldots,s_{l-1}),s_l]}_{\mathcal{A}}, \quad\text{for all }s_1,\ldots,s_l\in\mathcal{E}.
\end{equation}
\end{defi}

Note that this condition places no restriction on the constant term $h_0\in\mathcal{A}$. Such a vector field $X$, if it exists, must be symplectic. Due to the nondegeneracy of the pairing $\innprod$ on the fields $\mathcal{E}$, there can be only one such vector field associated to a Hamiltonian functional $h$; denote this vector field by $X^h$.

We may now identify Hamiltonian functionals with symplectic vector fields, cf. \cite[\S 5.3.2]{CosEffThy} and \cite[\S 2]{GinzNC}.

\begin{prop} \label{prop_HamSympcorrespondence}
Let $\mathcal{E}$ be a free BV-theory. The map
\[ h\mapsto X^h \]
defines a one-to-one correspondence of degree one between the subspace of $\KontHamP{\mathcal{E},\mathcal{A}}$ consisting of all Hamiltonian functionals and the Lie subalgebra of $\Derdelim^{\mathcal{A}}{\bigl(\ctalgdelim{(\mathcal{E}^{\dag})}\bigr)}$ consisting of all symplectic vector fields.
\end{prop}

\begin{proof}
This follows directly from the fact that the map \eqref{eqn_HamtoCycinvariants} is an isomorphism.
\end{proof}

Recall from Section \ref{sec_LocVecfields} that the Lie algebra $\Derdelim^{\mathcal{A}}{\bigl(\ctalgdelim{(\mathcal{E}^{\dag})}\bigr)}$ of $\mathcal{A}$-valued vector fields acts on the space of $\mathcal{A}$-valued distributions $\ctalgdelim{(\mathcal{E}^{\dag})}\cotimes\mathcal{A}$. This action descends to the quotient space $\KontHam{\mathcal{E},\mathcal{A}}$. In fact, we may arrive at the following formula for this action:
\begin{equation} \label{eqn_vfieldaction}
X_n(h_l) = \mu\bigl(h_l\cycsum{l}\cotimes\mathds{1}\bigr)\bigl(\mathds{1}_{\mathcal{E}^{\cotimes l-1}}\cotimes X_n\bigr)\in\Hom_{\gf}\bigl(\mathcal{E}^{\cotimes n+l-1},\mathcal{A}\bigr)_{\cyc{(n+l-1)}},
\end{equation}
for all $X_n\in\Hom_{\gf}\bigl(\mathcal{E}^{\cotimes n},\mathcal{E}\cotimes\mathcal{A}\bigr)$, $h_l\in\Hom_{\gf}\bigl(\mathcal{E}^{\cotimes l},\mathcal{A}\bigr)_{\cyc{l}}$ and $l\geq 1$.

Using this action we may define a bracket on the Hamiltonian functionals in $\KontHam{\mathcal{E},\mathcal{A}}$, cf. \cite{GinzNC, KontSympGeom}.

\begin{defi} \label{def_Hambracket}
Let $\mathcal{E}$ be a free BV-theory. Given two Hamiltonian functionals $f,h\in\KontHam{\mathcal{E},\mathcal{A}}$ define
\begin{equation} \label{eqn_Hambracket}
\{f,h\} := (-1)^{|f||h|+|h|}X^f(h)\in\KontHam{\mathcal{E},\mathcal{A}}.
\end{equation}
Note that \eqref{eqn_Hambracket} is still well-defined even if $h$ is not a Hamiltonian functional, providing that $f$ still is. In this way, the definition of the bracket may be extended to the case when only one of the two functionals is Hamiltonian.
\end{defi}

This bracket has degree one. We will see presently that it defines an (odd) Lie bracket on the space of Hamiltonian functionals which corresponds to the commutator bracket on symplectic vector fields under the correspondence defined by Proposition \ref{prop_HamSympcorrespondence}.

\begin{prop} \label{prop_oddbracket}
Let $\mathcal{E}$ be a free BV-theory and $f,h\in\KontHam{\mathcal{E},\mathcal{A}}$ be any two Hamiltonian functionals. Then $\{f,h\}$ is also a Hamiltonian functional and
\begin{equation} \label{eqn_HamtoSympbracket}
[X^f,X^h] = (-1)^{|f||h|+|h|}X^{\{f,h\}}.
\end{equation}

Consequently, $\{-,-\}$ defines a Lie bracket of degree one on the subspace of $\KontHam{\mathcal{E},\mathcal{A}}$ consisting of all Hamiltonian functionals. Furthermore, the Lie subalgebra of $\Derdelim^{\mathcal{A}}{\bigl(\ctalgdelim{(\mathcal{E}^{\dag})}\bigr)}$ consisting of all symplectic vector fields acts on this subspace. More generally, the Lie algebra of Hamiltonian functionals acts on $\KontHam{\mathcal{E},\mathcal{A}}$ according to Equation \eqref{eqn_Hambracket}.
\end{prop}

\begin{rem}
Note that since $\{-,-\}$ is an \emph{odd} Lie bracket, it will be graded \emph{symmetric}.
\end{rem}

\begin{proof}[Proof of Proposition \ref{prop_oddbracket}]
Note that if
\[ f\in\Hom_{\gf}\bigl(\mathcal{E}^{\cotimes n},\mathcal{A}\bigr)_{\cyc{n}} \quad\text{and}\quad h\in\Hom_{\gf}\bigl(\mathcal{E}^{\cotimes m},\mathcal{A}\bigr)_{\cyc{m}} \]
are Hamiltonian functionals then it follows from Equation \eqref{eqn_vfieldaction} and \eqref{eqn_HamiltonianFunctional} that
\begin{equation} \label{eqn_Poisbrack}
(-1)^{|f||h|+|h|}\{f,h\} = X^f(h) = \innprod_{\mathcal{A}}\bigl(X^h \cotimes X^f\bigr).
\end{equation}
Equation \eqref{eqn_HamtoSympbracket} now follows directly from Lemma \ref{lem_sympvfieldbracket}. The symmetry and Jacobi identities for the odd bracket $\{-,-\}$ then follow from this equation. That symplectic vector fields preserve the subspace of Hamiltonian functionals is now just a consequence of Proposition \ref{prop_HamSympcorrespondence}.
\end{proof}

\subsubsection{Local functionals}

One important property of local functionals which we will now establish is that they are always Hamiltonian. Recall that $\KontHam{\mathcal{E},\mathcal{A}}$ is a quotient of the space $\ctalgdelim{(\mathcal{E}^{\dag})}\cotimes\mathcal{A}$ of $\mathcal{A}$-valued distributions.

\begin{defi}
Let $\mathcal{E}$ be a free BV-theory. We will say that a functional in $\KontHam{\mathcal{E},\mathcal{A}}$ is \emph{local} if it can be represented by a local $\mathcal{A}$-valued distribution from $\ctalgdelim{(\mathcal{E}^{\dag})}\cotimes\mathcal{A}$. We will denote the subspace of $\KontHam{\mathcal{E},\mathcal{A}}$ consisting of all local functionals by $\KontHamL{\mathcal{E},\mathcal{A}}$. Likewise, $\KontHamPL{\mathcal{E},\mathcal{A}}$ will denote the corresponding subspace of $\KontHamP{\mathcal{E},\mathcal{A}}$.
\end{defi}

\begin{prop} \label{prop_localisHam}
Let $\mathcal{E}$ be a free BV-theory. Every local functional $h\in\KontHamL{\mathcal{E},\mathcal{A}}$ is Hamiltonian and $\KontHamL{\mathcal{E},\mathcal{A}}$ forms a Lie subalgebra of the Lie algebra of Hamiltonian functionals under the bracket $\{-,-\}$ defined by \eqref{eqn_Hambracket}.

Furthermore, the map $h\mapsto X^h$ of Proposition \ref{prop_HamSympcorrespondence} determines a one-to-one correspondence between $\KontHamPL{\mathcal{E},\mathcal{A}}$ and the Lie subalgebra of $\DerLdelim^{\mathcal{A}}{\bigl(\ctalgdelim{(\mathcal{E}^{\dag})}\bigr)}$ that consists of all local symplectic vector fields.
\end{prop}

\begin{proof}
It follows from Remark \ref{rem_localdistribution} that:
\begin{itemize}
\item
every local functional $h\in\KontHamL{\mathcal{E},\mathcal{A}}$ is Hamiltonian,
\item
that the symplectic vector field $X^h$ that is associated to this local functional is a local vector field, and
\item
that if $X^h$ is a \emph{local} symplectic vector field defined by a Hamiltonian functional $h$, then the functional $h$ itself must be local.
\end{itemize}

From this and Equation \eqref{eqn_Hambracket} it follows, since local vector fields act on local functionals, that the bracket of two local functionals is another local functional.
\end{proof}

\begin{example} \label{exm_kinetic}
Let $\mathcal{E}$ be a free BV-theory and consider the kinetic term for our gauge theory, which is
\begin{equation} \label{eqn_kinetic}
\kappa := \frac{1}{2}\innprod (Q\cotimes\mathds{1}):\mathcal{E}\cotimes\mathcal{E}\to\gf.
\end{equation}
This has degree zero and is symmetric in $\mathcal{E}$. Since this is a local distribution, it defines a local functional which we will also denote by $\kappa\in\KontHamL{\mathcal{E}}$. The symplectic vector field corresponding to $\kappa$ is just the operator
\[ X^{\kappa} = Q:\mathcal{E}\to\mathcal{E}. \]
Hence, for all $h\in\KontHam{\mathcal{E}}$;
\[ \{\kappa,h\} = (-1)^{|h|}Q(h). \]
By Proposition \ref{prop_HamSympcorrespondence} and Equation \eqref{eqn_HamtoSympbracket}, since $Q^2=0$, it follows that $\{\kappa,\kappa\}=0$.
\end{example}

\subsection{Classical master equation}

We now describe the classical master equation in noncommutative geometry. This is the starting point for the quantization procedure in the Batalin-Vilkovisky formalism. In the framework of commutative geometry, the classical master equation arises from the (infinitesimal) action of the gauge group on the space of fields. The equation imposes a symmetry constraint on the action functional, namely that it is invariant under the action of the gauge group. We will explain how these ideas are modified in the context of noncommutative geometry and the corresponding meaning of the classical master equation in this context.

\subsubsection{The classical master equation in noncommutative geometry}

\begin{defi}
Let $\mathcal{E}$ be a free BV-theory and $S\in\KontHam{\mathcal{E},\mathcal{A}}$ be a Hamiltonian functional of degree zero. We say that $S$ satisfies the \emph{classical master equation} if
\[ \{S,S\} = 0. \]
\end{defi}

\begin{example} \label{exm_CMEinteraction}
If $\kappa\in\KontHamL{\mathcal{E}}\subset\KontHamL{\mathcal{E},\mathcal{A}}$ is the kinetic term for our free BV-theory given by \eqref{eqn_kinetic} and $I\in\KontHamL{\mathcal{E},\mathcal{A}}$ is a Hamiltonian functional of degree zero, then from Example \ref{exm_kinetic} we see that the action functional $S:=\kappa+I$ satisfies the classical master equation if and only if
\begin{equation} \label{eqn_CMEinteraction}
QI+\frac{1}{2}\{I,I\}=0.
\end{equation}
We will call \eqref{eqn_CMEinteraction} the \emph{classical master equation for the interaction $I$}; or, by an abuse of terminology, simply the classical master equation.
\end{example}

\subsubsection{Local algebra structures}

In the framework of noncommutative geometry, the action of the gauge group is replaced by a bimodule structure.

\begin{defi}
Let $U$ be a vector bundle over a compact manifold $M$ and $\mathcal{U}:=\Gamma(M,U)$ be its space of smooth sections. A \emph{local algebra} structure on $\mathcal{U}$ is a local operator
\[ \mathcal{U}\cotimes\mathcal{U}\to\mathcal{U} \]
satisfying the associativity constraint.

If $V$ is another vector bundle over $M$ and $\mathcal{V}$ denotes its space of smooth sections, then a \emph{local $\mathcal{U}$-bimodule} structure on $\mathcal{V}$ consists of a pair of local operators
\[ \mathcal{U}\cotimes\mathcal{V}\to\mathcal{V} \quad\text{and}\quad \mathcal{V}\cotimes\mathcal{U}\to\mathcal{V} \]
satisfying the usual associativity constraints.

The data of a \emph{twisted local $\mathcal{U}$-bimodule} structure on $\mathcal{V}$ consists of a local $\mathcal{U}$-bimodule structure, together with a local operator
\[ d:\mathcal{U}\to\mathcal{V} \]
of degree zero, which is a derivation with respect to the structures above. We will call any such operator $d$ a \emph{twisting operator}.
\end{defi}

Note that a local bimodule is the same thing as a twisted local bimodule whose twisting operator vanishes, cf. \cite[\S 1.7]{NijRicDeform}. If our spaces $\mathcal{U}$ and $\mathcal{V}$ happen to be graded, then the maps above must be homogeneous of degree zero.

\begin{example} \label{exm_AUMmodule}
Let $A$ be a finite-dimensional $\gf$-algebra and take $U$ to be the trivial $A$-bundle over a compact manifold $M$, so that
\[ \mathcal{U} = \smooth{M}\underset{\mathbb{R}}{\otimes} A. \]
This has the canonical structure of a local algebra.

Consider now the vector bundle $V:=T^*M\otimes_{\mathbb{R}} A$, whose space of sections
\[ \mathcal{V} = \Gamma(M,T^*M)\underset{\mathbb{R}}{\otimes} A \]
has a canonical local $\mathcal{U}$-bimodule structure. The exterior derivative
\[ d_{\mathrm{DR}}\otimes\mathds{1}:\smooth{M}\underset{\mathbb{R}}{\otimes} A\longrightarrow\Gamma(M,T^*M)\underset{\mathbb{R}}{\otimes} A \]
provides $\mathcal{V}$ with the structure of a twisted local $\mathcal{U}$-bimodule.
\end{example}

\subsubsection{Fields and ghosts}

Given a local algebra $\mathcal{U}$ and a twisted local $\mathcal{U}$-bimodule $\mathcal{V}$, we form the graded topological vector space
\[ \mathcal{X}:=\mathcal{U}\oplus\Sigma^{-1}\mathcal{V}. \]
The algebra, bimodule and twisting local operators give this space the structure of a differential graded algebra. Explicitly, the multiplication $m$ is defined according to the Koszul sign rule by
\[ m(u_1,u_2)=u_1u_2, \quad m(u,v)=(-1)^{|u|}uv, \quad m(v,u)=vu, \quad m(v_1,v_2)=0; \]
for all $u_1,u_2,u\in\mathcal{U}$ and $v_1,v_2,v\in\mathcal{V}$. The differential $d$ is defined by the twisting operator, which vanishes on $\Sigma^{-1}\mathcal{V}$.

Next we take the suspension of this graded topological vector space, which is
\[ \Sigma\mathcal{X} = \Sigma\mathcal{U}\oplus\mathcal{V}. \]
We call the component $\Sigma\mathcal{U}$ the \emph{ghosts} and the component $\mathcal{V}$ the \emph{fields}. The multiplication $m$ and differential $d$ combine to form the components of a local vector field $X^{\mathrm{gauge}}$ of degree one on $\Sigma\mathcal{X}$. Explicitly, they are related by the equations:
\begin{equation} \label{eqn_Xgauge}
\begin{split}
X^{\mathrm{gauge}}_2 &:= \Sigma\circ m\circ(\Sigma^{-1})^{\cotimes 2} :(\Sigma\mathcal{X})^{\cotimes 2}\to\Sigma\mathcal{X}, \\
X^{\mathrm{gauge}}_1 &:= \Sigma\circ d\circ\Sigma^{-1}:\Sigma\mathcal{X}\to\Sigma\mathcal{X}.
\end{split}
\end{equation}
It is well-known, cf. for instance \cite{GetJonCyclic}, that the identities for $d$ and $m$ are equivalent to the equation
\begin{equation} \label{eqn_XgaugeCME}
[X^{\mathrm{gauge}},X^{\mathrm{gauge}}] = 0.
\end{equation}

\subsubsection{Anti-fields and anti-ghosts}

\begin{defi} \label{def_anticonstruction}
Given a graded vector bundle $W$ over a compact manifold $M$, we define the vector bundle
\[ W^! := W^*\underset{\mathbb{R}}{\otimes}\den{M} = \Hom_{\gf}\Bigl(W,\den{M}\underset{\mathbb{R}}{\otimes}\gf\Bigr); \]
where $W^*$ denotes the dual vector bundle. There is then a skew-symmetric nondegenerate local pairing of degree minus-one on the bundle $W\oplus\Sigma^{-1}W^!$ given by
\[ \innprodloc[f,w] := f(w), \quad f\in W^!,w\in W.\]
\end{defi}

\begin{example} \label{exm_dformsharp}
Suppose that $W$ is a graded vector bundle over a compact oriented manifold $M$ of dimension $n$ and consider the graded bundle
\[ V:=\Lambda^k T^*M\underset{\mathbb{R}}{\otimes}W, \]
where $\Lambda^k T^*M$ is concentrated in degree zero. Since $M$ is oriented, we may identify the density bundle $\den{M}$ with $\Lambda^n T^*M$. It follows then that we may identify
\[ V^! = \Lambda^{n-k} T^*M\underset{\mathbb{R}}{\otimes}W^*. \]
The bundles may be identified in such a way that the local pairing between $V^!$ and $V$ is then given by
\[ \innprodloc[\omega_{n-k}\otimes f,\omega_k\otimes w] = \omega_{n-k}\omega_kf(w); \]
for all $\omega_{n-k}\in\Lambda^{n-k} T^*M$, $\omega_k\in\Lambda^k T^*M$, $f\in W^*$ and $w\in W$.
\end{example}

Let $\mathcal{U}=\Gamma(M,U)$ be a local algebra and $\mathcal{V}=\Gamma(M,V)$ be a local $\mathcal{U}$-bimodule and consider the graded bundle
\[ \Sigma U\oplus V, \]
whose sections consist of the ghosts and fields, as before. Performing the construction of Definition \ref{def_anticonstruction}, we obtain the bundle
\begin{equation} \label{eqn_antifieldbundle}
E:=\Sigma U\oplus V\oplus \Sigma^{-1}V^!\oplus\Sigma^{-2} U^!,
\end{equation}
which has a nondegenerate skew-symmetric local pairing of degree minus-one. Denoting the sections of $U^!$ and $V^!$ by $\mathcal{U}^!$ and $\mathcal{V}^!$ respectively, we call the space $\Sigma^{-1}\mathcal{V}^!$ the \emph{anti-fields} and the space $\Sigma^{-2}\mathcal{U}^!$ the \emph{anti-ghosts}. In this way, $\mathcal{E}:=\Gamma(M,E)$ is made up of ghosts, fields, anti-fields and anti-ghosts.

\begin{example} \label{exm_Frobeniusbundle}
Let $\textgoth{A}$ be a bundle of graded Frobenius algebras over a compact oriented three-manifold $M$. This means that there is a nondegenerate degree zero trace map from the bundle $\textgoth{A}$ to the trivial line bundle $\gf$. The trace map provides an identification of the bundle $\textgoth{A}$ with its dual $\textgoth{A}^*$.

Setting $U:=\textgoth{A}$, the sections $\mathcal{U}$ form a local algebra. Setting $V:=T^*M\otimes_{\mathbb{R}}\textgoth{A}$, the sections $\mathcal{V}$ form a local $\mathcal{U}$-bimodule. From Example \ref{exm_dformsharp} we see that we may identify
\[ U^!=\Lambda^3 T^*M\underset{\mathbb{R}}{\otimes}\textgoth{A} \quad\text{and}\quad V^!=\Lambda^2 T^*M\underset{\mathbb{R}}{\otimes}\textgoth{A}. \]
These bundles may be identified in such a manner that the local pairing on the bundle
\[ E = \Sigma\Lambda^{\bullet}T^*M\underset{\mathbb{R}}{\otimes}\textgoth{A}, \]
which is defined by \eqref{eqn_antifieldbundle}, is given by the trace pairing on $\textgoth{A}$ and the multiplication of exterior powers. Note that in the above expression for the graded bundle $E$, the exterior algebra is given its usual grading. Consequently, the sections of $E$ are
\[ \mathcal{E} = \Sigma\dRham{M}\underset{\mathbb{R}}{\otimes}\textgoth{A}. \]
The skew-symmetric pairing on $\mathcal{E}$ derived from the local pairing on the bundle is then given by;
\begin{equation} \label{eqn_Frobeniuspairing}
\begin{split}
\innprod[\omega_1,\omega_2] &= (-1)^{|\omega_1|}\int_M \Tr(\omega_1\wedge\omega_2), \\
\innprod[\omega_3,\omega_0] &= (-1)^{|\omega_3|}\int_M \Tr(\omega_3\wedge\omega_3);
\end{split}
\end{equation}
for all $\omega_i\in\dRham[i]{M}\otimes_{\mathbb{R}}\textgoth{A}$. Here $|\omega_i|$ refers to the degree in $\dRham[i]{M}\otimes_{\mathbb{R}}\textgoth{A}$, so that if $\textgoth{A}$ is concentrated in degree zero then $|\omega_i|=i$.
\end{example}

\subsubsection{Encoding the local algebraic structures}

Let $\mathcal{U}$ be a local algebra and $\mathcal{V}$ be a twisted local $\mathcal{U}$-bimodule. Following Lemma 3.5.1 of \cite[\S 5.3.5]{CosEffThy}, the algebraic structure of the ghosts and fields may be encoded in a local action functional on
\[ \mathcal{E} = \Sigma\mathcal{U}\oplus\mathcal{V}\oplus\Sigma^{-1}\mathcal{V}^!\oplus\Sigma^{-2}\mathcal{U}^!. \]
Consider the subspace
\[ \ctalgdelim{\bigl([\Sigma\mathcal{U}\oplus\mathcal{V}]^{\dag}\bigr)}\cotimes\bigl[\Sigma^{-1}\mathcal{V}^!\oplus\Sigma^{-2}\mathcal{U}^!\bigr]^{\dag}\subset\KontHamP{\mathcal{E}}. \]
We say that the functionals in this subspace are \emph{linear on the fibers of the projection from $\mathcal{E}$ to $\Sigma\mathcal{U}\oplus\mathcal{V}$.}

\begin{lemma} \label{lem_Sgauge}
For every local algebra $\mathcal{U}$ and twisted local $\mathcal{U}$-bimodule $\mathcal{V}$, there is a unique local action functional
\[ S^{\mathrm{gauge}}\in\KontHamPL{\mathcal{E}} \]
such that:
\begin{itemize}
\item
$S^{\mathrm{gauge}}$ is linear on the fibers of the projection from $\mathcal{E}$ to $\Sigma\mathcal{U}\oplus\mathcal{V}$, and
\item
the local symplectic vector field on $\mathcal{E}$ associated to $S^{\mathrm{gauge}}$ by the correspondence of Proposition \ref{prop_localisHam} extends the vector field $X^{\mathrm{gauge}}$ that is defined on $\Sigma\mathcal{U}\oplus\mathcal{V}$ by Equation \eqref{eqn_Xgauge}.
\end{itemize}
Furthermore, $S^{\mathrm{gauge}}$ has degree zero and satisfies the classical master equation
\[ \{S^{\mathrm{gauge}},S^{\mathrm{gauge}}\} = 0. \]
\end{lemma}

\begin{proof}
Set $\mathcal{X}:=\mathcal{U}\oplus\Sigma^{-1}\mathcal{V}$ and denote the projection from $\mathcal{E}$ to $\Sigma\mathcal{X}$ by $\pi_{\mathcal{X}}$. A vector field $X'$ on $\mathcal{E}$ extends a vector field $X$ on $\Sigma\mathcal{X}$ if its restriction to $\ctalgdelim{([\Sigma\mathcal{X}]^{\dag})}$ agrees with $X$. This is equivalent to the equation
\[ \pi_{\mathcal{X}}\circ X'_n = X_n\circ\pi_{\mathcal{X}}^{\cotimes n}, \quad\text{for all }n\geq 0. \]

Note that if $S\in\KontHamP{\mathcal{E}}$ is a Hamiltonian functional that is linear on the fibers of $\pi_{\mathcal{X}}$, and whose corresponding symplectic vector field $X'$ on $\mathcal{E}$ extends a vector field $X$ on $\Sigma\mathcal{X}$, then
\begin{displaymath}
\begin{split}
S(x_1,\ldots,x_n,y) &= S\cycsum{n+1}(x_1,\ldots,x_n,y) = \innprod[X'_n(x_1,\ldots,x_n),y] \\
&= \innprod[\pi_{\mathcal{X}} X'_n(x_1,\ldots,x_n),y] = \innprod[X_n(x_1,\ldots,x_n),y];
\end{split}
\end{displaymath}
for all $x_1,\ldots,x_n\in\Sigma\mathcal{X}$ and $y\in\Sigma^{-1}\mathcal{V}^!\oplus\Sigma^{-2}\mathcal{U}^!$. It follows that $S^{\mathrm{gauge}}$ is unique and is given by the formula
\begin{equation} \label{eqn_Sgauge}
S^{\mathrm{gauge}}(x_1,\ldots,x_n,y) = \innprod[X^{\mathrm{gauge}}_n(x_1,\ldots,x_n),y].
\end{equation}
By Proposition \ref{prop_HamSympcorrespondence} and \ref{prop_oddbracket}, the functional $S^{\mathrm{gauge}}$ satisfies the classical master equation if and only if the vector field $[X',X']$ vanishes, where $X'$ is the symplectic vector field corresponding to $S^{\mathrm{gauge}}$. Note that since $X'$ extends $X^{\mathrm{gauge}}$, it follows from Equation \eqref{eqn_XgaugeCME} that $[X',X']$ vanishes on $[\Sigma\mathcal{X}]^{\dag}$.

To prove that it vanishes on the complementary subspace of $\mathcal{E}^{\dag}$, we note that it follows from the fact that $S^{\mathrm{gauge}}$ is linear on the fibers of $\pi_{\mathcal{X}}$ that the vector field $X'$ will map this complimentary subspace to
\[ X'\bigl(\bigl[\Sigma^{-1}\mathcal{V}^!\oplus\Sigma^{-2}\mathcal{U}^!\bigr]^{\dag}\bigr) \subset \prod_{n=0}^\infty\prod_{k=0}^n \bigl([\Sigma\mathcal{X}]^{\dag}\bigr)^{\cotimes k}\cotimes\bigl[\Sigma^{-1}\mathcal{V}^!\oplus\Sigma^{-2}\mathcal{U}^!\bigr]^{\dag}\cotimes\bigl([\Sigma\mathcal{X}]^{\dag}\bigr)^{\cotimes n-k}. \]
It then follows that the same thing will be true of the vector field $[X',X']$, since $X'$ extends $X^{\mathrm{gauge}}$, which maps $\ctalgdelim{([\Sigma\mathcal{X}]^{\dag})}$ to itself.

Now note that for all $x_0,\ldots,x_n\in\Sigma\mathcal{X}$ and $y\in\Sigma^{-1}\mathcal{V}^!\oplus\Sigma^{-2}\mathcal{U}^!$,
\begin{multline*}
\innprod[{[X',X'](x_1,\ldots,x_k,y,x_{k+1},\ldots x_n),x_0}] = \pm\innprod[{[X',X'](x_{k+1},\ldots x_n,x_0,\ldots,x_k),y}] \\
= \pm\innprod[{[X^{\mathrm{gauge}},X^{\mathrm{gauge}}](x_{k+1},\ldots x_n,x_0,\ldots,x_k),y}] = 0;
\end{multline*}
where we have used the fact that $[X',X']$ is symplectic and extends $[X^{\mathrm{gauge}},X^{\mathrm{gauge}}]$, the latter of which vanishes. It therefore follows that $[X',X']$ vanishes on the subspace of $\mathcal{E}^{\dag}$ that is complementary to $[\Sigma\mathcal{X}]^{\dag}$, and hence on the whole of $\mathcal{E}^{\dag}$, as claimed.
\end{proof}

\subsubsection{Gauge symmetry}

Suppose that $\mathcal{U}$ is a local algebra and $\mathcal{V}$ is a twisted local $\mathcal{U}$-bimodule with twisting operator $d$. Given $a\in\mathcal{U}$ we may consider the vector fields
\[ da \in \mathcal{V}=\Hom_{\gf}(\gf,\mathcal{V}) \quad\text{and}\quad [a,-]:v\mapsto [a,v] \in\Hom_{\gf}(\mathcal{V},\mathcal{V}). \]
We obtain a map of graded Lie algebras
\[ \mathcal{U}\to\Derdelim{\bigl(\ctalgdelim{(\mathcal{V}^{\dag})}\bigr)}, \qquad a\mapsto -da + [a,-] \]
from the commutator algebra of $\mathcal{U}$ to the Lie algebra of vector fields on $\mathcal{V}$. Strictly speaking, the latter is equipped with the opposite commutator bracket since taking the dual reverses the order of composition.

\begin{defi}
Given a local algebra $\mathcal{U}$ and a twisted local $\mathcal{U}$-bimodule $\mathcal{V}$, we will say that a functional $S^{\mathcal{V}}\in\KontHam{\mathcal{V}}$ is \emph{invariant} under the action of $\mathcal{U}$ if
\[ (-da + [a,-])S^{\mathcal{V}} = 0, \quad\text{for all }a\in\mathcal{U}. \]
\end{defi}

Recall that we have a local vector field $X^{\mathrm{gauge}}$ defined on $\Sigma\mathcal{U}\oplus\mathcal{V}$ by Equation \eqref{eqn_Xgauge} and a local functional $S^{\mathrm{gauge}}\in\KontHamPL{\mathcal{E}}$ defined by Lemma \ref{lem_Sgauge}.

\begin{prop} \label{prop_CMEgauge}
Let $\mathcal{U}$ be a local algebra and $\mathcal{V}$ be a twisted local $\mathcal{U}$-bimodule. Given a local functional $S^{\mathcal{V}}\in\KontHamL{\mathcal{V}}$ of degree zero, consider the local action functional
\[ S := S^{\mathrm{gauge}} + S^{\mathcal{V}} \in\KontHamL{\mathcal{E}}. \]
\begin{enumerate}
\item \label{itm_CMEgauge1}
The functional $S^{\mathcal{V}}$ satisfies the classical master equation.
\item \label{itm_CMEgauge2}
The functional $S$ satisfies the classical master equation if and only if $X^{\mathrm{gauge}}S^{\mathcal{V}}$ vanishes. The latter holds if and only if for all $a\in\mathcal{U}$ and $l\geq 1$;
\begin{displaymath}
\begin{split}
& S^{\mathcal{V}}_1(da) = 0, \quad\text{and} \\
& S^{\mathcal{V}}_{l+1}\cycsum{l+1}(da,v_1,\ldots,v_l) = S^{\mathcal{V}}_l\cycsum{l}(a v_1,v_2,\ldots,v_l) - (-1)^pS^{\mathcal{V}}_l\cycsum{l}(v_1,\ldots,v_{l-1},v_l a);
\end{split}
\end{displaymath}
where $p:=|a|(|v_1|+\cdots+|v_l|)$, the map $\cycsum{l}$ is defined by \eqref{eqn_cyclicsum} and $S^{\mathcal{V}}_l\in\bigl[(\mathcal{V}^{\dag})^{\cotimes l}\bigr]_{\cyc{l}}$ denote the components of $S^{\mathcal{V}}$, as in Definition \ref{def_Hamfunctional}.
\item \label{itm_CMEgauge3}
If $S$ satisfies the classical master equation then the functional $S^{\mathcal{V}}$ will be invariant under the action of $\mathcal{U}$.
\end{enumerate}
\end{prop}

\begin{rem}
Note that the theorem still holds if we replace $S^{\mathcal{V}}$ with $-S^{\mathcal{V}}$, and likewise with $S^{\mathrm{gauge}}$.
\end{rem}

\begin{proof}[Proof of Proposition \ref{prop_CMEgauge}]
Let $Y$ be the symplectic vector field on $\mathcal{E}$ that is associated to the local functional $S^{\mathcal{V}}$ by Proposition \ref{prop_HamSympcorrespondence}. We have
\[ \innprod[{Y_l(x_1,\ldots,x_l),v'}] = S^{\mathcal{V}}_{l+1}\cycsum{l+1}(x_1,\ldots,x_l,v') = 0, \]
for all $x_1,\ldots,x_l\in\mathcal{E}$ and $v'\in\Sigma^{-1}\mathcal{V}^!$. From this it follows that the $\mathcal{V}$ component of $Y_l$ vanishes for all $l\geq 0$, or equivalently, that the vector field $Y$ vanishes on $\ctalgdelim{(\mathcal{V}^{\dag})}$. Hence
\[ \{S^{\mathcal{V}},S^{\mathcal{V}}\} = YS^{\mathcal{V}} = 0. \]

Having now proven \eqref{itm_CMEgauge1}, it follows from Lemma \ref{lem_Sgauge} that
\[ \{S,S\} = 2\{S^{\mathrm{gauge}},S^{\mathcal{V}}\} = 2X^{\mathrm{gauge}}S^{\mathcal{V}}. \]
To compute the right-hand side, denote the canonical $n$-cycle by $z_n:=(1\ldots n)$ and the components of the vector field $X^{\mathrm{gauge}}$ defined by \eqref{eqn_Xgauge} by
\[ X_d:\Sigma\mathcal{U}\to\mathcal{V}, \quad X_L:\Sigma\mathcal{U}\cotimes\mathcal{V}\to\mathcal{V} \quad\text{and}\quad X_R:\mathcal{V}\cotimes\Sigma\mathcal{U}\to\mathcal{V}. \]
We compute
\begin{displaymath}
\begin{split}
X^{\mathrm{gauge}}S^{\mathcal{V}} &= \sum_{l=1}^{\infty}\sum_{i=1}^l S^{\mathcal{V}}_l\bigl(\mathds{1}^{\cotimes i-1}\cotimes(X_d+X_L+X_R)\cotimes\mathds{1}^{\cotimes l-i}\bigr) \\
&= S^{\mathcal{V}}_1 X_d + \sum_{l=1}^{\infty}\sum_{i=1}^{l+1} S^{\mathcal{V}}_{l+1}(\mathds{1}^{\cotimes i-1}\cotimes X_d\cotimes\mathds{1}^{\cotimes l+1-i})z_{l+1}^{i-1} \\
&+ \sum_{l=1}^{\infty}\sum_{i=1}^l S^{\mathcal{V}}_l(\mathds{1}^{\cotimes i-1}\cotimes X_L\cotimes\mathds{1}^{\cotimes l-i})z_{l+1}^{i-1} + S^{\mathcal{V}}_l(\mathds{1}^{\cotimes i-1}\cotimes X_R\cotimes\mathds{1}^{\cotimes l-i})z_{l+1}^i \\
&= S^{\mathcal{V}}_1 X_d + \sum_{l=1}^{\infty} \bigl(S^{\mathcal{V}}_{l+1}\cycsum{l+1}(X_d\cotimes\mathds{1}^{\cotimes l}) + S^{\mathcal{V}}_l\cycsum{l}(X_L\cotimes\mathds{1}^{\cotimes l-1}) + S^{\mathcal{V}}_l\cycsum{l}(\mathds{1}^{\cotimes l-1}\cotimes X_R)z_{l+1}^{-1}\bigr).
\end{split}
\end{displaymath}
On the second line we have used the fact that we are working modulo the action of the cyclic group. This proves \eqref{itm_CMEgauge2}.

Finally, given $a\in\mathcal{U}$ consider the vector field
\[ Y_a := \Sigma a \in\Sigma\mathcal{U}=\Hom_{\gf}(\gf,\Sigma\mathcal{U}) \]
on $\Sigma\mathcal{X}=\Sigma\mathcal{U}\oplus\mathcal{V}$. We have
\[ Y_a X^{\mathrm{gauge}}S^{\mathcal{V}} = [Y_a,X^{\mathrm{gauge}}]S^{\mathcal{V}} = (da - [a,-])S^{\mathcal{V}}, \]
from which \eqref{itm_CMEgauge3} follows.
\end{proof}

\subsubsection{Noncommutative Chern-Simons theory} \label{sec_NCChernSimonsCME}

As an example of the preceding constructions, we will consider a noncommutative analogue of Chern-Simons theory.

\begin{example} \label{exm_NCCStheory}
The initial data for our construction consists of a compact oriented three-manifold $M$ and a Frobenius algebra $\textgoth{A}$, which for the sake of simplicity we will assume is concentrated in degree zero. The local algebra $\mathcal{U}$ and twisted local $\mathcal{U}$-bimodule $\mathcal{V}$ are then defined according to Example \ref{exm_AUMmodule}.

Applying the construction of Example \ref{exm_Frobeniusbundle} we obtain the space
\[ \mathcal{E} = \Sigma\dRham{M}\underset{\mathbb{R}}{\otimes}\textgoth{A}. \]
This has a skew-symmetric pairing of degree minus-one given by Equation \eqref{eqn_Frobeniuspairing}. From Lemma \ref{lem_Sgauge} we obtain a local action functional $S^{\mathrm{gauge}}\in\KontHam{\mathcal{E}}$ defined by Equation \eqref{eqn_Sgauge}. Define a local action functional $S^{\mathcal{V}}\in\KontHam{\mathcal{V}}$ on the fields $\mathcal{V}$ by;
\begin{displaymath}
\begin{split}
S^{\mathcal{V}}_2(A_1,A_2) &:= \frac{1}{2}\int_M \Tr(A_1\wedge dA_2), \\
S^{\mathcal{V}}_3(A_1,A_2,A_3) &:= \frac{1}{3}\int_M\Tr(A_1\wedge A_2\wedge A_3).
\end{split}
\end{displaymath}
It may be checked directly that $S^{\mathcal{V}}$ satisfies the constraints of Proposition \ref{prop_CMEgauge}\eqref{itm_CMEgauge2}. Therefore, the local functional
\[ S := S^{\mathrm{gauge}} + S^{\mathcal{V}} \]
satisfies the classical master equation.

It may be computed directly that $S=\kappa+I$, where $\kappa$ is the kinetic term defined by \eqref{eqn_kinetic} for the free BV-theory that is given by taking $Q$ to be the exterior derivative on $\mathcal{E}$, and the interaction term is given by
\begin{equation} \label{eqn_CSinteraction}
I(A_1,A_2,A_3) := (-1)^{|A_2|}\frac{1}{3}\int_M\Tr(A_1\wedge A_2\wedge A_3), \quad A_1,A_2,A_3\in\mathcal{E}.
\end{equation}
This free BV-theory is the same as that defined by Example \ref{exm_CStheory} for the commutator algebra of $\textgoth{A}$. From Example \ref{exm_CMEinteraction} we see that the interaction $I$ satisfies the classical master equation \eqref{eqn_CMEinteraction}.
\end{example}

\section{The renormalization group flow} \label{sec_RGFlow}

In the preceding section, we discussed the analogue of the classical master equation for the BV-formalism in the context of noncommutative geometry. As in \cite{CosEffThy}, in order to discuss the quantum master equation, it will first be necessary to recall from \cite{NCRGF} some of the basic details and results concerning the formulation of the notion of an effective field theory within the framework of noncommutative geometry. This is because the naive quantum master equation suffers from the same singular behavior that manifests itself in the renormalization of effective field theories, and so must instead be formulated as a family of equations parameterized according to the length scale of the theory.

In this section we will recall the definition of the renormalization group flow in noncommutative geometry and the notion of an effective field theory. Following \cite{CosEffThy}, we will refer to the latter as a ``pretheory'', since we too wish to reserve the word ``theory'' in the context of the BV-formalism for those structures satisfying the quantum master equation. Our aim here is to expeditiously and concisely recall the necessary results from \cite{NCRGF}, which the reader may wish to consult if they are interested in additional details.

In what follows, $\mathcal{A}$ will denote a commutative topological algebra of the form described in Section \ref{sec_gaugefixingfamily}.

\subsection{Feynman amplitudes}

The renormalization group flow is defined through the use of Feynman amplitudes, and these have different definitions in commutative and noncommutative geometry.

\subsubsection{Attaching topological vector spaces to a set}

Here we recall some basic material from the theory of modular operads \cite{GetKap}.

\begin{defi}
Let $\mathcal{V}$ be a $\mathbb{Z}$-graded nuclear space and $X$ a finite set of cardinality $n$. Define
\[ \mathcal{V}(\!(X)\!):=\Biggl[\bigoplus_{h\in\Bij\left(\{1,\ldots,n\},X\right)}\mathcal{V}^{\cotimes n}\Biggr]^{\sg{n}}, \]
where $\Bij\left(\{1,\ldots,n\},X\right)$ denotes the set of bijections and $\sg{n}$ acts simultaneously on both this set of bijections and the tensor product $\mathcal{V}^{\cotimes n}$. This is a functor on the groupoid of finite sets.
\end{defi}

There are canonical maps
\begin{equation} \label{eqn_attachmapssym}
\bigl[\mathcal{V}^{\cotimes n}\bigr]^{\sg{n}}\to\mathcal{V}(\!(X)\!)\to\bigl[\mathcal{V}^{\cotimes n}\bigr]_{\sg{n}}
\end{equation}
and canonical identifications
\[ \mathcal{V}(\!(X\sqcup Y)\!) = \mathcal{V}(\!(X)\!)\cotimes\mathcal{V}(\!(Y)\!). \]
Additionally, when $\mathcal{V}$ is a nuclear Fre\'chet space then by Proposition \ref{prop_tensordual} we may identify
\[ \mathcal{V}^{\dag}(\!(X)\!) = \mathcal{V}(\!(X)\!)^{\dag}. \]

More generally, given a finite set $\mathscr{X}$ of cardinality $k$ and a function $\mathcal{V}$ that assigns to every element $X$ of $\mathscr{X}$, a nuclear space $\mathcal{V}_X$, we define
\[ \underset{X\in\mathscr{X}}{\widehat{\bigotimes}}\mathcal{V}_X := \Biggl[\bigoplus_{\mathbf{X}\in\Bij(\{1,\ldots,k\},\mathscr{X})}\mathcal{V}_{X_1}\cotimes\ldots\cotimes\mathcal{V}_{X_k}\Biggr]^{\sg{k}}. \]

When $\mathscr{X}$ is a finite collection of finite sets and $\mathcal{V}$ is a nuclear space, we may identify
\[ \underset{X\in\mathscr{X}}{\widehat{\bigotimes}}\mathcal{V}(\!(X)\!) = \mathcal{V}\Bigl(\!\Bigl(\coprod_{X\in\mathscr{X}}X\Bigr)\!\Bigr). \]
Now suppose that we are given a function $v$ that assigns to every set $X$ in $\mathscr{X}$, a tensor $v_X$ in $\mathcal{V}(\!(X)\!)$. Suppose also that every tensor $v_X$ has even degree, except possibly one of them. Then there is a canonically defined tensor
\[ \underset{X\in\mathscr{X}}{\otimes} v_X \in \underset{X\in\mathscr{X}}{\widehat{\bigotimes}}\mathcal{V}(\!(X)\!) = \mathcal{V}\Bigl(\!\Bigl(\coprod_{X\in\mathscr{X}}X\Bigr)\!\Bigr). \]

\subsubsection{Stable graphs in commutative geometry}

Here we will recall from \cite{CosEffThy} the definition of the Feynman amplitude in commutative geometry.

\begin{defi} \label{def_stablegraph}
A \emph{stable graph} $\Gamma$ consists of the following data:
\begin{itemize}
\item
A finite set $H(\Gamma)$, whose members are called the \emph{half-edges} of $\Gamma$.
\item
A finite set $V(\Gamma)$, whose members are called the \emph{vertices} of $\Gamma$.
\item
A map $\rho_{\Gamma}:H(\Gamma)\to V(\Gamma)$. We call the preimage of a vertex $v$ of $\Gamma$ under this map the set of \emph{half-edges incident to the vertex $v$}. By an abuse of notation, this preimage will be denoted simply by the same letter $v$. We define the valency of $v$ to be its cardinality~$|v|$.
\item
An involution $\kappa_{\Gamma}:H(\Gamma)\to H(\Gamma)$. The fixed points of this involution are called the \emph{legs} of $\Gamma$. The subset of all legs will be denoted by $L(\Gamma)$. Those orbits generated by $\kappa_{\Gamma}$ that form unordered pairs will be called the \emph{edges} of $\Gamma$. The set of all edges will be denoted by~$E(\Gamma)$.
\item
A map $l$ that assigns to every vertex $v\in V(\Gamma)$ a nonnegative integer $l(v)$, called the \emph{loop number} of $v$.
\end{itemize}
We impose the following restriction on the valencies of the vertices:
\[ 2l(v)+|v|\geq 3, \quad\text{for all }v\in V(\Gamma). \]

We define the \emph{loop number} of the stable graph $\Gamma$ by
\[ \ell(\Gamma) := \mathbf{b}_1(\Gamma) + \sum_{v\in V(\Gamma)}l(v), \]
where $\mathbf{b}_1(\Gamma)$ is the first Betti number of the graph.
\end{defi}

\begin{rem}
In \cite{GetKap} the term ``genus'' was used in place of ``loop number'', but as we will soon see, this terminology will conflict with some of our subsequent definitions.
\end{rem}

\begin{defi}
Given a free BV-theory $\mathcal{E}$ we define
\[ \intcomm{\mathcal{E}}:=\csalgdelim{\bigl(\mathcal{E}^{\dag}\bigr)}[[\hbar]] = \prod_{i=0}^\infty\hbar^i\csalgdelim{\bigl(\mathcal{E}^{\dag}\bigr)} = \gf[[\hbar]]\cotimes\csalgdelim{\bigl(\mathcal{E}^{\dag}\bigr)}. \]
We can write any functional $I\in\intcomm{\mathcal{E}}$ as a series
\[ I = \sum_{i,j=0}^{\infty}\hbar^i I_{ij}, \quad\text{where } I_{ij}\in\bigl[(\mathcal{E}^{\dag})^{\cotimes j}\bigr]_{\sg{j}}. \]
We will say that $I$ is an \emph{interaction} if it has degree zero and if for all $i,j\geq 0$;
\begin{equation} \label{eqn_interactionconstraints}
I_{ij} = 0, \quad\text{whenever }2i+j<3.
\end{equation}

More generally, we will say that
\[ I\in\intcomm{\mathcal{E},\mathcal{A}}:=\intcomm{\mathcal{E}}\cotimes\mathcal{A} \]
is a \emph{family of interactions} if it has total degree zero and satisfies the constraints prescribed by \eqref{eqn_interactionconstraints}, where now~$I_{ij}\in\bigl[(\mathcal{E}^{\dag})^{\cotimes j}\bigr]_{\sg{j}}\cotimes\mathcal{A}$. We will denote the subspace of $\intcomm{\mathcal{E},\mathcal{A}}$ consisting of those functionals satisfying \eqref{eqn_interactionconstraints} by $\intcommP{\mathcal{E},\mathcal{A}}$, and the subspace consisting of all families of interactions by $\intcommI{\mathcal{E},\mathcal{A}}$.
\end{defi}

Before proceeding to provide the definition of the Feynman amplitude, we note that if $X$ is any finite set of cardinality $n$ then there is a canonical map
\begin{equation} \label{eqn_attachtovertexsym}
\xymatrix{
\bigl[(\mathcal{E}^{\dag})^{\cotimes n}\bigr]_{\sg{n}}\cotimes\mathcal{A} \ar[rr]^{\sum_{\varsigma\in\sg{n}}\varsigma} && \bigl[(\mathcal{E}^{\dag})^{\cotimes n}\bigr]^{\sg{n}}\cotimes\mathcal{A} \ar[r] & \mathcal{E}^{\dag}(\!(X)\!)\cotimes\mathcal{A},
}
\end{equation}
which uses the left-hand map of \eqref{eqn_attachmapssym}. Additionally, for any such set $X$ there is a map
\[ \mathcal{A}(\!(X)\!)\to\bigl[\mathcal{A}^{\cotimes n}\bigr]_{\sg{n}}\to\mathcal{A}, \]
which uses the right-hand map of \eqref{eqn_attachmapssym} as the first arrow and the commutative multiplication $\mu_{\mathcal{A}}$ on $\mathcal{A}$ as the last.

\begin{defi}
Let $\mathcal{E}$ be a free BV-theory, $I\in\intcommI{\mathcal{E},\mathcal{A}}$ be a family of interactions and $P\in\mathcal{E}\cotimes\mathcal{E}\cotimes\mathcal{A}$ be a family of propagators. Given a stable graph $\Gamma$, we attach the interaction $I$ to each vertex $v\in V(\Gamma)$ by considering the image of $I_{l(v)|v|}$ under the map \eqref{eqn_attachtovertexsym}, which we denote by
\[ I(v)\in\mathcal{E}^{\dag}(\!(v)\!)\cotimes\mathcal{A}, \quad v\in V(\Gamma). \]
We then define
\begin{displaymath}
\begin{split}
I(\Gamma) := \underset{v\in V(\Gamma)}{\otimes} I(v) \in \underset{v\in V(\Gamma)}{\widehat{\bigotimes}}\bigl(\mathcal{E}^{\dag}(\!(v)\!)\cotimes\mathcal{A}\bigr) &= \mathcal{E}^{\dag}\bigl(\!\bigl(H(\Gamma)\bigr)\!\bigr)\cotimes\mathcal{A}\bigl(\!\bigl(V(\Gamma)\bigr)\!\bigr) \\
&= \Hom_{\gf}\bigl(\mathcal{E}\bigl(\!\bigl(H(\Gamma)\bigr)\!\bigr),\mathcal{A}\bigl(\!\bigl(V(\Gamma)\bigr)\!\bigr)\bigr).
\end{split}
\end{displaymath}

Next, we attach $P$ to every edge $e$ of $\Gamma$ using the left-hand map of \eqref{eqn_attachmapssym} and denote the result by
\[ P(e)\in\mathcal{E}(\!(e)\!)\cotimes\mathcal{A}, \quad e\in E(\Gamma). \]
We then define
\begin{equation} \label{eqn_graphpropagator}
P(\Gamma) := \underset{e\in E(\Gamma)}{\otimes} P(e) \in \underset{e\in E(\Gamma)}{\widehat{\bigotimes}}\bigl(\mathcal{E}(\!(e)\!)\cotimes\mathcal{A}\bigr) = \mathcal{E}\Bigl(\!\Bigl(\bigcup_{e\in E(\Gamma)}e\Bigr)\!\Bigr)\cotimes\mathcal{A}\bigl(\!\bigl(E(\Gamma)\bigr)\!\bigr).
\end{equation}

We define the \emph{Feynman amplitude}
\[ F_{\Gamma}(I,P)\in\Hom_{\gf}\bigl(\mathcal{E}\bigl(\!\bigl(L(\Gamma)\bigr)\!\bigr),\mathcal{A}\bigr) = \mathcal{E}^{\dag}\bigl(\!\bigl(L(\Gamma)\bigr)\!\bigr)\cotimes\mathcal{A} \]
by evaluating $I(\Gamma)$ on $P(\Gamma)$, that is we form the composite
\begin{equation} \label{eqn_Feynmanamplitude}
\xymatrix{
\mathcal{E}\bigl(\!\bigl(L(\Gamma)\bigr)\!\bigr) \ar[rr]^-{x\mapsto P(\Gamma)\otimes x} \ar[d]_{F_{\Gamma}(I,P)} &&
\mathcal{E}\Bigl(\!\Bigl(\underset{e\in E(\Gamma)}{\bigcup}e\Bigr)\!\Bigr)\cotimes\mathcal{A}\bigl(\!\bigl(E(\Gamma)\bigr)\!\bigr)\cotimes\mathcal{E}\bigl(\!\bigl(L(\Gamma)\bigr)\!\bigr) \ar@{=}[d] \\
\mathcal{A} &
\mathcal{A}\bigl(\!\bigl(V(\Gamma)\bigr)\!\bigr)\cotimes\mathcal{A}\bigl(\!\bigl(E(\Gamma)\bigr)\!\bigr) \ar[l]_-{\mu} &
\mathcal{E}\bigl(\!\bigl(H(\Gamma)\bigr)\!\bigr)\cotimes\mathcal{A}\bigl(\!\bigl(E(\Gamma)\bigr)\!\bigr) \ar[l]_-{I(\Gamma)\cotimes\mathds{1}}
}
\end{equation}

The \emph{Feynman weight}
\[ w_{\Gamma}(I,P)\in\bigl[(\mathcal{E}^{\dag})^{\cotimes j}\bigr]_{\sg{j}}\cotimes\mathcal{A}, \]
where $j:=|L(\Gamma)|$ is the number of legs of $\Gamma$, is obtained from the Feynman amplitude $F_{\Gamma}(I,P)$ by applying the right-hand map of \eqref{eqn_attachmapssym}.
\end{defi}

\subsubsection{Stable ribbon graphs in noncommutative geometry} \label{sec_stabribgraph}

We now recall from \cite{NCRGF} the definition of the Feynman amplitude for a stable ribbon graph. This takes place in the framework of noncommutative geometry that was introduced by Kontsevich in \cite{KontSympGeom, KontFeyn} and further developed in \cite{baran}.

\begin{defi}
A \emph{cyclic decomposition} of a set $X$ is a partition of $X$ into cyclically ordered subsets.
\end{defi}

This is precisely the same thing as a permutation on $X$, which defines a cyclic decomposition of $X$ through its factorization into a product of disjoint cycles.

\begin{defi}
A \emph{stable ribbon graph} consists of the following data:
\begin{itemize}
\item
A finite set of half-edges $H(\Gamma)$.
\item
A finite set of vertices $V(\Gamma)$.
\item
A map $\rho_{\Gamma}:H(\Gamma)\to V(\Gamma)$, which determines the incident half-edges of a vertex, as explained previously in Definition \ref{def_stablegraph}.
\item
An involution $\kappa_{\Gamma}:H(\Gamma)\to H(\Gamma)$ as before, describing the legs and edges of the graph.
\item
A cyclic decomposition $C(v)$ of the incident half-edges of every vertex $v\in V(\Gamma)$.
\item
Maps $g$ and $b$ that assign nonnegative integers $g(v)$ and $b(v)$ to every vertex $v$ of $\Gamma$, called the \emph{genus} and \emph{boundary} of $v$ respectively.
\end{itemize}
We define the \emph{loop number} of a vertex $v$ of $\Gamma$ by
\[ l(v) := 2g(v)+b(v)+|C(v)|-1. \]
We impose the requirement that for all $v\in V(\Gamma)$,
\[ |C(v)|+b(v)>0 \quad\text{and}\quad 2l(v)+|v|\geq 3. \]
\end{defi}

Stable ribbon graphs were introduced by Kontsevich in \cite{KontAiry}, where they were used to parameterize regions in compactifications of the moduli space of Riemann surfaces. The edges of a stable ribbon graph may be contracted, which in turn describes the degeneration of the corresponding Riemann surface. In the following definition we will simply summarize the rules, but more detailed discussions are available from \cite{ChLaOTFT, HamCompact, KontAiry, Mondello}.

\begin{defi} \label{def_contractedge}
Let $\Gamma$ be a stable ribbon graph. Given an edge $e$ of $\Gamma$, we may form a new stable ribbon graph $\Gamma/e$ by contracting the edge $e$ according to the following set of rules:
\begin{enumerate}
\item
If the edge $e$ joins distinct vertices $v_1$ and $v_2$ of $\Gamma$, then we delete this edge and join these two vertices to form a new vertex $v$ in $\Gamma/e$. The cyclic decomposition of $v$ is inherited unaltered from those of $v_1$ and $v_2$, except that the cycles $c_1$ of $v_1$ and $c_2$ of $v_2$ that intersect $e$ are joined in the obvious manner to form a new cycle $c$ of $v$, see for instance Figure \ref{fig_ContractEdgeBracket} in Section \ref{sec_BVFeynmandiagrams}. The genus and boundary of $v$ are then defined as the sum of those for $v_1$ and $v_2$.

The only exception to the above rule occurs when both $c_1$ and $c_2$ contain just a single half-edge, for then the resulting cycle $c$ would be empty. Instead of this, the boundary $b(v)$ is further increased by one.

\item
If the edge $e$ is a loop that joins a vertex $v$ to itself, then we distinguish two cases:
\begin{enumerate}
\item
If the endpoints of $e$ lie in distinct cycles $c_1$ and $c_2$ of $v$, then we delete the edge $e$ and these cycles are joined to form a new cycle $c$ of the vertex $v$, in the same manner as above. The genus $g(v)$ is increased by one and the boundary $b(v)$ is unaltered.

The only exception occurs, as before, when both $c_1$ and $c_2$ consist of a single half-edge. In this case, as above, we increase $b(v)$ by one.
\item
Finally, if the endpoints of $e$ both lie in the same cycle $c$ of $v$, then the edge $e$ is deleted and the cycle $c$ is split into two cycles $c_1$ and $c_2$, with canonically defined cyclic orderings inherited from that on $c$; see for instance Figure \ref{fig_ContractLoopCobracket} in Section \ref{sec_BVFeynmandiagrams}. The genus $g(v)$ and boundary $b(v)$ are unaltered.

An exception to this rule occurs when the endpoints of $e$ lie next to each other in the cyclic ordering on $c$, for in this case one of the cycles $c_1$ or $c_2$ would be empty. Instead of allowing an empty cycle, the boundary $b(v)$ is increased by one---unless \emph{both} $c_1$ and $c_2$ would be empty, in which case the boundary $b(v)$ is increased by \emph{two}.
\end{enumerate}
\end{enumerate}
\end{defi}

It may be checked that the result of contracting the edges of a stable ribbon graph does not depend upon the order in which they are contracted. To define the number of boundary components of a stable ribbon graph, we do the following.

\begin{defi}
Given a connected\footnote{By convention, the empty graph is not connected.} stable ribbon graph $\Gamma$, we denote the collection of all cycles of $\Gamma$ by
\[ C(\Gamma):=\bigcup_{v\in V(\Gamma)} C(v). \]
Denote by $c_{\Gamma}$, the permutation of the half-edges of $\Gamma$ that is formed by combining the permutations associated to the cyclic decompositions $C(v)$, for each vertex $v$ of $\Gamma$.

Define the permutation $\beta_{\Gamma}:=c_{\Gamma}\kappa_{\Gamma}$. Count the number of cycles (including $1$-cycles) in the cyclic decomposition of $\beta_{\Gamma}$ and denote the result by $\beta_{\Gamma}^{\#}$.

We define the genus and the number of boundary components of $\Gamma$ by
\begin{displaymath}
\begin{split}
g(\Gamma) &:= 1 - |V(\Gamma)| + \frac{1}{2}\bigl(|E(\Gamma)|+|C(\Gamma)|-\beta_{\Gamma}^{\#}\bigr) + \sum_{v\in V(\Gamma)} g(v), \\
B(\Gamma) &:= \beta_{\Gamma}^{\#} + \sum_{v\in V(\Gamma)} b(v).
\end{split}
\end{displaymath}
The \emph{loop number} of $\Gamma$ is defined by
\begin{equation} \label{eqn_loopnum}
\ell(\Gamma) := 2g(\Gamma) + B(\Gamma) - 1 = \mathbf{b}_1(\Gamma) + \sum_{v\in V(\Gamma)} l(v).
\end{equation}

We will also define
\[ b(\Gamma):=B(\Gamma) - |C(\Gamma/\Gamma)|, \]
where we use $\Gamma/\Gamma$ to denote the graph $\Gamma$ with all of its edges contracted.
\end{defi}

As a result of the rules introduced in Definition \ref{def_contractedge}, the numbers $g(\Gamma)$, $B(\Gamma)$ and $b(\Gamma)$ do not change when we contract an edge in a graph. In particular, $g(\Gamma)$ and $b(\Gamma)$ are both nonnegative. These numbers do indeed track the genus and the number of boundary components of a surface associated to the graph, see \cite[\S 3.1.3]{NCRGF} for details.

\begin{defi}
Given a free BV-theory $\mathcal{E}$, we define
\[ \intnuP{\mathcal{E}} := \csalgP{\KontHam{\mathcal{E}}}, \]
where $\KontHam{\mathcal{E}}$ was defined by Definition \ref{def_KontHam}. We then define
\[ \intnoncomm{\mathcal{E}} := \intnuP{\mathcal{E}}[[\gamma]] = \gf[[\gamma]]\cotimes\intnuP{\mathcal{E}}. \]
\end{defi}

In order to explain our choice of notation above, denote the generator of $\gf$ by $\nu$ and write
\begin{displaymath}
\begin{split}
\intnoncomm{\mathcal{E}} \subset \gf[[\gamma]]\cotimes\csalg{\KontHam{\mathcal{E}}} &= \gf[[\gamma]]\cotimes\csalg{\gf\oplus\KontHamP{\mathcal{E}}} \\
&= \gf[[\gamma]]\cotimes\csalg{\gf}\cotimes\csalg{\KontHamP{\mathcal{E}}} = \gf[[\gamma,\nu]]\cotimes\csalg{\KontHamP{\mathcal{E}}}.
\end{split}
\end{displaymath}

We may write any $I\in\intnoncomm{\mathcal{E}}$ as a series
\[ I = \sum_{i,j,k=0}^{\infty} \gamma^i\nu^j I_{ijk}, \quad\text{where }I_{ijk}\in\bigl[\KontHamP{\mathcal{E}}^{\cotimes k}\bigr]_{\sg{k}}. \]
Note that by the definition, $I_{ijk}$ vanishes when $j=k=0$. We may also write
\[ I_{ijk} = \sum_{l=0}^{\infty} I_{ijkl}, \]
where~$I_{ijkl}\in\bigl[\KontHamP{\mathcal{E}}^{\cotimes k}\bigr]_{\sg{k}}$ is a homogenous tensor of order $l$ in $\mathcal{E}^{\dag}$.

\begin{defi}
Let $\mathcal{E}$ be a free BV-theory. We say the $I\in\intnoncomm{\mathcal{E}}$ is an \emph{interaction} if it has degree zero and if
\begin{equation} \label{eqn_interactionNCconstraints}
I_{ijkl} = 0, \quad\text{whenever } 2(2i+j+k-1)+l<3.
\end{equation}

More generally, we will say that
\[ I\in\intnoncomm{\mathcal{E},\mathcal{A}}:=\intnoncomm{\mathcal{E}}\cotimes\mathcal{A} \]
is a \emph{family of interactions} if it has total degree zero and satisfies the constraints prescribed by \eqref{eqn_interactionNCconstraints}, where now~$I_{ijkl}\in\bigl[\KontHamP{\mathcal{E}}^{\cotimes k}\bigr]_{\sg{k}}\cotimes\mathcal{A}$ and has homogeneous order $l$ in $\mathcal{E}^{\dag}$. We will denote by $\intnoncommP{\mathcal{E},\mathcal{A}}$, the subspace of $\intnoncomm{\mathcal{E},\mathcal{A}}$ that consists of those $I$ which satisfy \eqref{eqn_interactionNCconstraints}, and the subspace consisting of all families of interactions by $\intnoncommI{\mathcal{E},\mathcal{A}}$.
\end{defi}

Before we define the Feynman amplitude for stable ribbon graphs, we begin by noting that if $X$ is a cyclically ordered set of cardinality $n$ and $\mathcal{V}$ is a $\mathbb{Z}$-graded nuclear space, then there are canonical maps
\[ \bigl[\mathcal{V}^{\cotimes n}\bigr]^{\cyc{n}}\to\mathcal{V}(\!(X)\!)\to\bigl[\mathcal{V}^{\cotimes n}\bigr]_{\cyc{n}}. \]
It follows that we have maps
\begin{equation} \label{eqn_attachmapscyc}
\begin{split}
& \xymatrix{
\KontHam{\mathcal{E}} \ar[r] & \bigl[(\mathcal{E}^{\dag})^{\cotimes n}\bigr]_{\cyc{n}} \ar[r]^{\cycsum{n}} & \bigl[(\mathcal{E}^{\dag})^{\cotimes n}\bigr]^{\cyc{n}} \ar[r] & \mathcal{E}^{\dag}(\!(X)\!)
} \\
& \xymatrix{
\mathcal{E}^{\dag}(\!(X)\!) \ar[r] & \bigl[(\mathcal{E}^{\dag})^{\cotimes n}\bigr]_{\cyc{n}} \ar[r] & \KontHam{\mathcal{E}}
}
\end{split}
\end{equation}
where $\cycsum{n}$ was defined by \eqref{eqn_cyclicsum}, and we have used the projections from and inclusions into $\KontHam{\mathcal{E}}$.

More generally, if $C$ is a cyclic decomposition of a set $X$ which consists of $k$ cycles then there are maps
\begin{align}
\label{eqn_attachmapscycdecompon}
& \xymatrix{
\bigl[\KontHamP{\mathcal{E}}^{\cotimes k}\bigr]_{\sg{k}} \ar[rr]^{\sum_{\varsigma\in\sg{k}}\varsigma} \ar[rrrd] && \bigl[\KontHamP{\mathcal{E}}^{\cotimes k}\bigr]^{\sg{k}} \ar[r] & \KontHamP{\mathcal{E}}(\!(C)\!) = \underset{c\in C}{\bigotimes}\KontHamP{\mathcal{E}} \ar[d] \\
&&& \underset{c\in C}{\bigotimes}\mathcal{E}^{\dag}(\!(c)\!) = \mathcal{E}^{\dag}(\!(X)\!)
} \\
\label{eqn_attachmapscycdecompoff}
& \xymatrix{
\mathcal{E}^{\dag}(\!(X)\!) = \underset{c\in C}{\bigotimes}\mathcal{E}^{\dag}(\!(c)\!) \ar[r] & \underset{c\in C}{\bigotimes}\KontHamP{\mathcal{E}} = \KontHamP{\mathcal{E}}(\!(C)\!) \ar[r] & \bigl[\KontHamP{\mathcal{E}}^{\cotimes k}\bigr]_{\sg{k}}
}
\end{align}
which use the maps defined by \eqref{eqn_attachmapssym} and \eqref{eqn_attachmapscyc}.

The following example provides us with a canonical way to represent our interactions.

\begin{example} \label{exm_cyclicwordproduct}
Suppose we begin with nonnegative integers $k$ and $l$ and a list $r_1,\ldots,r_k$ of $k$ positive integers whose sum is $l$. We will denote by $C_{\mathbf{r}}$, the canonical cyclic decomposition consisting of $k$ cycles on the set of integers between $1$ and $l$ defined by
\begin{equation} \label{eqn_cycdecpartition}
(1,\ldots,r_1)(r_1+1,\ldots,r_1+r_2)\ldots(l-r_k+1,\ldots,l).
\end{equation}
Given a tensor
\[ x = x_{11}\cdots x_{1r_1}x_{21}\cdots x_{2r_2}\cdots x_{k1}\cdots x_{kr_k} \in (\mathcal{E}^{\dag})^{\cotimes l} = \mathcal{E}^{\dag}(\!(\{1,\ldots,l\})\!), \]
the image of $x$ under the map \eqref{eqn_attachmapscycdecompoff} defined by the cyclic decomposition $C_{\mathbf{r}}$ is the product of cyclic words
\[ (x_{11}\cdots x_{1r_1})(x_{21}\cdots x_{2r_2})\cdots (x_{k1}\cdots x_{kr_k}). \]

Using the above, we may write for any family of interactions $I\in\intnoncomm{\mathcal{E},\mathcal{A}}$;
\begin{equation} \label{eqn_intcycrep}
I_{ijkl} = \sum_{\begin{subarray}{c} r_1,\ldots,r_k\geq 1: \\ r_1+\cdots+r_k=l \end{subarray}}\lceil I_{ijkl}^{\mathbf{r}}\rceil, \quad\text{for some } I_{ijkl}^\mathbf{r}\in(\mathcal{E}^{\dag})^{\cotimes l}\cotimes\mathcal{A};
\end{equation}
where $\lceil I_{ijkl}^{\mathbf{r}}\rceil$ denotes the image of $I_{ijkl}^{\mathbf{r}}$ under the map \eqref{eqn_attachmapscycdecompoff} defined by the cyclic decomposition $C_{\mathbf{r}}$. Note in particular that the choice of the $I_{ijkl}^{\mathbf{r}}$ are not unique.
\end{example}

\begin{defi} \label{def_FeynmanNCamplitude}
Let $\mathcal{E}$ be a free BV-theory, $I\in\intnoncommI{\mathcal{E},\mathcal{A}}$ be a family of interactions and $P\in\mathcal{E}\cotimes\mathcal{E}\cotimes\mathcal{A}$ be a family of propagators. Given a stable ribbon graph $\Gamma$, we attach the interaction $I$ to every vertex $v$ of $\Gamma$ by considering the image of $I_{g(v)b(v)|C(v)||v|}$ under the map \eqref{eqn_attachmapscycdecompon} that is defined by the cyclic decomposition $C(v)$ of the vertex $v$, and we denote this image by
\[ I(v)\in\mathcal{E}^{\dag}(\!(v)\!)\cotimes\mathcal{A}, \quad v\in V(\Gamma). \]
As before, we then define
\[ I(\Gamma) := \underset{v\in V(\Gamma)}{\otimes} I(v) \in \Hom_{\gf}\bigl(\mathcal{E}\bigl(\!\bigl(H(\Gamma)\bigr)\!\bigr),\mathcal{A}\bigl(\!\bigl(V(\Gamma)\bigr)\!\bigr)\bigr). \]

Likewise, following Equation \eqref{eqn_graphpropagator}, we attach $P$ to every edge of $\Gamma$ to form
\[ P(\Gamma) \in \mathcal{E}\Bigl(\!\Bigl(\bigcup_{e\in E(\Gamma)}e\Bigr)\!\Bigr)\cotimes\mathcal{A}\bigl(\!\bigl(E(\Gamma)\bigr)\!\bigr). \]
The Feynman amplitude
\[ F_{\Gamma}(I,P) \in \mathcal{E}^{\dag}\bigl(\!\bigl(L(\Gamma)\bigr)\!\bigr)\cotimes\mathcal{A} \]
is then defined by \eqref{eqn_Feynmanamplitude}, as before.
\end{defi}

To define the Feynman weight of a stable ribbon graph, we will need to explain how we arrive at a canonical cyclic decomposition of the legs.

\begin{defi}
We will say that a stable ribbon graph (or stable graph) is a \emph{corolla} if it has only one vertex and no edges.
\end{defi}

Since a stable ribbon graph that is a corolla is permitted to have legs, there is a canonical cyclic decomposition of the legs of such a corolla, which is determined by the cyclic decomposition of the lone vertex. More generally, contracting all the edges of a connected stable ribbon graph $\Gamma$ produces a corolla with the same set of legs as $\Gamma$; this provides $\Gamma$ with a cyclic decomposition of its legs.

\begin{defi}
Given a connected stable ribbon graph $\Gamma$, we will call the cyclic decomposition of the legs of $\Gamma$ that is obtained by contracting all of its edges the \emph{canonical decomposition}.
\end{defi}

\begin{defi}
Let $\mathcal{E}$ be a free BV-theory, $I\in\intnoncommI{\mathcal{E},\mathcal{A}}$ be a family of interactions and $P\in\mathcal{E}\cotimes\mathcal{E}\cotimes\mathcal{A}$ be a family of propagators. Given a connected stable ribbon graph $\Gamma$, we define the Feynman weight
\[ w_{\Gamma}(I,P)\in\csalg{\KontHamP{\mathcal{E}}}\cotimes\mathcal{A} \]
to be the image of the Feynman amplitude $F_{\Gamma}(I,P)$ under the map \eqref{eqn_attachmapscycdecompoff} that is determined by the canonical decomposition of the legs of $\Gamma$.
\end{defi}

\subsubsection{Amplitudes of diagrams with distinguished vertices and edges} \label{sec_specialFamplitudes}

We now consider some important variants of Feynman amplitudes that will be relevant when we come to discuss the quantum master equation in Section \ref{sec_QME}.

\begin{defi}
Given a stable ribbon graph $\Gamma$ and an edge $e$ of $\Gamma$, we call the pair $(\Gamma,e)$ a \emph{stable ribbon graph with a distinguished edge}. A morphism between two such objects consists of an isomorphism of stable ribbon graphs that respects the distinguished edges.

A \emph{stable ribbon graph with a distinguished vertex} is a pair $(\Gamma,v)$, where $\Gamma$ is a stable ribbon graph and $v$ is a vertex of $\Gamma$. A morphism between two such objects is an isomorphism of stable ribbon graphs that respects the distinguished vertex.
\end{defi}

These graphs give rise to the following corresponding notions of Feynman amplitudes.

\begin{defi}
Let $\mathcal{E}$ be a free BV-theory and $(\Gamma,e)$ be a stable ribbon graph with a distinguished edge. Suppose that we are given
\[ I\in\intnoncommI{\mathcal{E},\mathcal{A}} \quad\text{and}\quad P,K\in\mathcal{E}\cotimes\mathcal{E}\cotimes\mathcal{A}; \]
where $P$ is a propagator, and $K$ is symmetric in $\mathcal{E}$ but may have any degree.

We define the Feynman amplitude
\[ F_{(\Gamma,e)}(I;P,K)\in \mathcal{E}^{\dag}\bigl(\!\bigl(L(\Gamma)\bigr)\!\bigr)\cotimes\mathcal{A} \]
in precisely the way described by Definition \ref{def_FeynmanNCamplitude}, except that we replace $P(\Gamma)$ with
\[ (P,K)(\Gamma) := K(e)\otimes\Bigl(\underset{e'\neq e}{\otimes} P(e')\Bigr) \in \mathcal{E}\Bigl(\!\Bigl(\bigcup_{e'\in E(\Gamma)}e'\Bigr)\!\Bigr)\cotimes\mathcal{A}\bigl(\!\bigl(E(\Gamma)\bigr)\!\bigr). \]
As before, we define the Feynman weight $w_{(\Gamma,e)}(I;P,K)$ to be the image of the Feynman amplitude under the map \eqref{eqn_attachmapscycdecompoff} that is determined by the canonical decomposition of the legs of $\Gamma$.
\end{defi}

\begin{defi}
Let $\mathcal{E}$ be a free BV-theory and $(\Gamma,v)$ be a stable ribbon graph with a distinguished vertex. Suppose that we are given
\[ I\in\intnoncommI{\mathcal{E},\mathcal{A}}, \quad J\in\intnoncommP{\mathcal{E},\mathcal{A}} \quad\text{and}\quad P\in\mathcal{E}\cotimes\mathcal{E}\cotimes\mathcal{A}; \]
where $P$ is a propagator and $J$ may have any degree.

We define the Feynman amplitude
\[ F_{(\Gamma,v)}(I,J;P) \in \mathcal{E}^{\dag}\bigl(\!\bigl(L(\Gamma)\bigr)\!\bigr)\cotimes\mathcal{A} \]
in exactly the manner described by Definition \ref{def_FeynmanNCamplitude}, except that we replace $I(\Gamma)$ with
\[ (I,J)(\Gamma) := J(v)\otimes\Bigl(\underset{v'\neq v}{\otimes} I(v')\Bigr) \in \Hom_{\gf}\bigl(\mathcal{E}\bigl(\!\bigl(H(\Gamma)\bigr)\!\bigr),\mathcal{A}\bigl(\!\bigl(V(\Gamma)\bigr)\!\bigr)\bigr). \]
Again, we define the Feynman weight $w_{(\Gamma,v)}(I,J;P)$ to be the image of the Feynman amplitude under the map \eqref{eqn_attachmapscycdecompoff}.
\end{defi}

We now consider two categories. The first category is the category of connected stable ribbon graphs with a distinguished edge. The second category may be defined as follows.

\begin{defi} \label{def_2ndCat}
We may define a category whose objects consist of triples
\[ ((\Gamma,v),\iota,G); \]
where:
\begin{itemize}
\item
$(\Gamma,v)$ is a connected stable ribbon graph with a distinguished vertex,
\item
$G$ is a connected stable ribbon graph having precisely one edge and satisfying:
\[ b(G)=b(v), \quad g(G)=g(v); \]
\item
$\iota$ is a bijective map from the legs of $G$ to the incident half-edges of $v$ that respects the cyclic decompositions.
\end{itemize}
A morphism from $((\Gamma,v),\iota,G)$ to $((\Gamma',v'),\iota',G')$ is a pair of isomorphisms
\[ \phi:(\Gamma,v)\to(\Gamma',v') \quad\text{and}\quad \psi:G\to G' \]
that are compatible with the respective maps $\iota$ and $\iota'$.
\end{defi}

Given a stable ribbon graph $\Gamma$ and an edge $e$ of $\Gamma$, we may consider the stable ribbon graph $G(e)$ that consists of just this single edge $e$ and the either one or two vertices of $\Gamma$ that are joined to it. Note that the legs of $G(e)$ are the incident half-edges of the vertex $v$ of $\Gamma/e$ that is formed when we contract the edge $e$ in $\Gamma$, and this gives us a map $\iota$ satisfying the above conditions, which is just the identity map on elements; see Figure \ref{fig_graph}.

\begin{figure}[htp]
\centering
\includegraphics{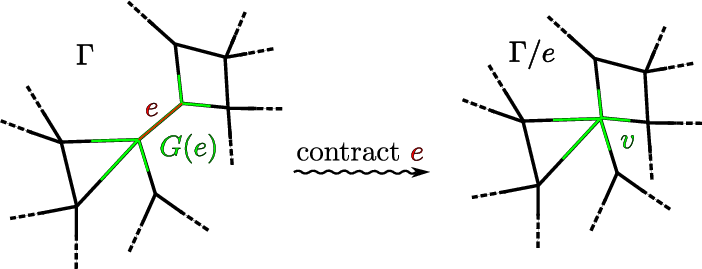}
\caption{A correspondence defining an equivalence of categories.}
\label{fig_graph}
\end{figure}

\begin{theorem} \label{thm_catequiv}
The category of connected stable ribbon graphs with a distinguished edge and the category defined by Definition \ref{def_2ndCat} are equivalent under the correspondence
\[ (\Gamma,e) \mapsto ((\Gamma/e,v),\iota,G(e)). \]
\end{theorem}

\begin{proof}
This is a special instance of Theorem 3.24 of \cite{NCRGF}. The two categories above are subcategories of the categories covered by the cited theorem that correspond under the equivalence that is defined there.
\end{proof}

Later, we will also need a definition for the Feynman amplitude corresponding to such triples.

\begin{defi}
Let $\mathcal{E}$ be a free BV-theory and $((\Gamma,v),\iota,G)$ be a triple from the category described in Definition \ref{def_2ndCat}. Suppose that we are given
\[ I\in\intnoncommI{\mathcal{E},\mathcal{A}} \quad\text{and}\quad P,K\in\mathcal{E}\cotimes\mathcal{E}\cotimes\mathcal{A}; \]
where $P$ is a propagator, and $K$ is symmetric in $\mathcal{E}$ but may have any degree.

We define the Feynman amplitude
\[ F_{((\Gamma,v),\iota,G)}(I;P,K)\in \mathcal{E}^{\dag}\bigl(\!\bigl(L(\Gamma)\bigr)\!\bigr)\cotimes\mathcal{A} \]
in precisely the way described by Definition \ref{def_FeynmanNCamplitude}, without changing $P(\Gamma)$, except that we replace $I(\Gamma)$ with
\[ (I,K)(\Gamma) := \iota^{\#}(F_G(I,K))\otimes\Bigl(\underset{v'\neq v}{\otimes} I(v')\Bigr), \]
where we have attached $K$ to the lone edge of $G$. The corresponding Feynman weight $w_{((\Gamma,v),\iota,G)}(I;P,K)$ is defined as before.
\end{defi}

\subsection{Properties of the renormalization group flow}

We are now ready to introduce the renormalization group flow and recall some of the basic results that we will need in order to eventually describe the role of the quantum master equation in noncommutative geometry and effective field theory.

\subsubsection{The renormalization group flow in commutative geometry}

We begin by recalling from \cite[\S 2.3]{CosEffThy} the definition of the renormalization group flow in commutative geometry and its most basic and fundamental properties.

\begin{defi}
Let $\mathcal{E}$ be a free BV-theory and suppose that we are given families of interactions and propagators
\[ I\in\intcommI{\mathcal{E},\mathcal{A}} \quad\text{and}\quad P\in\mathcal{E}\cotimes\mathcal{E}\cotimes\mathcal{A}. \]
We define the renormalization group flow by
\[ W(I,P) := \sum_{\Gamma} \frac{\hbar^{\ell(\Gamma)}}{|\Aut(\Gamma)|}w_{\Gamma}(I,P) \in \intcommI{\mathcal{E},\mathcal{A}}, \]
where the sum is taken over all (isomorphism classes of) connected stable graphs.
\end{defi}

The flow is continuous, and satisfies the following fundamental identities that describe the action of the group of propagators on the space of interactions; see Lemma 3.3.1 and Lemma 3.4.1 of \cite[\S 2.3]{CosEffThy}.

\begin{theorem}
Given families
\[ I\in\intcommI{\mathcal{E},\mathcal{A}} \quad\text{and}\quad P_1,P_2\in\mathcal{E}\cotimes\mathcal{E}\cotimes\mathcal{A} \]
of interactions and propagators for a free BV-theory $\mathcal{E}$ we have:
\begin{align*}
W(I,0) &= I, \\
W(I,P_1+P_2) &= W(W(I,P_1),P_2).
\end{align*}
\end{theorem}

\subsubsection{The renormalization group flow in noncommutative geometry} \label{sec_rgflownoncomm}

We now recall from \cite{NCRGF} the corresponding details for the renormalization group flow in noncommutative geometry.

\begin{defi}
Let
\[ I\in\intnoncommI{\mathcal{E},\mathcal{A}} \quad\text{and}\quad P\in\mathcal{E}\cotimes\mathcal{E}\cotimes\mathcal{A} \]
be families of interactions and propagators for a free BV-theory $\mathcal{E}$. We define
\begin{equation} \label{eqn_RGFnoncomm}
W(I,P) := \sum_{\Gamma} \frac{\gamma^{g(\Gamma)}\nu^{b(\Gamma)}}{|\Aut(\Gamma)|}w_{\Gamma}(I,P) \in \intnoncommI{\mathcal{E},\mathcal{A}},
\end{equation}
where now the sum is taken over all connected stable \emph{ribbon} graphs.
\end{defi}

The following is Theorem 3.30 of \cite{NCRGF}.

\begin{theorem} \label{thm_intgrpact}
Let $\mathcal{E}$ be a free BV-theory. Given families of interactions and propagators
\[ I\in\intnoncommI{\mathcal{E},\mathcal{A}} \quad\text{and}\quad P_1,P_2\in\mathcal{E}\cotimes\mathcal{E}\cotimes\mathcal{A} \]
we have:
\begin{align}
\label{eqn_intgrpactzero}
W(I,0) &= I, \\
\label{eqn_intgrpactsum}
W(I,P_1+P_2) &= W(W(I,P_1),P_2).
\end{align}
\end{theorem}

\subsubsection{The tree-level flow}

The description of the renormalization group flow at the tree level is significantly simpler. We will see later in Section \ref{sec_CMEtreelevelflow} that it describes how solutions to the classical master equation transform and hence will be important when we come to talk about the process of quantization in the BV-formalism.

\begin{defi} \label{def_tree}
We will say that a stable ribbon graph (or stable graph) is a \emph{tree} if it is connected and has loop number zero.
\end{defi}

Trees are instances of ordinary ribbon graphs, cf. \cite[\S 5.5]{OpAlgTopPhys}; though the converse is not true.

\begin{defi}
Let $\mathcal{E}$ be a free BV-theory. We will say that $I\in\intnoncommI{\mathcal{E},\mathcal{A}}$ is a \emph{tree-level interaction} if
\[ I_{ijk}=0, \quad\text{whenever }(i,j,k)\neq(0,0,1). \]
The subspace of $\intnoncommI{\mathcal{E},\mathcal{A}}$ consisting of all tree-level interactions will be denoted by~$\intnoncommItree{\mathcal{E},\mathcal{A}}$.
\end{defi}

Consider the projection
\[ \intnoncommI{\mathcal{E},\mathcal{A}}\to\intnoncommItree{\mathcal{E},\mathcal{A}}, \qquad I\mapsto I_{001}. \]
The following is Proposition 3.36 of \cite{NCRGF}.

\begin{prop} \label{prop_treeflow}
Let $P\in\mathcal{E}\cotimes\mathcal{E}\cotimes\mathcal{A}$ be a family of propagators for a free BV-theory $\mathcal{E}$. Define $W^{\mathrm{Tree}}(I,P)$ for a tree-level interaction by the same formula as \eqref{eqn_RGFnoncomm}, except that now the sum is taken over all \emph{trees}. Then the following diagram commutes:
\begin{displaymath}
\xymatrix{
\intnoncommI{\mathcal{E},\mathcal{A}} \ar[d] \ar[rr]^{W(-,P)} && \intnoncommI{\mathcal{E},\mathcal{A}} \ar[d] \\ \intnoncommItree{\mathcal{E},\mathcal{A}} \ar[rr]^{W^{\mathrm{Tree}}(-,P)} && \intnoncommItree{\mathcal{E},\mathcal{A}}
}
\end{displaymath}
\end{prop}

\subsubsection{Transformation properties of the flow}

The definition of the renormalization group flow in noncommutative geometry has two important properties. The first is that it lifts the renormalization group flow defined in commutative geometry. The second is that it is compatible with the notion of a two-dimensional Open Topological Field Theory (OTFT), which are of course more commonly known as Frobenius algebras \cite{AtiyahTFT}. Here we will recall from \cite{NCRGF} these two basic facts, beginning with the first.

\begin{defi} \label{def_mapNCtoCom}
Let $\mathcal{E}$ be a free BV-theory and consider the natural quotient map
\[ \sigma:\KontHamP{\mathcal{E}}\to\csalgPdelim{(\mathcal{E}^{\dag})}\subset\intcomm{\mathcal{E}} \]
in which we pass from $\cyc{n}$ to $\sg{n}$-coinvariants. Extend this to a map
\[ \sigma_{\gamma,\nu}:\intnoncomm{\mathcal{E}}\to\intcomm{\mathcal{E}} \]
by defining
\[ \sigma_{\gamma,\nu}(\gamma^i\nu^j x_1\cdots x_k) = \hbar^{2i+j+k-1}\sigma(x_1)\cdots\sigma(x_k), \quad x_1,\ldots,x_k\in\KontHamP{\mathcal{E}}. \]
\end{defi}

The following is Theorem 3.38 of \cite{NCRGF}.

\begin{theorem} \label{thm_flowNCtoCom}
Let $\mathcal{E}$ be a free BV-theory and $P\in\mathcal{E}\cotimes\mathcal{E}\cotimes\mathcal{A}$ be a family of propagators. The map $\sigma_{\gamma,\nu}$ intertwines the commutative and noncommutative renormalization group flows; that is, the following diagram commutes:
\begin{displaymath}
\xymatrix{
\intnoncommI{\mathcal{E},\mathcal{A}} \ar[rr]^{W(-,P)} \ar[d]_{\sigma_{\gamma,\nu}} && \intnoncommI{\mathcal{E},\mathcal{A}} \ar[d]^{\sigma_{\gamma,\nu}} \\
\intcommI{\mathcal{E},\mathcal{A}} \ar[rr]^{W(-,P)} && \intcommI{\mathcal{E},\mathcal{A}}
}
\end{displaymath}
\end{theorem}

Next we address the second property. Suppose that $\textgoth{A}$ is a (unital) $\mathbb{Z}$-graded Frobenius algebra, which means that it has a degree zero nondegenerate trace map $\Tr$ for which there is a unique symmetric degree zero tensor
\[ \innprod_{\textgoth{A}}^{-1} = \sum_i x_i\otimes y_i \in \textgoth{A}\otimes \textgoth{A} \]
satisfying
\[ a = \sum_i \Tr(ax_i)y_i, \quad\text{for all }a\in \textgoth{A}. \]

We may define a family of multilinear maps from our Frobenius algebra by considering the maps that are assigned to standard surfaces by the OTFT that is constructed from $\textgoth{A}$. Here we will be as brief as possible, simply providing formulas; but more details are available from \cite{NCRGF} as well as \cite[\S 3]{ChLaOTFT}.

\begin{defi}
Let $\textgoth{A}$ be a Frobenius algebra and consider the degree zero central elements
\[ \Xi_{\mathrm{bdry}} := \sum_i x_i y_i \quad\text{and}\quad \Xi_{\mathrm{gen}} := \sum_{i,j} (-1)^{|x_j||y_i|}x_ix_jy_iy_j. \]
These are used to define multilinear maps
\[
\OTFT{g,b}{r_1,\ldots,r_k}:\textgoth{A}^{\otimes r_1}\otimes\cdots\otimes \textgoth{A}^{\otimes r_k}\to\gf; \quad g,b\geq 0; r_1,\ldots,r_k\geq 1;
\]
according to the formula
\begin{multline*}
\OTFT{g,b}{r_1,\ldots,r_k}(a_{11},\ldots,a_{1r_1};\ldots;a_{k1},\ldots, a_{kr_k}) := \\
\sum_{i_1,\ldots,i_k} (-1)^p \Tr\bigl(x_{i_k}\cdots x_{i_1}\bigr)\Tr\bigl(\Xi_{\mathrm{bdry}}^b\Xi_{\mathrm{gen}}^g y_{i_1}a_{11}\cdots a_{1r_1}\cdots y_{i_k}a_{k1}\cdots a_{kr_k}\bigr),
\end{multline*}
where the sign given by the Koszul sign rule is $p:=\sum_{1\leq t<s\leq k}\sum_{j=1}^{r_t}|y_{i_s}||a_{tj}|$.
\end{defi}

An important example, with implications for the large $N$ correspondence, occurs when we take $\textgoth{A}$ to be a matrix algebra.

\begin{example} \label{exm_matrixOTFT}
Consider the Frobenius algebra of $N$-by-$N$ matrices $\textgoth{A}:=\mat{N}{\gf}$. The OTFT maps are
\[ \OTFT{g,b}{r_1,\ldots,r_k}(A_{11},\ldots,A_{1r_1};\ldots;A_{k1},\ldots,A_{kr_k}) = N^b\Tr(A_{11}\cdots A_{1r_1})\cdots\Tr(A_{k1}\cdots A_{kr_k}). \]
Details may be found in \cite[\S 4.5]{GiGqHaZeLQT}.
\end{example}

The following construction for a free theory is similar to Example \ref{exm_Frobeniusbundle}.

\begin{example} \label{exm_OTFTfreethy}
Let $\mathcal{E}=\Gamma(M,E)$ be a free BV-theory and $\textgoth{A}$ be a differential graded Frobenius algebra and consider the tensor product $E\otimes \textgoth{A}$ of the vector bundle $E$ with the trivial bundle $\textgoth{A}$. Denote the sections by
\[ \mathcal{E}_{\textgoth{A}}:=\Gamma(M,E\otimes \textgoth{A}) = \mathcal{E}\otimes \textgoth{A}. \]
We may define a local pairing on this bundle by
\[ \innprodloc[v_1\otimes a_1,v_2\otimes a_2]^{\textgoth{A}} := (-1)^{|a_1||v_2|}\innprodloc[v_1,v_2]\Tr(a_1 a_2).  \]
We may give $\mathcal{E}_{\textgoth{A}}$ the structure of a free BV-theory by defining the corresponding differential operator by
\[ Q_{\textgoth{A}} := (Q\otimes\mathds{1}+\mathds{1}\otimes d_{\textgoth{A}}):\mathcal{E}\otimes\textgoth{A}\to\mathcal{E}\otimes\textgoth{A}. \]

Now suppose that
\[ Q^{\mathrm{GF}}:\mathcal{E}\to\mathcal{E}\cotimes\mathcal{A} \]
is a family of gauge-fixing operators for $\mathcal{E}$. We may define a family of gauge-fixing operators for $\mathcal{E}_{\textgoth{A}}$ by
\[ Q^{\mathrm{GF}}_{\textgoth{A}} := (\mathds{1}\cotimes\tau)(Q^{\mathrm{GF}}\cotimes\mathds{1}):\mathcal{E}_{\textgoth{A}}\to\mathcal{E}_{\textgoth{A}}\cotimes\mathcal{A}. \]
The heat kernel for this family of gauge-fixing operators will be
\begin{equation} \label{eqn_Frobeniusheatkernel}
K_{\textgoth{A}} = K\otimes\innprod_{\textgoth{A}}^{-1} \in \smooth{0,\infty}\underset{\mathbb{R}}{\cotimes}\mathcal{E}_{\textgoth{A}}\cotimes\mathcal{E}_{\textgoth{A}}\cotimes\mathcal{A},
\end{equation}
where $K$ is the heat kernel for the family $Q^{\mathrm{GF}}$. Consequently, if $P_{\textgoth{A}}$ denotes the family of propagators on $\mathcal{E}_{\textgoth{A}}$ defined by Equation \eqref{eqn_canonicalpropagator} for the family $Q^{\mathrm{GF}}_{\textgoth{A}}$, then
\begin{equation} \label{eqn_OTFTpropagator}
P_{\textgoth{A}} = P\otimes\innprod_{\textgoth{A}}^{-1} \in \mathcal{E}_{\textgoth{A}}\cotimes\mathcal{E}_{\textgoth{A}}\cotimes\mathcal{A}, \end{equation}
where $P$ is the family of propagators associated by \eqref{eqn_canonicalpropagator} to the family $Q^{\mathrm{GF}}$.

We will be particularly interested in the case when $\mathcal{E}$ is the de Rham algebra and $\textgoth{A}$ is a matrix algebra, which will lead us to field theories based on spaces of connections.
\end{example}

\begin{defi} \label{def_OTFTtransformation}
Suppose that $\mathcal{E}$ is a free BV-theory and $\textgoth{A}$ is a Frobenius algebra and recall from Example \ref{exm_cyclicwordproduct} that for any $I\in\intnoncomm{\mathcal{E}}$ we may represent
\[ I_{ijkl} = \sum_{\mathbf{r}}\lceil I_{ijkl}^{\mathbf{r}}\rceil \]
for some (nonunique) choice of distributions $I_{ijkl}^\mathbf{r}\in(\mathcal{E}^{\dag})^{\cotimes l}$.

We use the above representation to define a map
\[ \Morita:\intnoncomm{\mathcal{E}}\longrightarrow\intnoncomm{\mathcal{E}_{\textgoth{A}}}, \qquad I=\sum_{i,j,k,l,\mathbf{r}}\gamma^i\nu^j\lceil I_{ijkl}^{\mathbf{r}}\rceil \longmapsto \sum_{i,j,k,l,\mathbf{r}}\gamma^i\nu^j\lceil I_{ijkl}^{\mathbf{r}}\otimes\OTFT{i,j}{\mathbf{r}}\rceil; \]
where
\[ I_{ijkl}^{\mathbf{r}}\otimes\OTFT{i,j}{\mathbf{r}}\in(\mathcal{E}^{\dag})^{\cotimes l}\otimes(\textgoth{A}^{\dag})^{\otimes l} = (\mathcal{E}_{\textgoth{A}}^{\dag})^{\cotimes l} \]
and $\lceil I_{ijkl}^{\mathbf{r}}\otimes\OTFT{i,j}{\mathbf{r}}\rceil$ denotes, as before, its image in $\intnoncomm{\mathcal{E}_{\textgoth{A}}}$ under the map \eqref{eqn_attachmapscycdecompoff} determined by the cyclic decomposition $C_{\mathbf{r}}$ defined by \eqref{eqn_cycdecpartition}.
\end{defi}

The symmetry properties for an OTFT ensure that the map $\Morita$ is well-defined, see \cite[\S 3.3.1]{NCRGF}. Theorem 3.44 of \cite{NCRGF} states that this map is compatible with the renormalization group flow.

\begin{theorem} \label{thm_flowOTFT}
Let $\mathcal{E}$ be a free BV-theory and $\textgoth{A}$ be a Frobenius algebra. Given a family of propagators $P\in\mathcal{E}\cotimes\mathcal{E}\cotimes\mathcal{A}$, consider the family of propagators $P_{\textgoth{A}}$ for $\mathcal{E}_{\textgoth{A}}$ that is defined by Equation \eqref{eqn_OTFTpropagator}. Then the following diagram commutes:
\begin{displaymath}
\xymatrix{
\intnoncommI{\mathcal{E},\mathcal{A}} \ar[rr]^{W(-,P)} \ar[d]_{\Morita} && \intnoncommI{\mathcal{E},\mathcal{A}} \ar[d]^{\Morita} \\
\intnoncommI{\mathcal{E}_{\textgoth{A}},\mathcal{A}} \ar[rr]^{W(-,P_{\textgoth{A}})} && \intnoncommI{\mathcal{E}_{\textgoth{A}},\mathcal{A}}
}
\end{displaymath}
\end{theorem}

\subsection{Pretheories}

In this section we will use the renormalization group flow to introduce the definition of a pretheory. In \cite{NCRGF} these were referred to as effective field theories; but as we have explained, as in \cite{CosEffThy}, we wish to reserve this nomenclature for theories satisfying the quantum master equation.

\subsubsection{Local functionals}

The definition of a theory will contain a locality requirement based on Definition \ref{def_localdistribution}.

\begin{defi}
Given a free BV-theory $\mathcal{E}$, we will say that a functional $I\in\intcomm{\mathcal{E},\mathcal{A}}$ is \emph{local} if for all $i,j\geq 0$,
\[ I_{ij}\in\bigl[(\mathcal{E}^{\dag})^{\cotimes j}\bigr]_{\sg{j}}\cotimes\mathcal{A} \]
may be represented by a local $\mathcal{A}$-valued distribution. We will denote the subspace of all local functionals by $\intcommL{\mathcal{E},\mathcal{A}}$ and the subspace of all local interactions by $\intcommLI{\mathcal{E},\mathcal{A}}$.

Likewise, in the noncommutative case we will say that a functional $I\in\intnoncomm{\mathcal{E},\mathcal{A}}$ is \emph{local} if for every $i,j,k,l\geq 0$, there is a representation
\[ I_{ijkl} = \sum_{\mathbf{r}}\lceil I_{ijkl}^{\mathbf{r}}\rceil, \quad I_{ijkl}^\mathbf{r}\in(\mathcal{E}^{\dag})^{\cotimes l}\cotimes\mathcal{A} \]
of the form \eqref{eqn_intcycrep} in which every $I_{ijkl}^{\mathbf{r}}$ is a local $\mathcal{A}$-valued distribution. Denote by $\intnoncommL{\mathcal{E},\mathcal{A}}$ the subspace consisting of all local functionals, and by $\intnoncommLI{\mathcal{E},\mathcal{A}}$ the subspace consisting of all local interactions.
\end{defi}

Note that the map $\sigma_{\gamma,\nu}$ of Definition \ref{def_mapNCtoCom} and the map $\Morita$ defined by Definition \ref{def_OTFTtransformation} both send local functionals to local functionals.

\subsubsection{Renormalization and counterterms}

Theories are constructed using counterterms through the process of renormalization. The first definition we require is that of a renormalization scheme.

\begin{defi}
Denote by $\smoothlimit{0,1}$, the subspace of $\smooth{0,1}$ consisting of all those functions that admit a limit as $\varepsilon\to 0$. A choice of complimentary subspace
\[ \smooth{0,1} = \smoothlimit{0,1}\oplus\smoothsing{0,1}. \]
will be called a \emph{renormalization scheme}.
\end{defi}

By Theorem 13.4.3 of \cite[\S 2.13.4]{CosEffThy}, there is a series of local counterterms for any local interaction in $\intcomm{\mathcal{E},\mathcal{A}}$. Theorem 4.9 of \cite{NCRGF} below deals with the noncommutative case.

\begin{theorem}
Let $I\in\intnoncommLI{\mathcal{E},\mathcal{A}}$ be a family of local interactions for a free BV-theory $\mathcal{E}$. Given a family of gauge-fixing operators $Q^{\mathrm{GF}}$, let $P(\varepsilon,L)$ be the propagator defined by Equation \eqref{eqn_canonicalpropagator} of Definition \ref{def_canonicalpropagator}. Then there is a unique series of purely singular local counterterms
\[ I^{\mathrm{CT}}\in\intnoncommI{\mathcal{E},\mathcal{A}}\underset{\mathbb{R}}{\cotimes}\smooth{0,1} \]
such that:
\begin{itemize}
\item
for all $i,j,k,l\geq 0$;
\[ I^{\mathrm{CT}}_{ijkl}\in\intnoncommL{\mathcal{E},\mathcal{A}}\underset{\mathbb{R}}{\otimes}\smoothsing{0,1}, \]
\item
and for all $L>0$, the limit
\[ \lim_{\varepsilon\to 0}W\bigl(I-I^{\mathrm{CT}}(\varepsilon),P(\varepsilon,L)\bigr) \]
exists, where convergence takes place in $\intnoncommI{\mathcal{E},\mathcal{A}}$. In the above expression, $I^{\mathrm{CT}}(\varepsilon)\in\intnoncommI{\mathcal{E},\mathcal{A}}$ denotes the value of $I^{\mathrm{CT}}$ at the point $\varepsilon\in (0,1)$.
\end{itemize}
\end{theorem}

\begin{rem}
We emphasize that in the above, the use of $\otimes$ denotes the ordinary algebraic tensor product in which only finite sums are allowed.
\end{rem}

A very basic result is that no counterterms are required for trees, see Lemma 4.13 of \cite{NCRGF}.

\begin{lemma} \label{lem_treecterterm}
Suppose that $I\in\intnoncommI{\mathcal{E},\mathcal{A}}$ is a family of interactions for a free BV-theory $\mathcal{E}$ for which $I_{001}$ is a local interaction, and let $P(\varepsilon,L)$ be a family of propagators of the form \eqref{eqn_canonicalpropagator} defined from a family of gauge-fixing operators $Q^{\mathrm{GF}}$.
\begin{enumerate}
\item \label{itm_treecterterm1}
Let $\Gamma$ be a connected stable ribbon graph such that:
\begin{itemize}
\item
the first Betti number $\mathbf{b}_1(\Gamma)$ vanishes, and
\item
the loop number of every vertex of $\Gamma$, except possibly one, vanishes.
\end{itemize}
Then for all $L>0$, the limit
\[ \lim_{\varepsilon\to 0} F_{\Gamma}(I,P(\varepsilon,L)) \]
exists in $\Hom_{\gf}(\mathcal{E}(\!(L(\Gamma))\!),\mathcal{A})$. Moreover, it then converges to zero as $L\to 0$, unless $\Gamma$ is a corolla.
\item \label{itm_treecterterm2}
The tree-level counterterms vanish,
\[ I_{001}^{\mathrm{CT}} = 0. \]
\end{enumerate}
\end{lemma}

\subsubsection{The structure of a pretheory}

We begin with the definition of a pretheory in noncommutative geometry, which follows Definition 8.4.1 of a pretheory from \cite[\S 5.8.4]{CosEffThy} by replacing the space of interactions $\intcomm{\mathcal{E}}$ with $\intnoncomm{\mathcal{E}}$.

\begin{defi} \label{def_NCprethy}
Let $\mathcal{E}$ be a free BV-theory and $Q^{\mathrm{GF}}:\mathcal{E}\to\mathcal{E}\cotimes\mathcal{A}$ be a family of gauge-fixing operators. Let $P(\varepsilon,L)$ be the family of propagators that is defined from $Q^{\mathrm{GF}}$ by Equation \eqref{eqn_canonicalpropagator}.

Suppose that
\[ I\in\intnoncommI{\mathcal{E},\mathcal{A}}\underset{\mathbb{R}}{\cotimes}\smooth{0,\infty} \]
is an $L$-parameterized family of interactions, with $I[L]\in\intnoncommI{\mathcal{E},\mathcal{A}}$ denoting the interaction defined at the point $L>0$. We say that $I$ defines a (noncommutative) \emph{pretheory} if:
\begin{itemize}
\item
the renormalization group equation
\begin{equation} \label{eqn_RGequation}
I[L'] = W(I[L],P(L,L'))
\end{equation}
holds for all $L,L'>0$ and
\item
the interactions are asymptotically local; that is for all $i,j,k,l\geq 0$ there is a small $L$ asymptotic expansion
\[ I_{ijkl}[L]\simeq\sum_{r=1}^{\infty}f_r(L)\Phi_r, \]
where $f_r\in\smooth{0,\infty}$ and $\Phi_r\in\intnoncommLI{\mathcal{E},\mathcal{A}}$ are local interactions.
\end{itemize}
We will denote the set of all pretheories by
\[ \NCPreThy{\mathcal{E},\mathcal{A}}\subset\intnoncommI{\mathcal{E},\mathcal{A}}\underset{\mathbb{R}}{\cotimes}\smooth{0,\infty}. \]
When $\mathcal{A}=\gf$, we will denote this simply by $\NCPreThy{\mathcal{E}}$.
\end{defi}

\begin{rem} \label{rem_asymplocal}
By the term asymptotic expansion above, we mean that there is a monotonically increasing sequence of integers $d_R$, tending to infinity, such that for all $R$
\[ \lim_{L\to 0}L^{-d_R}\biggl(I_{ijkl}[L]-\sum_{r=1}^R f_r(L)\Phi_r\biggr) = 0. \]
\end{rem}

\begin{rem} \label{rem_prethy}
The notion of a (commutative) pretheory, introduced in \cite[\S 5.8.4]{CosEffThy}, is defined as above, except that we replace $\intnoncomm{\mathcal{E},\mathcal{A}}$ everywhere with $\intcomm{\mathcal{E},\mathcal{A}}$ and $I_{ijkl}[L]$ with $I_{ij}[L]$. We will denote the set of all such (commutative) pretheories by $\PreThy{\mathcal{E},\mathcal{A}}$
\end{rem}

\begin{rem} \label{rem_GFdependence}
The definition of a pretheory depends upon the family of gauge-fixing operators $Q^{\mathrm{GF}}$. We will denote the set of all pretheories by
\[ \NCPreThy{\mathcal{E},\mathcal{A};Q^{\mathrm{GF}}} \quad\text{and}\quad \PreThy{\mathcal{E},\mathcal{A};Q^{\mathrm{GF}}} \]
when we want to make this dependence explicit. When $\mathcal{A}=\gf$, we will simply omit $\mathcal{A}$ from the notation.
\end{rem}

More generally, we will need the notion of a level $p$ pretheory, which is provided through the use of a filtration based on the loop number.

\begin{defi} \label{def_NCfiltration}
Let $\mathcal{E}$ be a free BV-theory. Given $I\in\intnoncomm{\mathcal{E},\mathcal{A}}$ define
\[ I_{[n]} := \sum_{\begin{subarray}{c} i,j,k\geq 0: \\ 2i+j+k-1 = n\end{subarray}} \gamma^i\nu^j I_{ijk}. \]
A complete Hausdorff decreasing filtration on $\intnoncomm{\mathcal{E},\mathcal{A}}$ is defined by taking $F_p\intnoncomm{\mathcal{E},\mathcal{A}}$ to be the subspace of $\intnoncomm{\mathcal{E},\mathcal{A}}$ that consists of all those functionals $I$ satisfying
\[ I_{[n]}=0, \quad\text{for all } 0\leq n<p. \]
\end{defi}

Note that the tree-level part of an interaction $I\in\intnoncommI{\mathcal{E},\mathcal{A}}$ is given by $I_{[0]}=I_{001}$. Hence the space of tree-level interactions may  be described in terms of the filtration by
\[ \intnoncommItree{\mathcal{E},\mathcal{A}} = \intnoncommI{\mathcal{E},\mathcal{A}}/F_1 \intnoncommI{\mathcal{E},\mathcal{A}}. \]

The renormalization group flow is well-defined modulo $F_p$, which allows us to introduce the definition of a level $p$ pretheory; see Proposition 4.25 of \cite{NCRGF}, whose results we recall in Proposition \ref{prop_filterformula} at the end of this section.

\begin{defi}
Let $\mathcal{E}$ be a free BV-theory and $Q^{\mathrm{GF}}:\mathcal{E}\to\mathcal{E}\cotimes\mathcal{A}$ be a family of gauge-fixing operators. Let $P(\varepsilon,L)$ be the family of propagators that is defined from $Q^{\mathrm{GF}}$ by Equation \eqref{eqn_canonicalpropagator}.

Given $p\geq 0$ we say that
\[ I\in\bigl(\intnoncommI{\mathcal{E},\mathcal{A}}/F_{p+1}\intnoncommI{\mathcal{E},\mathcal{A}}\bigr)\underset{\mathbb{R}}{\cotimes}\smooth{0,\infty} \]
is a \emph{level $p$ (noncommutative) pretheory} if:
\begin{itemize}
\item
the renormalization group equation \eqref{eqn_RGequation} holds, and
\item
each $I_{ijkl}[L]$ is asymptotically local as $L\to 0$.
\end{itemize}

We denote the set of level $p$ pretheories by
\[ \NCPreThyL{p}{\mathcal{E},\mathcal{A}}\subset\bigl(\intnoncommI{\mathcal{E},\mathcal{A}}/F_{p+1}\intnoncommI{\mathcal{E},\mathcal{A}}\bigr)\underset{\mathbb{R}}{\cotimes}\smooth{0,\infty}. \]
A level zero pretheory will be called a \emph{tree-level pretheory}.
\end{defi}

\begin{rem}
As in Remark \ref{rem_prethy}, one can similarly define the notion of a (commutative) level $p$ pretheory, though we will not require it for our purposes in this paper.
\end{rem}

As in \cite{CosEffThy}, pretheories are produced from local interactions through renormalization.

\begin{defi} \label{def_intrenormalized}
Let $\mathcal{E}$ be a free BV-theory and $P(\varepsilon,L)$ be the family of propagators defined by Equation \eqref{eqn_canonicalpropagator} for a family of gauge-fixing operators $Q^{\mathrm{GF}}$. Given a local interaction $I\in\intnoncommLI{\mathcal{E},\mathcal{A}}$, we define a renormalized interaction
\[ I^\mathrm{R}[L] :=  \lim_{\varepsilon\to 0}W\bigl(I-I^{\mathrm{CT}}(\varepsilon),P(\varepsilon,L)\bigr), \quad L>0. \]
\end{defi}

The following is Theorem 4.35 of \cite{NCRGF}. It is the noncommutative analogue of one of the main theorems of \cite[\S 2.13.4]{CosEffThy}, Theorem 13.4.3.

\begin{theorem} \label{thm_locinteffthy}
Let $\mathcal{E}$ be a free BV-theory and $Q^{\mathrm{GF}}:\mathcal{E}\to\mathcal{E}\cotimes\mathcal{A}$ be a family of gauge-fixing operators.
\begin{enumerate}
\item
There is a one-to-one correspondence between local interactions and pretheories;
\[ \intnoncommLI{\mathcal{E},\mathcal{A}}\longrightarrow\NCPreThy{\mathcal{E},\mathcal{A}}, \qquad I\mapsto I^{\mathrm{R}}. \]
\item
There is a one-to-one correspondence between local interactions modulo $F_{p+1}$ and level $p$ pretheories;
\[ \intnoncommLI{\mathcal{E},\mathcal{A}}/F_{p+1}\intnoncommLI{\mathcal{E},\mathcal{A}}\longrightarrow\NCPreThyL{p}{\mathcal{E},\mathcal{A}}, \qquad I\mapsto I^{\mathrm{R}}. \]
\end{enumerate}
\end{theorem}

From the preceding theorem, we clearly see that the fiber of level $(p+1)$ pretheories that sit over a given level $p$ pretheory is acted upon freely and transitively by the group
\begin{equation} \label{eqn_liftgroup}
\mathcal{G}_{p+1} := F_{p+1}\intnoncommLI{\mathcal{E},\mathcal{A}}/F_{p+2}\intnoncommLI{\mathcal{E},\mathcal{A}}.
\end{equation}
To describe the action of this group on the fiber, we introduce the following terminology.

\begin{defi}
We will say that a stable ribbon graph (or stable graph) $\Gamma$ is a \emph{$p$-tree} if:
\begin{itemize}
\item
it is connected,
\item
the first Betti number $\mathbf{b}_1(\Gamma)$ vanishes, and
\item
the loop number of every vertex of $\Gamma$ vanishes, except one, which must have loop number $p$.
\end{itemize}
We will denote the class of $p$-trees by $\ptree$. Note that a $0$-tree is the same thing as a regular tree, as defined by Definition \ref{def_tree}.
\end{defi}

The action of $\mathcal{G}_{p+1}$ on the fiber is described by Lemma 4.34 of \cite{NCRGF}.

\begin{prop} \label{prop_prethyliftgrpact}
Let $\mathcal{E}$ be a free BV-theory, $Q^{\mathrm{GF}}:\mathcal{E}\to\mathcal{E}\cotimes\mathcal{A}$ be a family of gauge-fixing operators and $p$ be a nonnegative integer.

For all $I\in\intnoncommLI{\mathcal{E},\mathcal{A}}/F_{p+2}\intnoncommLI{\mathcal{E},\mathcal{A}}$ and $J\in\mathcal{G}_{p+1}$,
\begin{displaymath}
(I+J)^{\mathrm{R}}[L] = I^{\mathrm{R}}[L]  + \sum_{\Gamma\in\ptree[p+1]}\frac{\gamma^{g(\Gamma)}\nu^{b(\Gamma)}}{|\Aut(\Gamma)|} \lim_{\varepsilon\to 0}w_{\Gamma}\bigl(I_{[0]}+J,P(\varepsilon,L)\bigr),
\end{displaymath}
where the sum is taken over all $(p+1)$-trees $\Gamma$.
\end{prop}

In Section \ref{sec_obsdef} we will also require the following formula from Proposition 4.25 of \cite{NCRGF}.

\begin{prop} \label{prop_filterformula}
Let $\mathcal{E}$ be a free BV-theory and $P\in\mathcal{E}\cotimes\mathcal{E}\cotimes\mathcal{A}$ be a family of propagators. For every $I\in\intnoncommI{\mathcal{E},\mathcal{A}}$ and $J\in F_p\intnoncommI{\mathcal{E},\mathcal{A}}$ we have, providing that $p>0$, the formula
\[ W(I+J,P) = W(I,P) + \sum_{\Gamma\in\ptree} \frac{\gamma^{g(\Gamma)}\nu^{b(\Gamma)}}{|\Aut(\Gamma)|}w_{\Gamma}\bigl(I_{[0]}+J_{[p]},P\bigr) \mod F_{p+1}. \]
\end{prop}

\section{Quantum BV-geometry} \label{sec_QuantumBV}

We are now ready to begin to describe the geometry of the quantum master equation in our framework of noncommutative geometry. Such an extension of the work of Kontsevich~\cite{KontSympGeom} on noncommutative symplectic geometry was considered by Barannikov in~\cite{baran}. While Kontsevich's original work was concerned with the classical master equation, Barannikov's provides a description of the quantum master equation in terms of modular operads \cite{GetKap}.

The situation presented here has some significant differences from those addressed in the above cited references. Most significantly, our space $\mathcal{E}$ is infinite-dimensional, and hence we will need to apply Costello's technique \cite{CosBVrenormalization} of renormalization in order to make sense of the quantum master equation. Additionally, and importantly, our space $\intnoncomm{\mathcal{E}}$ contains a parameter $\nu$ which, as we will see later in Section \ref{sec_largeN} and Section \ref{sec_powerseries}, plays a fundamental role in the description of large $N$ limits of gauge theories; cf. \cite{GiGqHaZeLQT}.

In what follows, $\mathcal{A}$ will be the de Rham algebra on some smooth manifold $X$ with corners.

\subsection{Batalin-Vilkovisky algebras}

Our first task will be to provide $\intnoncomm{\mathcal{E}}$ with the structure of a Batalin-Vilkovisky (BV) algebra. This will be built from a Lie bialgebra structure on $\KontHam{\mathcal{E}}$ that was first discovered by Movshev in \cite{movshevcobracket}.

\subsubsection{Lie algebra and bialgebra structures} \label{sec_Liebialgebra}

The Lie bracket on $\KontHam{\mathcal{E},\mathcal{A}}$ was first described by Kontsevich in \cite{KontSympGeom} and is directly related to the Lie bracket defined by Equation \eqref{eqn_Hambracket} in Section \ref{sec_SympandHam}.

\begin{defi}
Let $\mathcal{E}$ be a free BV-theory and suppose that we are given a $K\in\mathcal{E}\cotimes\mathcal{E}\cotimes\mathcal{A}$ of degree one which is symmetric in $\mathcal{E}$. Given
\[ f=f_1\cdots f_m\cdot a\in(\mathcal{E}^{\dag})^{\otimes m}\otimes\mathcal{A} \quad\text{and}\quad g=g_1\cdots g_n\cdot a'\in(\mathcal{E}^{\dag})^{\otimes n}\otimes\mathcal{A} \]
we define
\[ \{f,g\}_K := \sum_{i=1}^m\sum_{j=1}^n (-1)^{p_{ij}} f_{i+1}\cdots f_m\cdot f_1\cdots f_{i-1}\cdot g_{j+1}\cdots g_n\cdot g_1\cdots g_{j-1}\cdot \bigl((g_j\otimes f_i)[K] a a'\bigr), \]
where
\[ p_{ij} := \bigl(|f_i|+|a|+1\bigr)\sum_{r=1}^n |g_i| + \sum_{r=1}^m |f_i| + \sum_{1\leq r\leq i <s\leq m}|f_r||f_s| + \sum_{1\leq r\leq j <s\leq n}|g_r||g_s| \]
is the sign given by the Koszul sign rule.

The above formula gives a well-defined Lie bracket of degree one
\begin{equation} \label{eqn_bracket}
\{-,-\}_K: \KontHam{\mathcal{E},\mathcal{A}}\cotimes\KontHam{\mathcal{E},\mathcal{A}}\to\KontHam{\mathcal{E},\mathcal{A}},
\end{equation}
as well as a Lie bracket on $\intcomm{\mathcal{E},\mathcal{A}}$ defined by the same formula and also denoted by $\{-,-\}_K$.
\end{defi}

As the input for the above bracket, we may use the heat kernel defined from a family of gauge-fixing conditions. In the limit, we recover the bracket defined by Equation \eqref{eqn_Hambracket}.

\begin{lemma} \label{lem_bracketlimit}
Let $\mathcal{E}$ be a free BV-theory with a family of gauge-fixing operators $Q^{\mathrm{GF}}$, and let $K_t$ be the heat kernel corresponding to this family. Define
\[ \{f,h\}_L := -\{f,h\}_{K_L} \]
for $f,h\in\KontHam{\mathcal{E},\mathcal{A}}$. Then for any two functionals $f,h\in\KontHam{\mathcal{E},\mathcal{A}}$, if one of them is Hamiltonian then we have
\[ \lim_{L\to 0} \{f,h\}_L = \{f,h\}. \]
\end{lemma}

\begin{proof}
Let $f\in\Hom_{\gf}(\mathcal{E}^{\cotimes n},\mathcal{A})$ and $h\in\Hom_{\gf}(\mathcal{E}^{\cotimes m},\mathcal{A})$ represent our functionals and assume that $h$ is Hamiltonian, in which case
\begin{displaymath}
\begin{split}
\{f,h\}_L &= (-1)^{|f|+|h|}\mu\bigl(f\cycsum{n}\cotimes\innprod_{\mathcal{A}}(\mathds{1}_{\mathcal{E\cotimes\mathcal{A}}}\cotimes X^h)\bigr)\bigl(\mathds{1}_{\mathcal{E}}^{\cotimes n-1}\cotimes K_L \cotimes\mathds{1}_{\mathcal{E}}^{\cotimes m-1}\bigr) \\
&= (-1)^{|f|+1}\mu\bigl(f\cycsum{n}\cotimes\mathds{1}_{\mathcal{A}}\bigr)\bigl(\mathds{1}_{\mathcal{E}}^{\cotimes n-1}\cotimes(\mathds{1}_{\mathcal{E}}\cotimes\innprod_{\mathcal{A}})(K_L\cotimes\mathds{1}_{\mathcal{E}\cotimes\mathcal{A}})X^h\bigr) \\
&= (-1)^{|f|}\mu\bigl(f\cycsum{n}\cotimes\mathds{1}_{\mathcal{A}}\bigr)\bigl(\mathds{1}_{\mathcal{E}}^{\cotimes n-1}\cotimes(K_L\star X^h)\bigr).
\end{split}
\end{displaymath}
The result now follows from Equation \eqref{eqn_vfieldaction}, Equation \eqref{eqn_Hambracket} and the fundamental property of the heat kernel.
\end{proof}

Next, we introduce the cobracket structure from \cite{movshevcobracket}.

\begin{defi}
Let $\mathcal{E}$ be a free BV-theory and $K\in\mathcal{E}\cotimes\mathcal{E}\cotimes\mathcal{A}$ be symmetric in $\mathcal{E}$ and of degree one. There is a well-defined map of degree one
\begin{equation} \label{eqn_cobracket}
\nabla_K:\KontHam{\mathcal{E}}\to\bigl[\KontHam{\mathcal{E}}\cotimes\KontHam{\mathcal{E}}\bigr]_{\sg{2}}\cotimes\mathcal{A}
\end{equation}
defined by
\[ \nabla_K(f_1\cdots f_n) := \sum_{1\leq i<j\leq n} (-1)^{p_{ij}} (f_{i+1}\cdots f_{j-1})(f_{j+1}\cdots f_n f_1\cdots f_{i-1})\cdot(f_i\otimes f_j)[K], \]
where again the sign
\[ p_{ij} := \sum_{1\leq r\leq i<s\leq n}|f_r||f_s| + |f_j|\Bigl(\sum_{r=1}^i|f_r|+\sum_{r=j+1}^n|f_r|\Bigr) + \sum_{r=1}^n|f_r| \]
is given by the Koszul sign rule as before.
\end{defi}

When $\mathcal{A}=\gf$, the bracket \eqref{eqn_bracket} and cobracket \eqref{eqn_cobracket} together satisfy the requisite axioms to form an involutive Lie bialgebra, see \cite{movshevcobracket}. This result may be expressed more generally by combining these structures to form a differential graded Lie algebra.

\subsubsection{Construction of the differential graded Lie algebras}

We use the following standard construction from an involutive differential graded Lie bialgebra. Note that both $\KontHam{\mathcal{E},\mathcal{A}}$ and $\intnoncomm{\mathcal{E},\mathcal{A}}$ have the canonical structure of an $\mathcal{A}$-bimodule.

\begin{defi}
Let $\mathcal{E}$ be a free BV-theory and $K\in\mathcal{E}\cotimes\mathcal{E}\cotimes\mathcal{A}$ be symmetric in $\mathcal{E}$ and of degree one. Consider the Chevalley-Eilenberg differential on $\intnuP{\mathcal{E},\mathcal{A}}:=\csalgP{\KontHam{\mathcal{E}}}\cotimes\mathcal{A}$ defined by
\[ \delta_K(x_1\cdots x_n\cdot a) := \sum_{1\leq i<j\leq n}\pm x_1\cdots\widehat{x_i}\cdots\widehat{x_j}\cdots x_n\cdot(\{x_i,x_j\}_K a), \]
for all $x_1,\ldots,x_n\in\KontHam{E}$ and $a\in\mathcal{A}$. As before, the sign is given by the Koszul sign rule.

Extend the cobracket \eqref{eqn_cobracket} to a differential $\nabla_K$ on $\intnuP{\mathcal{E},\mathcal{A}}$ using the Leibniz rule
\[ \nabla_K(x_1\cdots x_n\cdot a) = \sum_{i=1}^n \pm x_1\cdots\widehat{x_i}\cdots x_n\cdot(\nabla_K(x_i)a). \]
Again, the sign is given by the Koszul sign rule.

Lastly, extend the bracket \eqref{eqn_bracket} on $\KontHam{\mathcal{E},\mathcal{A}}$ to $\intnuP{\mathcal{E},\mathcal{A}}$ using the Leibniz rule
\[ \{x_1\cdots x_n\cdot a,y_1\cdots y_m\cdot b\}_K := \sum_{i=1}^n\sum_{j=1}^m \pm x_1\cdots\widehat{x_i}\cdots x_n\cdot y_1\cdots\widehat{y_j}\cdots y_m\cdot (\{x_i,y_j\}_Kab). \]
Again, the sign is given by the Koszul sign rule.

All the structures defined above may be extended to $\intnoncomm{\mathcal{E},\mathcal{A}}=\intnuP{\mathcal{E},\mathcal{A}}[[\gamma]]$ by $\gamma$-linearity. We define a noncommutative analogue of the BV-Laplacian on this space by
\[ \Delta_K:=\nabla_K+\gamma\delta_K:\intnoncomm{\mathcal{E},\mathcal{A}}\to\intnoncomm{\mathcal{E},\mathcal{A}}. \]
\end{defi}

\begin{rem} \label{rem_bracketextension}
Recall that in Definition \ref{def_Hambracket} we defined a bracket
\[ \{f,h\} = (-1)^{|f||h|+|h|}X^f(h); \quad f\in\KontHamL{\mathcal{E},\mathcal{A}},h\in\KontHam{\mathcal{E},\mathcal{A}}. \]
Now note that $\intnoncomm{\mathcal{E},\mathcal{A}}$ is a commutative algebra generated by $\KontHam{\mathcal{E},\mathcal{A}}$ and $\gamma$. We may therefore extend the above action of $\KontHamL{\mathcal{E},\mathcal{A}}$ to functionals $h\in\intnoncomm{\mathcal{E},\mathcal{A}}$ using the Leibniz rule
in the standard way. This action preserves the subspace $\intnoncommL{\mathcal{E},\mathcal{A}}$ of local functionals.

As in Lemma \ref{lem_bracketlimit}, given a gauge-fixing operator $Q^{\mathrm{GF}}$ we have
\begin{equation} \label{eqn_bracketlimit}
\lim_{L\to 0} \{f,h\}_L = \{f,h\}, \quad\text{for all }f\in\KontHamL{\mathcal{E},\mathcal{A}},h\in\intnoncomm{\mathcal{E},\mathcal{A}};
\end{equation}
where $\{f,h\}_L := -\{f,h\}_{K_L}$.
\end{rem}

The following result is standard, cf. \cite{movshevcobracket}.

\begin{prop} \label{prop_dglanoncomm}
Given a free BV-theory $\mathcal{E}$ and a symmetric $K\in\mathcal{E}\cotimes\mathcal{E}\cotimes\mathcal{A}$ of degree one that is a $(Q+d_{\mathrm{DR}})$-cycle,
\[ (\intnoncomm{\mathcal{E},\mathcal{A}},Q-d_{\mathrm{DR}}+\Delta_K,\cdot,\gamma\{-,-\}_K) \]
forms a BV-algebra.
\end{prop}

\begin{proof}
This is straightforward to check. We note that $\Delta_K^2$ vanishes as a consequence of the axioms for an involutive Lie bialgebra. The commutator
\[ [Q-d_{\mathrm{DR}},\Delta_K] = -\Delta_{(Q+d_{\mathrm{DR}})[K]} = 0. \]
\end{proof}

In particular, we may use Lemma \ref{lem_heatkernelidentities} to apply the above result to the heat kernel associated to a family of gauge-fixing operators. The above story of course has a simpler commutative analogue, which we now recall from \cite[\S 5.9]{CosEffThy}.

\begin{defi}
Let $\mathcal{E}$ be a free BV-theory and $K\in\mathcal{E}\cotimes\mathcal{E}\cotimes\mathcal{A}$ be symmetric in $\mathcal{E}$ and of degree one. The corresponding $\hbar$-linear BV-Laplacian $\Delta_K$ is defined on $\intcomm{\mathcal{E},\mathcal{A}}$ by
\[ \Delta_K(f_1\cdots f_n\cdot a) = \sum_{1\leq i < j \leq n} \pm f_1\cdots\widehat{f_i}\cdots\widehat{f_j}\cdots f_n ((f_i\otimes f_j)[K]\cdot a), \]
for all $f_1,\ldots,f_n\in\mathcal{E}^{\dag}$ and $a\in\mathcal{A}$. Again, the sign is given by the Koszul sign rule.
\end{defi}

\begin{prop} \label{prop_dglacomm}
Given a free BV-theory $\mathcal{E}$ and a symmetric $K\in\mathcal{E}\cotimes\mathcal{E}\cotimes\mathcal{A}$ of degree one that is a $(Q+d_{\mathrm{DR}})$-cycle,
\[ (\intcomm{\mathcal{E},\mathcal{A}},Q-d_{\mathrm{DR}}+\hbar\Delta_K,\cdot,\hbar\{-,-\}_K) \]
forms a BV-algebra.
\end{prop}

\subsubsection{The passage from noncommutative to commutative geometry}

The commutative and noncommutative aspects of the BV-formalism are related very simply through the map $\sigma_{\gamma,\nu}$ defined by Definition \ref{def_mapNCtoCom}. We consider the differential graded Lie algebras
\[ \bigl(\intnoncomm{\mathcal{E},\mathcal{A}},Q-d_{\mathrm{DR}}+\Delta_K,\{-,-\}_K\bigr) \quad\text{and}\quad \bigl(\intcomm{\mathcal{E},\mathcal{A}},Q-d_{\mathrm{DR}}+\hbar\Delta_K,\{-,-\}_K\bigr) \]
that arise as a consequence of Proposition \ref{prop_dglanoncomm} and \ref{prop_dglacomm} above.

\begin{prop} \label{prop_dglanoncommtocomm}
Let $\mathcal{E}$ be a free BV-theory and $K\in\mathcal{E}\cotimes\mathcal{E}\cotimes\mathcal{A}$ be a degree one $(Q+d_{\mathrm{DR}})$-cycle that is symmetric in $\mathcal{E}$:
\begin{enumerate}
\item \label{itm_dglanoncommtocomm1}
The natural quotient map
\[ \sigma:\KontHam{\mathcal{E},\mathcal{A}}\to\intcomm{\mathcal{E},\mathcal{A}} \]
is a map of Lie algebras.
\item \label{itm_dglanoncommtocomm2}
The map
\[ \sigma_{\gamma,\nu}:\intnoncomm{\mathcal{E},\mathcal{A}} \to \intcomm{\mathcal{E},\mathcal{A}} \]
is a map of differential graded Lie algebras.
\end{enumerate}
\end{prop}

\begin{proof}
This is a straightforward check, cf. Proposition 4.5 of \cite{GiGqHaZeLQT}.
\end{proof}

\subsubsection{The transformation associated to an Open Topological Field Theory}

The transformation associated to an OTFT defined by Definition \ref{def_OTFTtransformation} also respects the Quantum BV-structures we have defined on $\intnoncomm{\mathcal{E},\mathcal{A}}$.

\begin{prop} \label{prop_dglanoncommOTFT}
Let $\mathcal{E}$ be a free BV-theory and $\textgoth{A}$ be a differential graded Frobenius algebra. Suppose that $K\in\mathcal{E}\cotimes\mathcal{E}\cotimes\mathcal{A}$ is a $(Q+d_{\mathrm{DR}})$-cycle of degree one that is symmetric in $\mathcal{E}$, and let $K_{\textgoth{A}}$ be defined by \eqref{eqn_Frobeniusheatkernel}. Consider the differential graded Lie algebras
\[ \bigl(\intnoncomm{\mathcal{E},\mathcal{A}},Q-d_{\mathrm{DR}}+\Delta_K,\{-,-\}_K\bigr) \quad\text{and}\quad \bigl(\intnoncomm{\mathcal{E}_{\textgoth{A}},\mathcal{A}},Q_{\textgoth{A}}-d_{\mathrm{DR}}+\Delta_{K_{\textgoth{A}}},\{-,-\}_{K_{\textgoth{A}}}\bigr). \]
The map $\Morita:\intnoncomm{\mathcal{E},\mathcal{A}}\longrightarrow\intnoncomm{\mathcal{E}_{\textgoth{A}},\mathcal{A}}$ is a map of differential graded Lie algebras.
\end{prop}

\begin{proof}
This is a consequence of the axioms for an Open Topological Field Theory, cf. Lemma 3.13 and Theorem 5.1 of \cite{HamNCBV}.
\end{proof}

\subsubsection{Interpretation of the bracket and BV-Laplacian through Feynman diagrams} \label{sec_BVFeynmandiagrams}

The Lie bracket and cobracket defined in Section \ref{sec_Liebialgebra} precisely mirror the operations on stable ribbon graphs that arise from contracting loops and edges. In this section we will give a precise description of this fact, which is important not only because it provides an intuitive description of these operations, but also because it will provide a key combinatorial identity that we will use to prove that the renormalization group flow is compatible with the quantum master equation that is derived from these structures.

Let $\mathcal{E}$ be a free BV-theory and $K\in\mathcal{E}\cotimes\mathcal{E}\cotimes\mathcal{A}$ have degree one and be symmetric in $\mathcal{E}$, as before. Given a stable ribbon graph $G$ that is a corolla, set
\[ \mathcal{E}^{\dag}(\!(G)\!) := \mathcal{E}^{\dag}(\!(H(G))\!) = \mathcal{E}^{\dag}(\!(L(G))\!). \]
Recall that \eqref{eqn_attachmapscycdecompoff} determines a canonical map
\begin{equation} \label{eqn_corollamap}
\mathcal{E}^{\dag}(\!(G)\!)\cotimes\mathcal{A} \longrightarrow \intnoncomm{\mathcal{E},\mathcal{A}}, \qquad \phi\longmapsto \gamma^{g(G)}\nu^{b(G)}\bar{\phi};
\end{equation}
where $\bar{\phi}$ denotes the image of $\phi$ under the map defined by \eqref{eqn_attachmapscycdecompoff}.

Suppose that we are given a pair of distinct half-edges $e=\{h_1,h_2\}$ from the corolla $G$, by which we may form the corolla $G/e$ by treating $e$ as an edge and then contracting it. As before, we may attach $K$ to $e$ to form
\[ K(e)\in\mathcal{E}(\!(e)\!)\cotimes\mathcal{A}. \]
Given $\phi\in\Hom_{\gf}(\mathcal{E}(\!(G)\!),\mathcal{A})$, we will denote by $\phi[K(e)]$ the composite map
\begin{displaymath}
\xymatrix{
\mathcal{E}(\!(G/e)\!) \ar^-{x\mapsto K\otimes x}[rr] && \mathcal{E}(\!(e)\!)\cotimes\mathcal{A}\cotimes\mathcal{E}(\!(G/e)\!) = \mathcal{E}(\!(G)\!)\cotimes\mathcal{A} \ar^-{\phi\cotimes\mathds{1}}[r] & \mathcal{A}\cotimes\mathcal{A} \ar^{\mu}[r] & \mathcal{A}.
}
\end{displaymath}
Composing \eqref{eqn_corollamap} with the above, we define the map
\[ K_e:\mathcal{E}^{\dag}(\!(G)\!)\cotimes\mathcal{A}\longrightarrow\intnoncomm{\mathcal{E},\mathcal{A}}, \qquad \phi\longmapsto (-1)^{|\phi|}\gamma^{g(G/e)}\nu^{b(G/e)}\overline{\phi[K(e)]}. \]

\begin{prop} \label{prop_kernelloopcontract}
Let $\mathcal{E}$ be a free BV-theory and $K\in\mathcal{E}\cotimes\mathcal{E}\cotimes\mathcal{A}$ be symmetric in $\mathcal{E}$ and of degree one. Then for every stable ribbon graph $G$ that is a corolla, the following diagram commutes;
\begin{displaymath}
\xymatrix{
\mathcal{E}^{\dag}(\!(G)\!)\cotimes\mathcal{A} \ar^-{\phi\mapsto \gamma^{g(G)}\nu^{b(G)}\bar{\phi}}[rrr] \ar_{\sum_e K_e}[rrrd] &&& \intnoncomm{\mathcal{E},\mathcal{A}} \ar^{\Delta_K}[d] \\
&&& \intnoncomm{\mathcal{E},\mathcal{A}}
}
\end{displaymath}
where for the map defined by the bottom left arrow, we sum over all distinct pairs of half-edges $e$ from the corolla $G$.
\end{prop}

\begin{proof}
This is a straightforward consequence of our definition for the Lie bracket and cobracket that we introduced in Section \ref{sec_Liebialgebra} and the rules that we introduced for contracting loops in stable ribbon graphs in Section \ref{sec_stabribgraph}, cf. Figure \ref{fig_ContractLoopCobracket}.
\end{proof}

\begin{figure}[htp]
\centering
\includegraphics{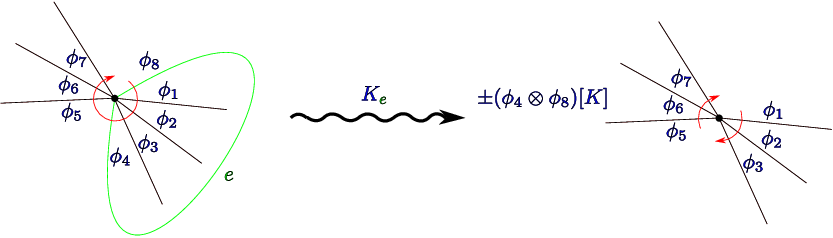}
\caption{Contracting loops in a ribbon graph defines the cobracket on $\KontHam{\mathcal{E},\mathcal{A}}$.}
\label{fig_ContractLoopCobracket}
\end{figure}

A similar story applies to the Lie bracket on $\intnoncomm{\mathcal{E},\mathcal{A}}$. Suppose that $G_1$ and $G_2$ are two corollas and take half-edges $h_1$ and $h_2$ from $G_1$ and $G_2$ respectively. By contracting the edge $e:=\{h_1,h_2\}$, we may form the corolla $(G_1\sqcup G_2)/e$. Given
\[ \phi_1\in\Hom_{\gf}(\mathcal{E}(\!(G_1)\!),\mathcal{A}) \quad\text{and}\quad \phi_2\in\Hom_{\gf}(\mathcal{E}(\!(G_2)\!),\mathcal{A}), \]
we will denote by $(\phi_1\otimes\phi_2)[K(e)]$ the map defined by the following commutative diagram
\begin{displaymath}
\xymatrix{
\mathcal{E}(\!(G_1\sqcup G_2/e)\!) \ar^-{x\mapsto K\otimes x}[rr] \ar_{(\phi_1\otimes\phi_2)[K(e)]}[d] && \mathcal{E}(\!(e)\!)\cotimes\mathcal{A}\cotimes\mathcal{E}(\!(G_1\sqcup G_2/e)\!) \ar@{=}[d] \\
\mathcal{A} & \mathcal{A}\cotimes\mathcal{A}\cotimes\mathcal{A} \ar_{\mu}[l] & \mathcal{E}(\!(G_1)\!)\cotimes\mathcal{E}(\!(G_2)\!) \cotimes\mathcal{A} \ar_-{\phi_1\cotimes\phi_2\cotimes\mathds{1}}[l]
}
\end{displaymath}
Composing \eqref{eqn_corollamap} with the above defines a map
\begin{multline*}
K_e:\bigl(\mathcal{E}^{\dag}(\!(G_1)\!)\cotimes\mathcal{A}\bigr)\cotimes\bigl(\mathcal{E}^{\dag}(\!(G_2)\!)\cotimes\mathcal{A}\bigr) \longrightarrow \intnoncomm{\mathcal{E},\mathcal{A}}, \\
(\phi_1,\phi_2)\longmapsto (-1)^{|\phi_1|+|\phi_2|}\gamma^{g(G_1\sqcup G_2/e)}\nu^{b(G_1\sqcup G_2/e)}\overline{(\phi_1\otimes\phi_2)[K(e)]}.
\end{multline*}

\begin{prop} \label{prop_kerneledgecontract}
Let $\mathcal{E}$ be a free BV-theory and $K\in\mathcal{E}\cotimes\mathcal{E}\cotimes\mathcal{A}$ be symmetric in $\mathcal{E}$ and of degree one. If $G_1$ and $G_2$ are any two stable ribbon graph corollas, then the following diagram commutes;
\begin{displaymath}
\xymatrix{
\bigl(\mathcal{E}^{\dag}(\!(G_1)\!)\cotimes\mathcal{A}\bigr)\cotimes\bigl(\mathcal{E}^{\dag}(\!(G_2)\!)\cotimes\mathcal{A}\bigr) \ar[r] \ar_{\sum_e K_e}[rd] & \intnoncomm{\mathcal{E},\mathcal{A}}\cotimes\intnoncomm{\mathcal{E},\mathcal{A}} \ar^{\{-,-\}_K}[d] \\
& \intnoncomm{\mathcal{E},\mathcal{A}}
}
\end{displaymath}
where the top arrow applies the maps defined by \eqref{eqn_corollamap}, and for the map defined by the bottom left arrow, we sum over all pairs $e=\{h_1,h_2\}$ for which the half-edges $h_1$ and $h_2$ are taken from $G_1$ and $G_2$ respectively.
\end{prop}

\begin{proof}
This is again a straightforward consequence of our definition for the Lie bracket $\{-,-\}_K$ and the rules for contracting edges in stable ribbon graphs, cf. Figure \ref{fig_ContractEdgeBracket}.
\end{proof}

\begin{figure}[htp]
\centering
\includegraphics{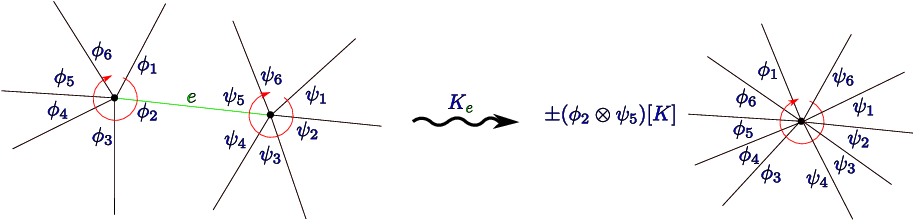}
\caption{Contracting edges in a ribbon graph defines the bracket on $\KontHam{\mathcal{E},\mathcal{A}}$.}
\label{fig_ContractEdgeBracket}
\end{figure}

\subsection{Quantum master equation} \label{sec_QME}

Having constructed our differential graded Lie algebra from the noncommutative geometry introduced by Kontsevich in \cite{KontSympGeom}, we are ready now to introduce the quantum master equation. Following \cite{CosEffThy}, this must be described by an equation at every length scale. We will discover, as in \cite{CosEffThy}, that the renormalization group flow that we defined in Section \ref{sec_rgflownoncomm} connects these solutions to the quantum master equation at different length scales.

\subsubsection{The scale $L$ quantum master equation}

\begin{defi}
Let $\mathcal{E}$ be a free BV-theory with a family of gauge-fixing operators $Q^{\mathrm{GF}}$ and let $K_t$ be the heat kernel corresponding to this family. Define
\[ \{I,J\}_L := -\{I,J\}_{K_L} \quad\text{and}\quad \Delta_L(I) := -\Delta_{K_L}(I) \]
for $I,J\in\intnoncomm{\mathcal{E},\mathcal{A}}$.

We will say that an interaction $I\in\intnoncommI{\mathcal{E},\mathcal{A}}$ satisfies the \emph{scale $L$ quantum master equation} (QME) if
\begin{equation} \label{eqn_scaleQME}
(Q-d_{\mathrm{DR}}+\Delta_L)I + \frac{1}{2}\{I,I\}_L = 0.
\end{equation}
\end{defi}

In a similar way, we may define the scale $L$ quantum master equation in $\intcomm{\mathcal{E},\mathcal{A}}$ using the differential graded Lie algebra from Proposition \ref{prop_dglacomm}, cf. \cite[\S 5.9.2]{CosEffThy}.

\subsubsection{Compatibility with the renormalization group flow}

We now establish one of the main results of this paper; that the renormalization group flow transforms solutions to the quantum master equation at scale $\varepsilon$ into solutions at scale $L$. For this theorem we will require the special Feynman amplitudes defined in Section \ref{sec_specialFamplitudes} and the following basic definition.

\begin{defi}
Let $\Gamma$ be a connected stable ribbon graph (or stable graph) and let $e$ be an edge of $\Gamma$. Denote by $(\Gamma-e)$, the result of cutting the graph $\Gamma$ at the edge $e$, which is effected by simply removing $e$ from the set of edges $E(\Gamma)$. We say that the edge $e$ is \emph{separating} if the graph becomes disconnected  after cutting it at $e$. Otherwise, we call the edge $e$ \emph{nonseparating}.
\end{defi}

We will also need the following basic lemma.

\begin{lemma} \label{lem_cyclicsetdettachattach}
Let $\mathcal{E}$ be a free BV-theory and let $u$ and $v$ be two sets equipped with cyclic decompositions. Consider the  canonical maps
\[ \mathcal{E}^{\dag}(\!(u)\!)\to\csalg{\KontHamP{\mathcal{E}}} \quad\text{and}\quad \csalg{\KontHamP{\mathcal{E}}}\to\mathcal{E}^{\dag}(\!(v)\!) \]
defined by \eqref{eqn_attachmapscycdecompoff} and \eqref{eqn_attachmapscycdecompon} respectively. Their composite is the map
\[ \mathcal{E}^{\dag}(\!(u)\!)\to\mathcal{E}^{\dag}(\!(v)\!), \qquad x\mapsto\sum_{\iota:u\to v}\iota^{\#}(x); \]
where the sum is taken over all those bijections $\iota:u\to v$ respecting the cyclic decompositions on $u$ and $v$ (with the answer being zero if there are none).
\end{lemma}

\begin{proof}
The statement ultimately follows tautologically from the definitions.
\end{proof}

\begin{theorem} \label{thm_QMEflow}
Let $\mathcal{E}$ be a free BV-theory with a family of gauge-fixing operators $Q^{\mathrm{GF}}$ and let $P(\varepsilon,L)$ denote the canonical propagator defined by Equation \eqref{eqn_canonicalpropagator}. Then for all interactions $I\in\intnoncommI{\mathcal{E},\mathcal{A}}$,
\begin{multline} \label{eqn_QMEflow}
\bigl(Q-d_{\mathrm{DR}}+\Delta_L\bigr)W(I,P(\varepsilon,L)) +\frac{1}{2}\bigl\{W(I,P(\varepsilon,L)),W(I,P(\varepsilon,L))\bigr\}_L \\
= \sum_{(\Gamma,v)}\frac{\gamma^{g(\Gamma)}\nu^{b(\Gamma)}}{|\Aut(\Gamma,v)|}w_{(\Gamma,v)}\Bigl(I,(Q-d_{\mathrm{DR}}+\Delta_{\varepsilon})I+\frac{1}{2}\{I,I\}_{\varepsilon};P(\varepsilon,L)\Bigr);
\end{multline}
where the sum is taken over all (isomorphism classes from the category of) connected stable ribbon graphs with a distinguished vertex.

In particular, the renormalization group flow $W(-,P(\varepsilon,L))$ takes a solution to the scale $\varepsilon$ QME, to a solution to the scale $L$ QME.
\end{theorem}

\begin{proof}
First, we compute
\begin{equation} \label{eqn_QMEflowaux2}
\begin{split}
(Q-d_{\mathrm{DR}})W(I,P(\varepsilon,L)) =& \sum_{\Gamma}\sum_{v\in V(\Gamma)} \frac{\gamma^{g(\Gamma)}\nu^{b(\Gamma)}}{|\Aut(\Gamma)|}w_{(\Gamma,v)}\bigl(I,(Q-d_{\mathrm{DR}})I;P(\varepsilon,L)\bigr) \\
& -\sum_{\Gamma}\sum_{e\in E(\Gamma)} \frac{\gamma^{g(\Gamma)}\nu^{b(\Gamma)}}{|\Aut(\Gamma)|}w_{(\Gamma,e)}\bigl(I;P(\varepsilon,L),(Q+d_{\mathrm{DR}})P(\varepsilon,L)\bigr).
\end{split}
\end{equation}
This is just a simple consequence of the Leibniz rule.

Next, it follows from Proposition \ref{prop_kernelloopcontract} that
\begin{equation} \label{eqn_QMEflowaux1}
\Delta_L W(I,P(\varepsilon,L)) = -\sum_{\Gamma}\sum_{e} \frac{\gamma^{g(\Gamma+e)}\nu^{b(\Gamma+e)}}{|\Aut(\Gamma)|}w_{(\Gamma+e,e)}\bigl(I;P(\varepsilon,L),K_L\bigr),
\end{equation}
where we sum over all pairs of half-edges $e$ that are taken from the legs of $\Gamma$, and $(\Gamma+e)$ denotes the graph that is obtained from $\Gamma$ by including $e$ as an edge.

Likewise, it follows from Proposition \ref{prop_kerneledgecontract} that
\begin{multline} \label{eqn_QMEflowaux3}
\bigl\{W(I,P(\varepsilon,L)),W(I,P(\varepsilon,L))\bigr\}_L \\
= -\sum_{\Gamma,\Gamma'}\sum_e \frac{\gamma^{g(\Gamma\sqcup\Gamma'+e)}\nu^{b(\Gamma\sqcup\Gamma'+e)}}{|\Aut(\Gamma)||\Aut(\Gamma')|} w_{(\Gamma\sqcup\Gamma'+e,e)}\bigl(I;P(\varepsilon,L),K_L\bigr),
\end{multline}
where the sum is taken over all those pairs of half-edges $e$ in which one half-edge is taken from the legs of the connected stable ribbon graph $\Gamma$, and the other is taken from the legs of $\Gamma'$. The graph $(\Gamma\sqcup\Gamma'+e)$ above denotes the connected stable ribbon graph formed by joining $\Gamma$ to $\Gamma'$ with the new edge $e$.

We may rewrite Equation \eqref{eqn_QMEflowaux1} as
\[ \Delta_L W(I,P(\varepsilon,L)) = -\sum_{(\Gamma,e)} \frac{\gamma^{g(\Gamma)}\nu^{b(\Gamma)}N_{\Gamma-e}}{|\Aut(\Gamma-e)|}w_{(\Gamma,e)}\bigl(I;P(\varepsilon,L),K_L\bigr); \]
where we sum over all connected stable ribbon graphs $\Gamma$ with a distinguished \emph{nonseparating} edge $e$, and $N_{\Gamma-e}$ counts the number of pairs $e'=\{h_1,h_2\}$ of half-edges $h_1$ and $h_2$ taken from $L(\Gamma-e)$ for which there is an automorphism of $(\Gamma-e)$ taking $e$ to $e'$. Obviously, $N_{\Gamma-e}$ may be computed using the orbit-stabilizer theorem, which provides us with
\begin{equation} \label{eqn_QMEflowaux4}
\Delta_L W(I,P(\varepsilon,L)) = -\sum_{(\Gamma,e)} \frac{\gamma^{g(\Gamma)}\nu^{b(\Gamma)}}{|\Aut(\Gamma,e)|}w_{(\Gamma,e)}\bigl(I;P(\varepsilon,L),K_L\bigr).
\end{equation}

Applying Lemma \ref{lem_kernelhomotopy} to Equation \eqref{eqn_QMEflowaux2} and employing the same counting arguments as above, we may arrive at the equation
\begin{equation} \label{eqn_QMEflowaux5}
\begin{split}
(Q-d_{\mathrm{DR}})W(I,P(\varepsilon,L)) =& \sum_{(\Gamma,v)} \frac{\gamma^{g(\Gamma)}\nu^{b(\Gamma)}}{|\Aut(\Gamma,v)|}w_{(\Gamma,v)}\bigl(I,(Q-d_{\mathrm{DR}})I;P(\varepsilon,L)\bigr) \\
& +\sum_{(\Gamma,e)} \frac{\gamma^{g(\Gamma)}\nu^{b(\Gamma)}}{|\Aut(\Gamma,e)|}w_{(\Gamma,e)}\bigl(I;P(\varepsilon,L),K_L-K_{\varepsilon})\bigr);
\end{split}
\end{equation}
where the first sum is taken over all connected graphs $\Gamma$ with distinguished vertex $v$, and the second over all connected graphs $\Gamma$ with distinguished edge $e$.

Turning now to Equation \eqref{eqn_QMEflowaux3}, we note that when $\Gamma$ is not isomorphic to $\Gamma'$ then that pair appears twice in the double sum; while when they are isomorphic, that pair appears only once, but then $(\Gamma\sqcup\Gamma')$ obtains an additional symmetry arising from permuting the two connected components. From this it follows that
\[ \frac{1}{2}\bigl\{W(I,P(\varepsilon,L)),W(I,P(\varepsilon,L))\bigr\}_L = -\sum_{(\Gamma,e)} \frac{\gamma^{g(\Gamma)}\nu^{b(\Gamma)}N_{\Gamma-e}}{|\Aut(\Gamma-e)|} w_{(\Gamma,e)}\bigl(I;P(\varepsilon,L),K_L\bigr), \]
where we sum over all connected stable ribbon graphs $\Gamma$ with a distinguished \emph{separating} edge $e$, and $N_{\Gamma-e}$ counts the number of pairs $e'=\{h_1,h_2\}$ of half-edges $h_1$ and $h_2$ taken from separate components of $(\Gamma-e)$, for which there is an automorphism of $(\Gamma-e)$ taking $e$ to $e'$. Again, this may be computed from the orbit-stabilizer theorem, and combining the result with Equation \eqref{eqn_QMEflowaux4} yields
\begin{multline} \label{eqn_QMEflowaux6}
\Delta_L W(I,P(\varepsilon,L)) + \frac{1}{2}\bigl\{W(I,P(\varepsilon,L)),W(I,P(\varepsilon,L))\bigr\}_L \\
= -\sum_{(\Gamma,e)} \frac{\gamma^{g(\Gamma)}\nu^{b(\Gamma)}}{|\Aut(\Gamma,e)|}w_{(\Gamma,e)}\bigl(I;P(\varepsilon,L),K_L\bigr);
\end{multline}
where now we sum over \emph{all} connected stable ribbon graphs $\Gamma$ with a distinguished edge~$e$.

Combining Equation \eqref{eqn_QMEflowaux5} with Equation \eqref{eqn_QMEflowaux6} we arrive at
\begin{multline} \label{eqn_QMEflowaux7}
(Q-d_{\mathrm{DR}}+\Delta_L)W(I,P(\varepsilon,L)) + \frac{1}{2}\bigl\{W(I,P(\varepsilon,L)),W(I,P(\varepsilon,L))\bigr\}_L = \\
\sum_{(\Gamma,v)} \frac{\gamma^{g(\Gamma)}\nu^{b(\Gamma)}}{|\Aut(\Gamma,v)|}w_{(\Gamma,v)}\bigl(I,(Q-d_{\mathrm{DR}})I;P(\varepsilon,L)\bigr) - \sum_{(\Gamma,e)} \frac{\gamma^{g(\Gamma)}\nu^{b(\Gamma)}}{|\Aut(\Gamma,e)|}w_{(\Gamma,e)}\bigl(I;P(\varepsilon,L),K_{\varepsilon})\bigr).
\end{multline}

Applying Theorem \ref{thm_intgrpact} to Equation \eqref{eqn_QMEflowaux6} with $L=\varepsilon$ yields
\[ \Delta_{\varepsilon}I + \frac{1}{2}\bigl\{I,I\bigr\}_{\varepsilon} = -\sum_G \frac{\gamma^{g(G)}\nu^{b(G)}}{|\Aut(G)|}w_G\bigl(I,K_{\varepsilon}\bigr), \]
where the sum is now taken over all connected stable ribbon graphs $G$ that have precisely one edge, to which the heat kernel $K_{\varepsilon}$ is attached.

Given a connected stable ribbon graph $\Gamma$ with a distinguished vertex $v$, we may apply Lemma \ref{lem_cyclicsetdettachattach} to the preceding equation to get
\[ w_{(\Gamma,v)}\Bigl(I,\Delta_{\varepsilon}I+\frac{1}{2}\{I,I\}_{\varepsilon};P(\varepsilon,L)\Bigr) = -\sum_{G,\iota}\frac{1}{|\Aut(G)|}w_{((\Gamma,v),\iota,G)}(I;P(\varepsilon,L),K_{\varepsilon}); \]
where the sum is taken over all connected stable ribbon graphs $G$:
\begin{itemize}
\item
with precisely one edge, and
\item
satisfying $g(G)=g(v)$ and $b(G)=b(v)$,
\end{itemize}
and all bijections $\iota$ from the legs of $G$ to the incident half-edges of $v$ that preserve the cyclic decompositions. From this we get
\begin{multline*}
\sum_{(\Gamma,v)}\frac{\gamma^{g(\Gamma)}\nu^{b(\Gamma)}}{|\Aut(\Gamma,v)|}w_{(\Gamma,v)}\Bigl(I,\Delta_{\varepsilon}I+\frac{1}{2}\{I,I\}_{\varepsilon};P(\varepsilon,L)\Bigr) \\
= -\sum_{((\Gamma,v),\iota,G)} \frac{\gamma^{g(\Gamma)}\nu^{b(\Gamma)}N_{((\Gamma,v),\iota,G)}}{|\Aut(\Gamma,v)||\Aut(G)|} w_{((\Gamma,v),\iota,G)}(I;P(\varepsilon,L),K_{\varepsilon}),
\end{multline*}
where the sum is taken over all (isomorphism classes of) triples from the category defined in Definition \ref{def_2ndCat} and $N_{((\Gamma,v),\iota,G)}$ counts the number of bijections $\iota'$ from the legs of $G$ to the incident half-edges of $v$ that can be obtained from $\iota$ through an automorphism of both $(\Gamma,v)$ and $G$. This is easily calculated in the same way as before, yielding
\[ N_{((\Gamma,v),\iota,G)} = \frac{|\Aut(\Gamma,v)||\Aut(G)|}{|\Aut((\Gamma,v),\iota,G)|}. \]

Now consider the equivalence of categories defined by Theorem \ref{thm_catequiv}. Under this correspondence we have
\[ w_{(\Gamma,e)}(I;P(\varepsilon,L),K_{\varepsilon}) = w_{((\Gamma/e,v),\iota,G(e))}(I;P(\varepsilon,L),K_{\varepsilon}). \]
Hence, using this correspondence we may write
\begin{multline*}
\sum_{(\Gamma,v)}\frac{\gamma^{g(\Gamma)}\nu^{b(\Gamma)}}{|\Aut(\Gamma,v)|}w_{(\Gamma,v)}\Bigl(I,\Delta_{\varepsilon}I+\frac{1}{2}\{I,I\}_{\varepsilon};P(\varepsilon,L)\Bigr) \\
= -\sum_{(\Gamma,e)} \frac{\gamma^{g(\Gamma)}\nu^{b(\Gamma)}}{|\Aut(\Gamma,e)|}w_{(\Gamma,e)}(I;P(\varepsilon,L),K_{\varepsilon}),
\end{multline*}
where the latter sum is taken over all connected stable ribbon graphs with a distinguished edge. It only remains to combine the above with Equation \eqref{eqn_QMEflowaux7} to complete our proof of Equation \eqref{eqn_QMEflow}.
\end{proof}

\subsubsection{The scale $L$ classical master equation and the tree-level flow} \label{sec_CMEtreelevelflow}

Let $\mathcal{E}$ be a free BV-theory with a family of gauge-fixing operators $Q^{\mathrm{GF}}$ and consider the filtration on $\intnoncomm{\mathcal{E},\mathcal{A}}$ defined by Definition \ref{def_NCfiltration}. Note that this is a filtration of the differential graded Lie algebra defined by Proposition \ref{prop_dglanoncomm} and hence we may consider the corresponding scale $L$ quantum master equation that is defined on the relevant quotients. In particular, if we look at the tree-level interactions
\[ I\in\intnoncommItree{\mathcal{E},\mathcal{A}} = \intnoncommI{\mathcal{E},\mathcal{A}}/F_1 \intnoncommI{\mathcal{E},\mathcal{A}} \]
then the BV-Laplacian $\Delta_L$ disappears in the filtration and the scale $L$ quantum master equation becomes
\[ (Q-d_{\mathrm{DR}})I + \frac{1}{2}\{I,I\}_L = 0, \]
which we call the \emph{scale $L$ classical master equation} for the tree-level interaction $I$.

Now, recall that Theorem \ref{thm_locinteffthy} states that there is a one-to-one correspondence between local tree-level interactions and tree-level pretheories that assigns to a local tree-level interaction $I$, the pretheory
\begin{equation} \label{eqn_treelevelcorrespondence}
I^\mathrm{R}[L] :=  \lim_{\varepsilon\to 0}W^{\mathrm{Tree}}\bigl(I,P(\varepsilon,L)\bigr),
\end{equation}
where in the above we have used the fact that there are no tree-level counterterms, see Lemma \ref{lem_treecterterm}\eqref{itm_treecterterm2}.

\begin{rem} \label{rem_treelimit}
Note also that it follows from Lemma \ref{lem_treecterterm}\eqref{itm_treecterterm1} and Equation \eqref{eqn_intgrpactzero} that
\[ \lim_{L\to 0} I^\mathrm{R}[L] = I, \]
see Proposition 4.30 of \cite{NCRGF}. We emphasize that this only holds for tree-level pretheories.
\end{rem}

\begin{prop} \label{prop_treeCME}
Let $\mathcal{E}$ be a free BV-theory with a family of gauge-fixing operators $Q^{\mathrm{GF}}$ and let
\[ I\in\intnoncommLItree{\mathcal{E},\mathcal{A}}\subset\KontHamL{\mathcal{E},\mathcal{A}} \]
be a local tree-level interaction. Then $I$ satisfies the classical master equation
\begin{equation} \label{eqn_treeCME}
(Q-d_{\mathrm{DR}})I + \frac{1}{2}\{I,I\} = 0
\end{equation}
if and only if for all $L>0$, the tree-level interactions $I^\mathrm{R}[L]$ satisfy the scale $L$ classical master equation
\[ (Q-d_{\mathrm{DR}})I^\mathrm{R}[L] + \frac{1}{2}\bigl\{I^\mathrm{R}[L],I^\mathrm{R}[L]\bigr\}_L = 0. \]
\end{prop}

\begin{proof}
From Equation \eqref{eqn_QMEflow} we get
\begin{multline} \label{eqn_treeCMEaux1}
\bigl(Q-d_{\mathrm{DR}}\bigr)W^{\mathrm{Tree}}(I,P(\varepsilon,L)) +\frac{1}{2}\bigl\{W^{\mathrm{Tree}}(I,P(\varepsilon,L)),W^{\mathrm{Tree}}(I,P(\varepsilon,L))\bigr\}_L \\
= \sum_{(\Gamma,v)}\frac{\gamma^{g(\Gamma)}\nu^{b(\Gamma)}}{|\Aut(\Gamma,v)|}w_{(\Gamma,v)}\Bigl(I,(Q-d_{\mathrm{DR}})I+\frac{1}{2}\{I,I\}_{\varepsilon};P(\varepsilon,L)\Bigr),
\end{multline}
where the sum is taken over all trees with a distinguished vertex.

Now it follows from Lemma \ref{lem_bracketlimit} that
\[ \lim_{\varepsilon\to 0}\Bigl[(Q-d_{\mathrm{DR}})I+\frac{1}{2}\{I,I\}_{\varepsilon}\Bigr] = (Q-d_{\mathrm{DR}})I+\frac{1}{2}\{I,I\}. \]
Hence taking the limit as $\varepsilon\to 0$ in \eqref{eqn_treeCMEaux1} yields
\begin{multline} \label{eqn_treeCMEaux2}
\bigl(Q-d_{\mathrm{DR}}\bigr)I^\mathrm{R}[L] +\frac{1}{2}\bigl\{I^\mathrm{R}[L],I^\mathrm{R}[L]\bigr\}_L \\
= \sum_{(\Gamma,v)}\frac{\gamma^{g(\Gamma)}\nu^{b(\Gamma)}}{|\Aut(\Gamma,v)|}\lim_{\varepsilon\to 0}\biggl[w_{(\Gamma,v)}\Bigl(I,(Q-d_{\mathrm{DR}})I+\frac{1}{2}\{I,I\};P(\varepsilon,L)\Bigr)\biggl].
\end{multline}
Of course, we must be careful above when taking the limit on the right-hand side of \eqref{eqn_treeCMEaux1}, as $P(\varepsilon,L)$ does not converge as $\varepsilon\to 0$. However, we are assisted by the fact that $I$ is local, the graphs are trees, and the corresponding convolution operator $P(\varepsilon,L)\star$ does converge; so the argument proceeds similarly to Lemma 4.13 of \cite{NCRGF}.

The conclusion now follows by taking the limit as $L\to 0$ in \eqref{eqn_treeCMEaux2}. By Lemma \ref{lem_treecterterm}\eqref{itm_treecterterm1} the sum on the right-hand side degenerates into a sum over all corollas, which by Equation \eqref{eqn_intgrpactzero} yields
\[ (Q-d_{\mathrm{DR}})I+\frac{1}{2}\{I,I\}. \]
\end{proof}

\subsubsection{A vanishing criteria arising from the Loday-Quillen-Tsygan Theorem} \label{sec_vanishcriteria}

An important vanishing criteria for our interaction terms arises as a consequence of the Loday-Quillen-Tsygan Theorem \cite{LodayQuillen, Tsygan}. This criteria, as we will explore in more detail in Section \ref{sec_EffBVtheories}, will allow us to connect the behavior of our theories in the noncommutative framework to the behavior of gauge theories in the large $N$ limit. This criteria appears as Theorem 5.2 of \cite{NCRGF}.

\begin{theorem} \label{thm_vanishing}
Let $\mathcal{E}$ be a free BV-theory. A functional $I\in\intnoncomm{\mathcal{E},\mathcal{A}}$ vanishes if and only if
\[ \sigma_{\gamma,\nu}\Morita[\mat{N}{\gf}](I)=0, \quad\text{for all } N\geq 1. \]
\end{theorem}

From this we obtain the following corollary.

\begin{cor}
Let $\mathcal{E}$ be a free BV-theory with a family of gauge-fixing operators $Q^{\mathrm{GF}}$. An interaction $I\in\intnoncommI{\mathcal{E},\mathcal{A}}$ satisfies the scale $L$ quantum master equation if and only if for all $N\geq 1$, the interaction
\begin{equation} \label{eqn_intlargeN}
I^{\mat{N}{\gf}}:=\sigma_{\gamma,\nu}\Morita[\mat{N}{\gf}](I)\in\intcomm{\mathcal{E}_{\mat{N}{\gf}},\mathcal{A}}
\end{equation}
satisfies the scale $L$ quantum master equation.
\end{cor}

\begin{proof}
This follows immediately from Theorem \ref{thm_vanishing}, Proposition \ref{prop_dglanoncommtocomm} and Proposition \ref{prop_dglanoncommOTFT}.
\end{proof}

Another related criteria that we will need is the following.
\begin{cor} \label{cor_vanishingmodconstants}
Let $\mathcal{E}$ be a free BV-theory with a family of gauge-fixing operators $Q^{\mathrm{GF}}$ and let $I\in\intnoncommI{\mathcal{E},\mathcal{A}}$ be an interaction; then
\[ (Q-d_{\mathrm{DR}}+\Delta_L)I + \frac{1}{2}\{I,I\}_L \in \gf[[\gamma,\nu]]\cotimes\mathcal{A} \]
if and only if for all $N\geq 1$, the interaction $I^{\mat{N}{\gf}}$ defined by \eqref{eqn_intlargeN} satisfies
\[ (Q-d_{\mathrm{DR}}+\hbar\Delta_L)I^{\mat{N}{\gf}} + \frac{1}{2}\bigl\{I^{\mat{N}{\gf}},I^{\mat{N}{\gf}}\bigr\}_L \in \gf[[\hbar]]\cotimes\mathcal{A}. \]
\end{cor}

\begin{proof}
This follows in a similar manner from Corollary 5.3 of \cite{NCRGF}.
\end{proof}

\section{Noncommutative effective field theories in the BV-formalism} \label{sec_NCEffThyBV}

We are now ready to give a definition, following \cite{CosEffThy}, of a theory in the framework of the Batalin-Vilkovisky formalism and noncommutative geometry. Such a theory will, on its face, depend upon the choice of a gauge-fixing operator. However, this dependence disappears when we pass to homotopy classes of theories. This independence on the choice of gauge is expressed by the statement that the space of theories is a Kan fibration over the space of gauge-fixing conditions.

In what follows, $\mathcal{A}$ will be the de Rham algebra on some smooth manifold $X$ with corners. Occasionally, we will consider the commutative algebra $\mathcal{A}\otimes\gf[\zeta]$, where $\zeta$ is a commuting generator of degree minus-one (and so $\zeta^2=0$).

\subsection{Theories in the BV-formalism} \label{sec_EffBVtheories}

After providing the definition for a theory in the noncommutative framework, we will proceed to describe how to pass from this noncommutative framework to the commutative framework used by Costello in \cite{CosEffThy}. We will then explain how to put these ideas together to provide a picture that describes the string/gauge theory correspondence in its most general sense.

\subsubsection{The definition of a theory}

As in \cite{CosEffThy}, theories will be families of effective interactions parameterized by the length scale and satisfying the quantum master equation that applies in this noncommutative framework.

\begin{defi}
Let $\mathcal{E}$ be a free BV-theory with a family of gauge-fixing operators $Q^{\mathrm{GF}}$. A \emph{theory in the noncommutative BV-formalism} consists of a pretheory
\[ I\in\NCPreThy{\mathcal{E},\mathcal{A}} \]
such that $I[L]$ satisfies the scale $L$ quantum master equation \eqref{eqn_scaleQME} for all $L>0$. We will denote the set of all such theories by $\NCThy{\mathcal{E},\mathcal{A}}$.
\end{defi}

\begin{rem}
Recalling Definition 9.3.1 of \cite[\S 5.9.3]{CosEffThy}, the definition of a theory in the (commutative) BV-formalism is obtained from the above by replacing $\NCPreThy{\mathcal{E},\mathcal{A}}$ with $\PreThy{\mathcal{E},\mathcal{A}}$. We will denote the set of all such (commutative) theories by $\Thy{\mathcal{E},\mathcal{A}}$.
\end{rem}

\begin{rem}
As in Remark \ref{rem_GFdependence}, we will denote the set of theories by
\[ \NCThy{\mathcal{E},\mathcal{A};Q^{\mathrm{GF}}} \quad\text{and}\quad \Thy{\mathcal{E},\mathcal{A};Q^{\mathrm{GF}}} \]
when we wish to make the dependence on the family of gauge-fixing operators $Q^{\mathrm{GF}}$ explicit, and omit $\mathcal{A}$ from the notation when $\mathcal{A}=\gf$.
\end{rem}

For the purposes of describing how to construct such a theory, a process known as \emph{quantization}, we will need the definition of a level $p$ theory. Recall that the filtration defined by Definition \ref{def_NCfiltration} is a filtration of the differential graded Lie algebra defined by Proposition \ref{prop_dglanoncomm}, hence the quantum master equation makes sense on the corresponding quotients.

\begin{defi}
Let $\mathcal{E}$ be a free BV-theory with a family of gauge-fixing operators $Q^{\mathrm{GF}}$. A \emph{level $p$ theory in the noncommutative BV-formalism} consists of a level $p$ pretheory
\[ I\in\NCPreThyL{p}{\mathcal{E},\mathcal{A}} \]
such that $I[L]$ satisfies the scale $L$ quantum master equation (modulo $F_{p+1}$) for all $L>0$. We will denote the set of all such level $p$ theories by $\NCThyL{p}{\mathcal{E},\mathcal{A}}$.
\end{defi}

\begin{rem}
As before, we may talk about level $p$ (commutative) theories, though we will not need them in this paper.
\end{rem}

\subsubsection{The passage from noncommutative to commutative geometry}

The map $\sigma_{\gamma,\nu}$ defined by Definition \ref{def_mapNCtoCom} provides us with a mechanism to pass from noncommutative theories to commutative theories.

\begin{prop} \label{prop_noncommtocommthy}
Let $\mathcal{E}$ be a free BV-theory with a family of gauge-fixing operators $Q^{\mathrm{GF}}$. The map $\sigma_{\gamma,\nu}$ transforms noncommutative theories into commutative theories,
\[ \sigma_{\gamma,\nu}:\NCThy{\mathcal{E},\mathcal{A}}\longrightarrow\Thy{\mathcal{E},\mathcal{A}}. \]
\end{prop}

\begin{proof}
This is a direct consequence of Theorem \ref{thm_flowNCtoCom} and Proposition \ref{prop_dglanoncommtocomm}.
\end{proof}

\subsubsection{The transformation of theories arising from an Open Topological Field Theory}

A mechanism for producing noncommutative theories arises from the transformation $\Morita$ defined by Definition \ref{def_OTFTtransformation} that arises from an Open Topological Field Theory.

\begin{prop} \label{prop_OTFTthy}
Let $\mathcal{E}$ be a free BV-theory with a family of gauge-fixing operators $Q^{\mathrm{GF}}$ and let $\textgoth{A}$ be a differential graded Frobenius algebra. The map $\Morita$ transforms noncommutative theories defined on $\mathcal{E}$ into noncommutative theories defined on $\mathcal{E}_{\textgoth{A}}$,
\[ \Morita:\NCThy{\mathcal{E},\mathcal{A}}\longrightarrow\NCThy{\mathcal{E}_{\textgoth{A}},\mathcal{A}}. \]
\end{prop}

\begin{proof}
Again, this is a direct consequence of Theorem \ref{thm_flowOTFT} and Proposition \ref{prop_dglanoncommOTFT}.
\end{proof}

\subsubsection{The large $N$ string/gauge theory correspondence} \label{sec_largeN}

It was discovered by 't~Hooft in \cite{tHooftplanar} that when calculations in $U(N)$ gauge theories are expanded in powers of the rank $N$, the Feynman diagram expansion arranges and refines itself into a sum over diagrams describing the interactions of open strings in which the planar (genus zero) diagrams dominate. This is often understood to be perhaps the most general manifestation of the string/gauge theory correspondence, cf. \cite[Ch. 1]{MAGOO}.

A noncommutative BV-theory is defined through the noncommutative renormalization group flow, which is in turn described through the use of stable ribbon graphs. These graphs parameterize orbi-cells in a certain compactification of the moduli space of Riemann surfaces with marked points \cite{KontAiry, LooCompact}; in other words, they describe open string interactions in which the ends of the open string are fixed to the marked points. Put another way, a noncommutative BV-theory is a kind of open string theory.

The transformations described in Proposition \ref{prop_noncommtocommthy} and \ref{prop_OTFTthy} provide a realization of this string/gauge theory correspondence. If $\mathcal{E}$ is a free BV-theory with a family of gauge-fixing operators $Q^{\mathrm{GF}}$ and we take $\textgoth{A}$ to be the Frobenius algebra $\mat{N}{\gf}$, then we have the following diagram of theories:
\begin{equation} \label{dig_largeN}
\xymatrix{
& \NCThy{\mathcal{E},\mathcal{A}} \ar^<<<<<<<<{\ldots}_>>>>>>>{\sigma_{\gamma,\nu}\circ\Morita[\mat{1}{\gf}]}[ld] \ar_{\ldots}^>>>{\, \sigma_{\gamma,\nu}\circ\Morita[\mat{N}{\gf}]}[d] \ar^>>>>>>{\qquad \sigma_{\gamma,\nu}\circ\Morita[\mat{N+1}{\gf}] \ \ldots}[rd] \\
\Thy{\mathcal{E}_{\mat{1}{\gf}},\mathcal{A}}\ldots & \ldots\Thy{\mathcal{E}_{\mat{N}{\gf}},\mathcal{A}} & \Thy{\mathcal{E}_{\mat{N+1}{\gf}},\mathcal{A}}\ldots
}
\end{equation}

There are no horizontal arrows in Diagram \eqref{dig_largeN}. In other words, there is no natural way to pass between theories at different ranks $N$. It is the notion of a noncommutative theory that unifies the theories for different ranks and enables the consideration of large $N$ phenomena. From Example \ref{exm_matrixOTFT} we see that the parameter $\nu$ is sent to $N\hbar$ under the above maps, while the parameter $\gamma$ is mapped to $\hbar^2$. Hence we see that at fixed order in $\hbar$, the genus zero diagrams dominate in the large $N$ expansion.

The results of Section \ref{sec_vanishcriteria} enable us to draw conclusions about the corresponding gauge theories from the properties of the noncommutative theory and vice-versa. We will apply this principle in Section \ref{sec_NCCSTheory} to the noncommutative analogue of Chern-Simons theory that was constructed in Section \ref{sec_NCChernSimonsCME}.

\subsection{Homotopy theory}

We now wish to provide a definition for when two noncommutative theories are equivalent. As in \cite[\S 5.10]{CosEffThy}, this is described in terms of abstract homotopy theory.

\subsubsection{Simplicial homotopy theory} \label{sec_simplicialhtopthy}

Let $\mathcal{E}$ be a free BV-theory. Recall that our convention has been to take our commutative topological algebra $\mathcal{A}$ to be the de Rham algebra on a standard $n$-simplex $\Delta^n$. Doing so leads to a simplicial set of gauge-fixing operators for the free BV-theory $\mathcal{E}$. The noncommutative theories over $\mathcal{E}$ can then likewise be turned into a simplicial set that is fibered over the simplicial set of gauge-fixing operators. This allows us to engage the apparatus of simplicial homotopy theory. We will use the standard notation $d_i$ to denote the face maps of a simplicial set.

\begin{defi}
Let $\mathcal{E}$ be a free BV-theory and let
\[ Q_0^{\mathrm{GF}},Q_1^{\mathrm{GF}}:\mathcal{E}\to\mathcal{E} \]
be two gauge-fixing operators for $\mathcal{E}$. We will say that two noncommutative theories
\[ I^0\in\NCThy{\mathcal{E};Q_0^{\mathrm{GF}}} \quad\text{and}\quad I^1\in\NCThy{\mathcal{E};Q_1^{\mathrm{GF}}} \]
are \emph{equivalent} if there is a family of gauge-fixing operators
\[ Q^{\mathrm{GF}}:\mathcal{E}\to\mathcal{E}\cotimes\Omega(\Delta^1,\gf) \]
and a family of theories
\[ I\in\NCThy{\mathcal{E},\Omega(\Delta^1,\gf);Q^{\mathrm{GF}}} \]
such that:
\begin{itemize}
\item
$d_0(Q^{\mathrm{GF}})=Q_0^{\mathrm{GF}}$ and $d_1(Q^{\mathrm{GF}})=Q_1^{\mathrm{GF}}$,
\item
$d_0(I)=I^0$ and $d_1(I)=I^1$.
\end{itemize}
\end{defi}

\subsubsection{Large length scale limit and solutions to the QME on cohomology}

Under certain conditions a noncommutative theory will have a large length scale limit, as in \cite[\S 5.10.7]{CosEffThy}. This limit will provide a solution to the (noncommutative) quantum master equation on the cohomology of the free BV-theory. The algebraic structure so formed will be a type of deformation of the structure of a cyclic $\ai$-algebra \cite{KontFeyn}. This was described in \cite[\S 5]{HamQME} where it was shown that such structures produce classes in a compactification of the moduli space of Riemann surfaces; cf. \cite{baran, HamQME, KontFeyn}. We begin here by describing a sufficient condition for such convergence to take place.

\begin{defi}
Let $\mathcal{E}$ be a free BV-theory. We will say that a gauge-fixing operator
\[ Q^{\mathrm{GF}}:\mathcal{E}\to\mathcal{E} \]
is \emph{nonnegative} if there is a Hermitian/Euclidean\footnote{Depending upon whether the ground field $\gf$ is $\mathbb{C}$ or $\mathbb{R}$ respectively.} metric $\dblinnprod$ on the vector bundle $E$ such that:
\begin{itemize}
\item
the metric $\dblinnprod$ is homogeneous in the sense that there is a metric on each graded component $E^n$, and $\dblinnprod$ is the sum of all such metrics;
\item
the operators $Q$ and $Q^{\mathrm{GF}}$ are adjoint under the pairing induced by $\dblinnprod$; that is,
\begin{equation} \label{eqn_nonnegadjoint}
\int_M\dblinnprod[Qs_1,s_2]\mathrm{d}\varrho = \int_M\dblinnprod[s_1,Q^{\mathrm{GF}}s_2]\mathrm{d}\varrho, \quad\text{for all }s_1,s_2\in\mathcal{E};
\end{equation}
where $\mathrm{d}\varrho$ is the density constructed from the Riemannian metric on $M$ that is associated to the generalized Laplacian $H=[Q,Q^{\mathrm{GF}}]$.
\end{itemize}
\end{defi}

\begin{example}
Let $M$ be a compact oriented Riemannian three-manifold and consider the Frobenius algebra $\textgoth{A}:=\mat{N}{\gf}$. We consider the free part of our noncommutative Chern-Simons theory defined by Example \ref{exm_NCCStheory}, which had space of fields
\[ \mathcal{E} = \Sigma\dRham{M}\underset{\mathbb{R}}{\otimes}\textgoth{A} \]
and operator $Q:=d_{\mathrm{DR}}\otimes\mathds{1}$. We define the gauge-fixing operator according to Example~\ref{exm_metricgaugefixing} by
\[ Q^{\mathrm{GF}}:=(-1)^k\ast Q\ast:\Sigma\dRham[k]{M}\underset{\mathbb{R}}{\otimes}\textgoth{A}\longrightarrow \Sigma\dRham[k-1]{M}\underset{\mathbb{R}}{\otimes}\textgoth{A}, \]
where now $\ast$ combines the standard Hodge star operator defined by the Riemannian metric with the standard star operator on $\textgoth{A}$ that takes a matrix to its conjugate transpose.

Define a metric on the vector bundle $E:=\Sigma\Lambda^{\bullet}T^*M\otimes_{\mathbb{R}}\textgoth{A}$ by
\[ \dblinnprod[\omega_k,\omega'_k] := \overline{\Tr\bigl(\ast(\omega_k\wedge\ast\omega'_k)\bigr)}, \quad\omega_k,\omega'_k\in\Lambda^k T^*M\underset{\mathbb{R}}{\otimes}\textgoth{A}. \]
Then the operators $Q$ and $Q^{\mathrm{GF}}$ satisfy Equation \eqref{eqn_nonnegadjoint}.
\end{example}

If $Q^{\mathrm{GF}}:\mathcal{E}\to\mathcal{E}$ is a nonnegative gauge-fixing operator then it follows from the results of \cite{BerGetVer} that:
\begin{itemize}
\item
The operator $D:=Q+Q^{\mathrm{GF}}$ is a self-adjoint Dirac operator with respect to the pairing induced by the metric $\dblinnprod$.
\item
Consequently, the generalized Laplacian $H:=[Q,Q^{\mathrm{GF}}]=D^2$ has nonnegative eigenvalues.
\item
There is, by Proposition 3.48 of \cite{BerGetVer}, a Hodge decomposition
\begin{displaymath}
\begin{split}
\mathcal{E} &= \Ker(D)\oplus\Image(D), \\
\Ker(D) &= \Ker(H) = \Ker(Q)\cap\Ker(Q^{\mathrm{GF}}), \\
\Image(D) &= \Image(H) = \Image(Q)\oplus\Image(Q^{\mathrm{GF}}).
\end{split}
\end{displaymath}
\item
Since $H$ is self-adjoint with respect to the skew-symmetric pairing $\innprod$, the subspaces $\Ker(H)$ and $\Image(H)$ are orthogonal with respect to this pairing as well, and not just the one induced by $\dblinnprod$. In particular, the pairing $\innprod$ must be nondegenerate on $\Ker(H)$.
\item
The corresponding inclusions and projections onto the kernel
\begin{equation} \label{eqn_homologymaps}
i_0:\Ker(H)\rightleftharpoons\mathcal{E}:\pi_0
\end{equation}
are quasi-isomorphisms of $Q$-complexes and hence the cohomology $\mathcal{H}:=\mathds{H}(\mathcal{E},Q)$ may be canonically identified with $\Ker(H)$, which must be finite-dimensional by Proposition 2.36 of \cite{BerGetVer}.
\item
Since the skew-symmetric pairing $\innprod$ on $\mathcal{E}$ is nondegenerate on cohomology, there exists an inverse pairing
\[ \innprod_0^{-1}\in\Ker(H)\otimes\Ker(H)\subset\mathcal{E}\cotimes\mathcal{E} \]
such that $\pi_0=\bigl(\innprod_0^{-1}\bigr)\star$. Denote the corresponding inverse pairing on $\mathcal{H}$ by $\innprod^{-1}_{\mathcal{H}}$.
\item
Denoting the heat kernel of the generalized Laplacian $H$ by $K_t$, it follows from Proposition 2.37 of \cite{BerGetVer}:
\begin{itemize}
\item
that the heat kernel $K_t$ converges to $\innprod_0^{-1}$ as $t\to\infty$, and
\item
the canonical propagator
\[ P(\varepsilon,L):= \int_{t=\varepsilon}^L \bigl(Q^{\mathrm{GF}}\cotimes\mathds{1}\bigr)[K_t]\,\mathrm{d}t \]
defined by Definition \ref{def_canonicalpropagator} converges as $L\to\infty$ (the only source of divergences are the zero modes, but these are annihilated by the operator $Q^{\mathrm{GF}}$). We will denote the limit by $P(\varepsilon,\infty)$.
\end{itemize}
\end{itemize}

As a consequence, we have well-defined differential graded Lie algebras
\[ \bigl(\intnoncomm{\mathcal{E}},Q+\Delta_{\infty},\{-,-\}_{\infty}\bigr) \quad\text{and}\quad \bigl(\intnoncomm{\mathcal{H}},\Delta_{\mathcal{H}},\{-,-\}_{\mathcal{H}}\bigr) \]
where
\begin{align*}
\{I,J\}_{\infty} &:= -\{I,J\}_{\innprod_0^{-1}}, & \Delta_{\infty}(I) &:= -\Delta_{\innprod_0^{-1}}(I); \\
\{I,J\}_{\mathcal{H}} &:= -\{I,J\}_{\innprod^{-1}_{\mathcal{H}}}, & \Delta_{\mathcal{H}}(I) &:= -\Delta_{\innprod^{-1}_{\mathcal{H}}}(I).
\end{align*}

The maps shown in \eqref{eqn_homologymaps} define maps of these differential graded Lie algebras
\[ i_0^{\dag}:\intnoncomm{\mathcal{E}}\rightleftharpoons\intnoncomm{\mathcal{H}}:\pi_0^{\dag}. \]

\begin{theorem}
Let $\mathcal{E}$ be a free BV-theory with a nonnegative gauge-fixing operator $Q^{\mathrm{GF}}:\mathcal{E}\to\mathcal{E}$. Given a noncommutative theory $I\in\NCThy{\mathcal{E}}$:
\begin{itemize}
\item
The limit
\[ I[\infty] :=\lim_{L\to\infty} I[L] \in \intnoncommI{\mathcal{E}} \]
exists.
\item
This limit satisfies the scale $\infty$ quantum master equation
\[ (Q+\Delta_{\infty})I[\infty] + \frac{1}{2}\{I[\infty],I[\infty]\}_{\infty} = 0. \]
\item
As such, it defines a solution $I_{\mathcal{H}}:=i_0^{\dag}(I[\infty])$ to the quantum master equation in $\intnoncomm{\mathcal{H}}$,
\[ \Delta_{\mathcal{H}}I_{\mathcal{H}} + \frac{1}{2}\{I_{\mathcal{H}},I_{\mathcal{H}}\}_{\mathcal{H}} = 0. \]
\end{itemize}
\end{theorem}

\begin{proof}
From the renormalization group equation we see that $I[L]$ converges to
\[ I[\infty] = \lim_{L\to\infty} W(I[\varepsilon],P(\varepsilon,L)) = W(I[\varepsilon],P(\varepsilon,\infty)). \]
Taking the limit of the scale $L$ quantum master equation \eqref{eqn_scaleQME} satisfied by $I[L]$ and applying Lemma \ref{lem_compositionconverge} we get
\begin{displaymath}
\begin{split}
0 &= \lim_{L\to\infty}\Bigl[(Q+\Delta_L)I[L] + \frac{1}{2}\{I[L],I[L]\}_L\Bigr] \\
&= (Q+\Delta_{\infty})I[\infty] + \frac{1}{2}\{I[\infty],I[\infty]\}_{\infty}.
\end{split}
\end{displaymath}
\end{proof}

\begin{rem}
Recall from Proposition \ref{prop_noncommtocommthy} that the map $\sigma_{\gamma,\nu}$ defined by Definition \ref{def_mapNCtoCom} transforms noncommutative theories into commutative theories. As has already been mentioned, Costello showed in \cite{CosEffThy} that the $L\to\infty$ limit of such a commutative theory yields a solution in $\intcomm{\mathcal{H}}$ to the corresponding quantum master equation. This solution will be the image under $\sigma_{\gamma,\nu}$ of the solution in $\intnoncomm{\mathcal{H}}$ defined above.

The same reasoning may be applied to the map $\Morita$ defined by a differential graded Frobenius algebra $\textgoth{A}$, which transforms noncommutative theories according to Proposition \ref{prop_OTFTthy}. The corresponding solutions to the quantum master equation on the cohomology will likewise correspond under $\Morita$.

Applying this to Diagram \eqref{dig_largeN}, which describes our large $N$ correspondence, we get a diagram
\begin{displaymath}
\xymatrix{
& \intnoncomm{\mathcal{H}} \ar^<<<<<<<<{\ldots}_>>>>>>>{\sigma_{\gamma,\nu}\circ\Morita[\mat{1}{\gf}]}[ld] \ar_{\ldots}^>>>{\, \sigma_{\gamma,\nu}\circ\Morita[\mat{N}{\gf}]}[d] \ar^>>>>>>{\qquad \sigma_{\gamma,\nu}\circ\Morita[\mat{N+1}{\gf}] \ \ldots}[rrd] \\
\intcomm{\mathcal{H}_{\mat{1}{\gf}}}\ldots & \ldots\intcomm{\mathcal{H}_{\mat{N}{\gf}}} && \intcomm{\mathcal{H}_{\mat{N+1}{\gf}}}\ldots
}
\end{displaymath}
under which the associated solutions to the quantum master equation on cohomology correspond.
\end{rem}

\begin{rem}
Solutions to the quantum master equation in $\intnoncomm{E}$, where $E$ is a finite-dimensional symplectic vector space such as $\mathcal{H}$, were studied in \cite{HamQME}; cf. \cite{baran, KontFeyn}. The tree level part of such a structure consists of a cyclic $\ai$-algebra. In this way, these solutions may be viewed as a type of deformation of this cyclic $\ai$-structure.

In \cite{baran, HamQME, KontFeyn} various explanations were provided for how to produce classes in compactifications of the moduli space of Riemann surfaces from such structures. In this way, we observe that to any noncommutative theory (with a nonnegative gauge-fixing condition) there is a canonically associated family of classes living on a compactification of the moduli space of Riemann surfaces.
\end{rem}

\subsection{Obstruction theory}

Suppose that $\mathcal{E}$ is a free BV-theory with a family of gauge-fixing operators $Q^{\mathrm{GF}}$. We now consider the problem of extending a level $p$ theory to a level $(p+1)$ theory. Following \cite[\S 5.11]{CosEffThy}, we identify the relevant obstruction and the cohomology theory in which it lives. Underlying this approach is the filtration of the differential graded Lie algebra
\[ \mathfrak{g}:=\bigl(\intnoncommP{\mathcal{E},\mathcal{A}},Q-d_{\mathrm{DR}}+\Delta_L,\{-,-\}_L\bigr) \]
defined by Definition \ref{def_NCfiltration} and the corresponding short exact sequence
\begin{equation} \label{eqn_SESfiltration}
0 \to F_{p+1}\mathfrak{g}/F_{p+2}\mathfrak{g} \to \mathfrak{g}/F_{p+2}\mathfrak{g} \to \mathfrak{g}/F_{p+1}\mathfrak{g} \to 0.
\end{equation}
Note that under this filtration
\begin{equation} \label{eqn_filtrationidentities}
(Q-d_{\mathrm{DR}})F_p\mathfrak{g}\subset F_p\mathfrak{g}, \quad \Delta_L F_p\mathfrak{g}\subset F_{p+1}\mathfrak{g} \quad\text{and}\quad \{F_p\mathfrak{g},F_q\mathfrak{g}\}_L\subset F_{p+q}\mathfrak{g}.
\end{equation}

The natural quotient maps formed from the filtration give us a canonical diagram
\begin{displaymath}
\xymatrix{
\NCThy{\mathcal{E},\mathcal{A}} \ar[d] \ar[rd] \\
\NCThyL{p}{\mathcal{E},\mathcal{A}} & \NCThyL{p+1}{\mathcal{E},\mathcal{A}} \ar[l]
}
\end{displaymath}
providing the formulation of our lifting problem. Note that by Theorem \ref{thm_locinteffthy}, any level $p$ \emph{pretheory} lifts to a level $(p+1)$ \emph{pretheory}. Proposition \ref{prop_prethyliftgrpact} parameterizes all such possible lifts.

Suppose that we start with a tree-level theory, which by Theorem \ref{thm_locinteffthy} and Proposition \ref{prop_treeCME} must be defined by Equation \eqref{eqn_treelevelcorrespondence} from a local tree-level interaction
\[ I^{\mathrm{Tree}}\in\intnoncommLI{\mathcal{E},\mathcal{A}}/F_1 \intnoncommLI{\mathcal{E},\mathcal{A}} = \intnoncommLItree{\mathcal{E},\mathcal{A}} \subset \KontHamL{\mathcal{E},\mathcal{A}} \]
satisfying the classical master equation \eqref{eqn_treeCME}. Denote by $\intnoncommLP{\mathcal{E},\mathcal{A}}$ the subspace of $\intnoncommP{\mathcal{E},\mathcal{A}}$ consisting of all local functionals. We will see that the cohomology theory describing the obstructions to lifting our tree-level theory to a full theory will be provided by the complexes
\begin{equation} \label{eqn_obscomplex}
\bigl(F_{p+1}\intnoncommLP{\mathcal{E},\mathcal{A}}/F_{p+2}\intnoncommLP{\mathcal{E},\mathcal{A}},Q-d_{\mathrm{DR}}+\{I^{\mathrm{Tree}},-\}\bigr), \quad p\geq 0.
\end{equation}
Note that the above definition for the differential of the complex makes sense by Remark~\ref{rem_bracketextension}.

\subsubsection{Definition of the obstruction at length scale $L$}

\begin{defi}
Let $\mathcal{E}$ be a free BV-theory with a family of gauge-fixing operators $Q^{\mathrm{GF}}$. Given a level $p$ theory $I\in\NCThyL{p}{\mathcal{E},\mathcal{A}}$, choose a lift $\tilde{I}$ of $I$ to a level $(p+1)$ pretheory and define
\begin{multline} \label{eqn_obstruction}
O_{p+1}[L] := (Q-d_{\mathrm{DR}}+\Delta_L)\tilde{I}[L] + \frac{1}{2}\bigl\{\tilde{I}[L],\tilde{I}[L]\bigr\}_L \\
\in F_{p+1}\intnoncommP{\mathcal{E},\mathcal{A}}/F_{p+2}\intnoncommP{\mathcal{E},\mathcal{A}}, \quad L>0;
\end{multline}
where we have used the short exact sequence \eqref{eqn_SESfiltration}.
\end{defi}

Note that this obstruction depends upon the choice of lift $\tilde{I}$ to a level $(p+1)$ pretheory. The next proposition describes how the obstruction depends upon this choice.

\begin{prop} \label{prop_obsscaleLclass}
The obstruction is a well-defined cohomology class
\[ O_{p+1}[L]\in\mathds{H}\bigl(F_{p+1}\intnoncommP{\mathcal{E},\mathcal{A}}/F_{p+2}\intnoncommP{\mathcal{E},\mathcal{A}},Q-d_{\mathrm{DR}}+\{I_{[0]}[L],-\}_L\bigr), \]
where $I_{[0]}$ denotes the tree-level theory underlying the level $p$ theory $I$.
\end{prop}

\begin{proof}
The fact that $O_{p+1}[L]$ is a cocycle is a consequence of the filtration identities \eqref{eqn_filtrationidentities} and the general fact that if $x$ is a degree zero element in a differential graded Lie algebra $\mathfrak{g}$ (with an odd bracket) then
\[ dO(x)+[x,O(x)] = 0, \quad\text{where }O(x):=dx+\frac{1}{2}[x,x]. \]

Now if we add $W_{p+1}\in F_{p+1}\intnoncommI{\mathcal{E},\mathcal{A}}/F_{p+2}\intnoncommI{\mathcal{E},\mathcal{A}}$ to $\tilde{I}[L]$ then our obstruction changes by
\begin{multline} \label{eqn_obschangecobdryscaleL}
(Q-d_{\mathrm{DR}}+\Delta_L)\bigl[\tilde{I}[L]+W_{p+1}\bigr] + \frac{1}{2}\bigl\{\tilde{I}[L]+W_{p+1},\tilde{I}[L]+W_{p+1}\bigr\}_L = \\
O_{p+1}[L] + (Q-d_{\mathrm{DR}})W_{p+1} + \bigl\{I_{[0]}[L],W_{p+1}\bigr\}_L \mod F_{p+2},
\end{multline}
where again we have used the filtration identities \eqref{eqn_filtrationidentities}. Hence the cohomology class is independent of the choice of lift $\tilde{I}$.
\end{proof}

\subsubsection{Transformation property of the scale $L$ obstruction}

To get a definition for the obstruction that does not depend upon the length parameter $L$ will require us to understand how the obstruction transforms according to the length scale. We will presently see that it transforms according to the renormalization group flow. However, since the obstruction has degree one it will be necessary to multiply it by a formal parameter $\zeta$ of degree minus-one in order to apply the flow. For this, we will use the commutative algebra
\[ \mathcal{A}_{\zeta} := \mathcal{A}\otimes\gf[\zeta]. \]
Note that the commutativity constraint imposes the relation $\zeta^2=0$, so that $\gf[\zeta]$ is two-dimensional.

\begin{prop} \label{prop_obstransform}
Let $\mathcal{E}$ be a free BV-theory with a family of gauge-fixing operators $Q^{\mathrm{GF}}:\mathcal{E}\to\mathcal{E}\cotimes\mathcal{A}$. Given a level $p$ theory $I\in\NCThyL{p}{\mathcal{E},\mathcal{A}}$, define the obstruction~\eqref{eqn_obstruction} using a lift $\tilde{I}$ of $I$ to a level $(p+1)$ pretheory. Then this obstruction transforms according to the formula
\begin{equation} \label{eqn_obstransform}
\zeta O_{p+1}[L'] = \sum_{\Gamma\in\ptree[p+1]}\frac{\gamma^{g(\Gamma)}\nu^{b(\Gamma)}}{|\Aut(\Gamma)|} w_{\Gamma}\bigl(I_{[0]}[L]+\zeta O_{p+1}[L],P(L,L')\bigr), \quad\text{for all }L,L'>0;
\end{equation}
where the sum is taken over all $(p+1)$-trees $\Gamma$.
\end{prop}

\begin{proof}
Applying Equation \eqref{eqn_QMEflow} of Theorem \ref{thm_QMEflow} to the renormalization group equation~\eqref{eqn_RGequation} of Definition \ref{def_NCprethy} yields
\[ O_{p+1}[L'] = \sum_{(\Gamma,v)}\frac{\gamma^{g(\Gamma)}\nu^{b(\Gamma)}}{|\Aut(\Gamma,v)|}w_{(\Gamma,v)}\bigl(\tilde{I}[L],O_{p+1}[L];P(L,L')\bigr). \]
Of course, the equation remains true if we multiply both $O_{p+1}[L']$ and $O_{p+1}[L]$ by $\zeta$. Since we are working in $F_{p+1}$ modulo $F_{p+2}$, we need only sum over graphs $\Gamma$ having loop number $(p+1)$. Likewise, the terms in the above sum will vanish unless the vertex $v$ also has loop number $(p+1)$, in which case by Equation \eqref{eqn_loopnum} the loop number of all the other vertices of $\Gamma$ must vanish, along with the first Betti number $\mathbf{b}_1(\Gamma)$; that is, $\Gamma$ must be a $(p+1)$-tree. Since $v$ will be the only vertex of $\Gamma$ with a nonzero loop number, any automorphism of $\Gamma$ must fix $v$, and the sum above reduces to \eqref{eqn_obstransform}.
\end{proof}

\subsubsection{Locality of the obstruction}

The next property that we must demonstrate, in order to get a definition for the obstruction that is independent of $L$, is that $O_{p+1}[L]$ is an asymptotically local functional in the sense of Remark \ref{rem_asymplocal}.

\begin{prop} \label{prop_obslocal}
Let $\mathcal{E}$ be a free BV-theory with a family of gauge-fixing operators $Q^{\mathrm{GF}}$. As before, define the obstruction $O_{p+1}[L]$ for a level $p$ theory $I\in\NCThyL{p}{\mathcal{E},\mathcal{A}}$ using a lift $\tilde{I}$ of $I$ to a level $(p+1)$ pretheory. Then $O_{p+1}[L]$ is an asymptotically local functional as $L\to 0$.
\end{prop}

\begin{proof}
By Theorem \ref{thm_locinteffthy} there is a local interaction $\hat{I}$ such that the pretheory
\[ \tilde{I}[L] = \hat{I}^R[L] = \lim_{\varepsilon\to 0}W\bigl(\hat{I}-\hat{I}^{\mathrm{CT}}(\varepsilon),P(\varepsilon,L)\bigr). \]
First we note that since $\tilde{I}[L]$ is asymptotically local, the same will be true after we apply $(Q-d_{\mathrm{DR}})$. Hence we need only show that
\[ \Delta_L\tilde{I}[L] + \frac{1}{2}\bigl\{\tilde{I}[L],\tilde{I}[L]\bigr\}_L = -\lim_{\varepsilon\to 0}\bigg[\sum_{(\Gamma,e)} \frac{\gamma^{g(\Gamma)}\nu^{b(\Gamma)}}{|\Aut(\Gamma,e)|}w_{(\Gamma,e)}\bigl(\hat{I}-\hat{I}^{\mathrm{CT}}(\varepsilon);P(\varepsilon,L),K_L\bigr)\bigg] \]
is asymptotically local, where in the above we have applied Equation \eqref{eqn_QMEflowaux6}.

In \cite[App. 1]{CosEffThy} it is explicitly stated in Remark 4.2 following Theorem 4.0.2 that the Feynman weight appearing in the above sum has a small $\varepsilon$ asymptotic expansion in terms of asymptotically local functionals---\emph{at least when we replace $K_L$ by $K_{\varepsilon}$}. To obtain the same statement for $K_L$ in place of $K_{\varepsilon}$, we note that those same results also imply such an expansion exists when we replace $K_L$ by
\[ -\int_{t=\varepsilon}^L (\mathds{1}_{\mathcal{E}\cotimes\mathcal{E}}\cotimes\mu)(\mathds{1}\cotimes\tau\cotimes\mathds{1})(H\cotimes\mathds{1}_{\mathcal{E}\cotimes\mathcal{A}})K_t\,\mathrm{d}t = \int_{t=\varepsilon}^L\frac{\mathrm{d}}{\mathrm{d}t}K_t\,\mathrm{d}t = K_L-K_{\varepsilon}. \]
We then simply add the two expansions together to get the expansion for the Feynman weight in the sum above.
\end{proof}

\subsubsection{Definition of the obstruction and its associated cohomology theory} \label{sec_obsdef}

We now have everything in place that we need to provide a definition for the obstruction that does not depend on the length scale $L$.

\begin{prop} \label{prop_obsdef}
Let $\mathcal{E}$ be a free BV-theory with a family of gauge-fixing operators $Q^{\mathrm{GF}}:\mathcal{E}\to\mathcal{E}\cotimes\mathcal{A}$ and define the obstruction $O_{p+1}[L]$ for a level $p$ theory $I\in\NCThyL{p}{\mathcal{E},\mathcal{A}}$ using a lift $\tilde{I}$ of $I$ to a level $(p+1)$ pretheory:
\begin{enumerate}
\item \label{itm_obsdef1}
The family of interactions
\[ \tilde{I}[L]+\zeta O_{p+1}[L]\in\intnoncommI{\mathcal{E},\mathcal{A}_{\zeta}}, \quad L>0; \]
form a level $(p+1)$ pretheory.
\item \label{itm_obsdef2}
There is a unique $O_{p+1}\in F_{p+1}\intnoncommLP{\mathcal{E},\mathcal{A}}/F_{p+2}\intnoncommLP{\mathcal{E},\mathcal{A}}$ such that
\begin{equation} \label{eqn_obsdef}
\zeta O_{p+1}[L] = \sum_{\Gamma\in\ptree[p+1]}\frac{\gamma^{g(\Gamma)}\nu^{b(\Gamma)}}{|\Aut(\Gamma)|} \lim_{\varepsilon\to 0}w_{\Gamma}\bigl(I^{\mathrm{Tree}}+\zeta O_{p+1},P(\varepsilon,L)\bigr),
\end{equation}
where the sum is taken over all $(p+1)$-trees, and
\begin{equation} \label{eqn_treethylimit}
I^{\mathrm{Tree}} := \lim_{L\to 0}I_{[0]}[L]
\end{equation}
denotes the local tree-level interaction underlying the tree-level theory $I_{[0]}$, see Remark \ref{rem_treelimit}. The obstruction $O_{p+1}$ has degree one.
\item \label{itm_obsdef3}
We have
\[ \lim_{L\to 0} O_{p+1}[L] = O_{p+1}. \]
\end{enumerate}
\end{prop}

\begin{proof}
Item \eqref{itm_obsdef1} follows directly from Proposition \ref{prop_filterformula}, the renormalization group equation and Equation \eqref{eqn_obstransform} of Proposition \ref{prop_obstransform}. Asymptotic locality follows from Proposition~\ref{prop_obslocal}.

Now $(\tilde{I}[L]+\zeta O_{p+1}[L])$ and $\tilde{I}[L]$ are both level $(p+1)$ pretheories that lift $I$. By Theorem \ref{thm_locinteffthy} they must be related by the action of the group $\mathcal{G}_{p+1}$ defined by \eqref{eqn_liftgroup} that is described by Proposition \ref{prop_prethyliftgrpact}. From this, \eqref{itm_obsdef2} follows.

Item \eqref{itm_obsdef3} now follows from taking the limit as $L\to 0$ in Equation \eqref{eqn_obsdef} and applying Lemma \ref{lem_treecterterm}\eqref{itm_treecterterm1}. In the limit, the sum degenerates into a sum over corollas and yields $\zeta O_{p+1}$ by Equation \eqref{eqn_intgrpactzero}.
\end{proof}

Of course, as before, the obstruction $O_{p+1}$ still depends upon a choice of lift $\tilde{I}$; but as we shall presently see, it determines a well-defined cohomology class in the complex \eqref{eqn_obscomplex}.

\begin{theorem} \label{thm_obsclass}
Let $\mathcal{E}$ be a free BV-theory with a family of gauge-fixing operators $Q^{\mathrm{GF}}$ and define the obstruction $O_{p+1}$ for a level $p$ theory $I$ using a lift $\tilde{I}$ of $I$ to a level $(p+1)$ pretheory:
\begin{enumerate}
\item \label{itm_obsclass1}
The obstruction is a cocycle,
\[ (Q-d_{\mathrm{DR}})O_{p+1} + \{I^{\mathrm{Tree}},O_{p+1}\} = 0; \]
where $I^{\mathrm{Tree}}$ is the local tree-level interaction underlying the tree-level theory defined by \eqref{eqn_treethylimit}.
\item \label{itm_obsclass2}
Choosing a different lift of $I$ to a level $(p+1)$ pretheory only changes the obstruction by a coboundary.
\item \label{itm_obsclass3}
Consequently, to every level $p$ theory $I$, there is associated a well-defined obstruction
\[ O_{p+1}\in\mathds{H}\bigl(F_{p+1}\intnoncommLP{\mathcal{E},\mathcal{A}}/F_{p+2}\intnoncommLP{\mathcal{E},\mathcal{A}},Q-d_{\mathrm{DR}}+\{I^{\mathrm{Tree}},-\}\bigr). \]
\end{enumerate}
\end{theorem}

\begin{proof}
Note that \eqref{itm_obsclass1} will follow from Proposition \ref{prop_obsscaleLclass} and Proposition \ref{prop_obsdef}\eqref{itm_obsdef3} by taking the limit as $L\to 0$ of the equation that expresses the fact that $O_{p+1}[L]$ is a cycle, providing that we can show that
\[ \lim_{L\to 0}\{I_{[0]}[L],O_{p+1}[L]\}_L = \{I^{\mathrm{Tree}},O_{p+1}\}. \]
Some care should be used in taking the above limit, as the kernel $K_L$ defining the bracket diverges as $L\to 0$.

Using Equation \eqref{eqn_treelevelcorrespondence}, Equation \eqref{eqn_obsdef} and Proposition \ref{prop_kerneledgecontract} we obtain the expression
\begin{multline*}
\zeta\{I_{[0]}[L],O_{p+1}[L]\}_L = \\
\sum_{\Gamma,\Gamma'}\sum_e \frac{\gamma^{g(\Gamma\sqcup\Gamma'+e)}\nu^{b(\Gamma\sqcup\Gamma'+e)}}{|\Aut(\Gamma)||\Aut(\Gamma')|} \lim_{\varepsilon\to 0}\Bigl[w_{(\Gamma\sqcup\Gamma'+e,e)}\bigl(I^{\mathrm{Tree}}+\zeta O_{p+1};P(\varepsilon,L),K_L\bigr)\Bigr];
\end{multline*}
where the sum is taken over all trees $\Gamma$, all $(p+1)$-trees $\Gamma'$, and all pairs of half-edges $e$ in which one half-edge is taken from the legs of $\Gamma$ and the other from the legs of $\Gamma'$. Now, since both $I^{\mathrm{Tree}}$ and $O_{p+1}$ are local functionals and $\Gamma\sqcup\Gamma'$ is a $(p+1)$-tree, an argument similar to Lemma \ref{lem_treecterterm}\eqref{itm_treecterterm1} (cf. Lemma 4.13 of \cite{NCRGF}) shows that when we take the limit as $L\to 0$ on the right hand side above, the sum degenerates into a sum over corollas and it follows from Equation \eqref{eqn_bracketlimit} that
\[ \lim_{L\to 0}\{I_{[0]}[L],O_{p+1}[L]\}_L = \lim_{L\to 0}\{I^{\mathrm{Tree}},O_{p+1}\}_L = \{I^{\mathrm{Tree}},O_{p+1}\}. \]

Now to prove \eqref{itm_obsclass2}, note that if $\tilde{I}'$ is a different lift of $I$ to a level $(p+1)$ pretheory then it follows from Theorem \ref{thm_locinteffthy} and Proposition \ref{prop_prethyliftgrpact} that there is an $\omega_{p+1}\in\mathcal{G}_{p+1}$---where $\mathcal{G}_{p+1}$ was defined by \eqref{eqn_liftgroup}---for which
\begin{equation} \label{eqn_intactchange}
\tilde{I}'[L] = \tilde{I}[L]  + \sum_{\Gamma\in\ptree[p+1]}\frac{\gamma^{g(\Gamma)}\nu^{b(\Gamma)}}{|\Aut(\Gamma)|} \lim_{\varepsilon\to 0}w_{\Gamma}\bigl(I^{\mathrm{Tree}}+\omega_{p+1},P(\varepsilon,L)\bigr).
\end{equation}
Denoting the sum on the right-hand side above by $W_{p+1}[L]$, it follows from Equation \eqref{eqn_obschangecobdryscaleL} that the obstruction $O_{p+1}[L]$ changes by
\[ O'_{p+1}[L] = O_{p+1}[L] + (Q-d_{\mathrm{DR}})W_{p+1}[L] + \bigl\{I_{[0]}[L],W_{p+1}[L]\bigr\}_L. \]
Note that it follows from Lemma \ref{lem_treecterterm}\eqref{itm_treecterterm1} that $W_{p+1}[L]$ converges to $\omega_{p+1}$ as $L\to 0$. Therefore, by Proposition \ref{prop_obsdef}\eqref{itm_obsdef3}, to prove that $O'_{p+1}$ and $O_{p+1}$ differ by a coboundary we need only show that the bracket
\begin{multline*}
\bigl\{I_{[0]}[L],W_{p+1}[L]\bigr\}_L = \\
\sum_{\Gamma,\Gamma'}\sum_e \frac{\gamma^{g(\Gamma\sqcup\Gamma'+e)}\nu^{b(\Gamma\sqcup\Gamma'+e)}}{|\Aut(\Gamma)||\Aut(\Gamma')|} \lim_{\varepsilon\to 0}\Bigl[w_{(\Gamma\sqcup\Gamma'+e,e)}\bigl(I^{\mathrm{Tree}}+\omega_{p+1};P(\varepsilon,L),K_L\bigr)\Bigr]
\end{multline*}
converges to $\{I^{\mathrm{Tree}},\omega_{p+1}\}$ as $L\to 0$; where as before, we have applied Proposition \ref{prop_kerneledgecontract} to obtain the above expression, in which we likewise sum over all trees $\Gamma$, all $(p+1)$-trees $\Gamma'$, and all pairs of half-edges $e$, as we did previously. The aforementioned convergence, however, just follows from the same degeneration argument that we used earlier, yielding the equation
\begin{equation} \label{eqn_obschangecobdry}
O'_{p+1} = O_{p+1} + (Q-d_{\mathrm{DR}})\omega_{p+1} + \bigl\{I^{\mathrm{Tree}},\omega_{p+1}\bigr\}.
\end{equation}
\end{proof}

\subsubsection{Vanishing obstructions and lifting level $p$ theories}

The importance of the obstruction defined in the preceding sections is that it controls the ability to lift a level $p$ theory to a level $(p+1)$ theory.

\begin{theorem} \label{thm_obsvanish}
Let $\mathcal{E}$ be a free BV-theory with a family of gauge-fixing operators $Q^{\mathrm{GF}}$:
\begin{enumerate}
\item \label{itm_obsvanish1}
A level $p$ theory $I\in\NCThy{\mathcal{E},\mathcal{A}}$ has a lift to a level $(p+1)$ theory if and only if the obstruction cohomology class $O_{p+1}$ defined by Proposition \ref{prop_obsdef}\eqref{itm_obsdef2} vanishes.
\item \label{itm_obsvanish2}
The fiber of level $(p+1)$ theories $\tilde{I}$ that sit over a given level $p$ theory $I$ is acted upon freely and transitively by the abelian group of degree zero cocycles $\omega_{p+1}$ of the complex \eqref{eqn_obscomplex} according to the formula
\[ \tilde{I}[L] \longmapsto \tilde{I}[L] + \sum_{\Gamma\in\ptree[p+1]}\frac{\gamma^{g(\Gamma)}\nu^{b(\Gamma)}}{|\Aut(\Gamma)|} \lim_{\varepsilon\to 0}w_{\Gamma}\bigl(I^{\mathrm{Tree}}+\omega_{p+1},P(\varepsilon,L)\bigr),
 \]
where the sum is taken over all $(p+1)$-trees $\Gamma$.
\end{enumerate}
\end{theorem}

\begin{proof}
Obviously, if $I$ lifts to a level $(p+1)$ theory $\tilde{I}$, then the representative $O_{p+1}[L]$ for the obstruction that is defined by \eqref{eqn_obstruction} will be zero, and taking the limit as $L\to 0$ we conclude that $O_{p+1}$ vanishes.

Conversely, if the cohomology class vanishes then we may write
\[ O_{p+1} + (Q-d_{\mathrm{DR}})\omega_{p+1} + \bigl\{I^{\mathrm{Tree}},\omega_{p+1}\bigr\} = 0, \]
for some $\omega_{p+1}\in\mathcal{G}_{p+1}$. Replacing the level $(p+1)$ pretheory $\tilde{I}$ used to define $O_{p+1}$ with the level $(p+1)$ pretheory $\tilde{I}'$ defined by Equation \eqref{eqn_intactchange}, our new obstruction $O'_{p+1}$ associated to this pretheory, which is given by Equation \eqref{eqn_obschangecobdry}, will vanish. It then follows from Equation \eqref{eqn_obsdef} that $O'_{p+1}[L]$ vanishes for all $L$; that is, $\tilde{I}'[L]$ satisfies the scale $L$ quantum master equation modulo $F_{p+2}$ for all $L>0$. This proves \eqref{itm_obsvanish1}.

The same arguments apply to prove \eqref{itm_obsvanish2}. We know by Theorem \ref{thm_locinteffthy} and Proposition \ref{prop_prethyliftgrpact} that $\mathcal{G}_{p+1}$ acts freely and transitively on the fiber of level $(p+1)$ pretheories that sit over $I$ according to the formula \eqref{eqn_intactchange}. Equation \eqref{eqn_obschangecobdry} describes how the representative for the obstruction changes under this action, and a pretheory is a theory if and only if this representative vanishes (as a cochain).
\end{proof}

\subsection{Independence of the choice of gauge}

One immediate and straightforward application of the obstruction theory that we have developed in the preceding section is to prove the independence of a theory in the BV-formalism on the choice of gauge-fixing condition, cf. \cite[\S 5.11.2]{CosEffThy}. This is expressed in the language of simplicial homotopy theory as the statement that the space of theories forms a fibration over the space of gauge-fixing conditions.

Recall from Section \ref{sec_simplicialhtopthy} that taking our commutative algebra $\mathcal{A}$ to be the de Rham algebra on an $n$-simplex $\Delta_n$, the collection of all noncommutative theories on a given free BV-theory will form a simplicial set that is fibered over the simplicial set of gauge-fixing operators.

\begin{theorem}
Given any free BV-theory $\mathcal{E}$ and local tree-level interaction
\[ I^{\mathrm{Tree}}\in\intnoncommLItree{\mathcal{E}} \]
satisfying the classical master equation
\[ QI^{\mathrm{Tree}} + \frac{1}{2}\{I^{\mathrm{Tree}},I^{\mathrm{Tree}}\} = 0, \]
the simplicial set of theories in the noncommutative BV-formalism that have fixed local tree-level interaction $I^{\mathrm{Tree}}$---that is, that satisfy Equation \eqref{eqn_treethylimit}---forms a Kan fibration over the simplicial set of gauge-fixing operators.
\end{theorem}

\begin{proof}
We recall that by Theorem \ref{thm_locinteffthy} we may identify pretheories with local interactions, which we will henceforth do for the purpose of this proof. Under this correspondence, the action of the group of degree zero cocycles of the complex \eqref{eqn_obscomplex} on level $(p+1)$ theories that is provided by Theorem \ref{thm_obsvanish}\eqref{itm_obsvanish2} corresponds, by Proposition \ref{prop_prethyliftgrpact}, to simply adding such a cocycle to the local interaction. Note that this group of cocycles forms a simplicial abelian group and hence must be a Kan complex, see Theorem 17.1 of \cite{MaySimplicial}. We will use this fact later in our proof.

Our starting data consists of:
\begin{itemize}
\item
A family of gauge-fixing operators
\[ Q^{\mathrm{GF}}:\mathcal{E}\to\mathcal{E}\cotimes\dRham{\Delta_n}. \]
\item
A sequence of theories
\[ \bigl[I^i\bigr]^{\mathrm{R}}\in\NCThy{\mathcal{E},\dRham{\Delta_{n-1}};d_iQ^{\mathrm{GF}}}, \quad 0\leq i\leq n, \ i\neq k; \]
represented by local interactions $I^i\in\intnoncommLI{\mathcal{E},\dRham{\Delta_{n-1}}}$ satisfying:
\begin{displaymath}
\begin{array}{ll}
d_i I^j = d_{j-1}I^i, &\text{for all }0\leq i<j\leq n\text{ with }i,j\neq k; \\
I^i_{[0]} = I^{\mathrm{Tree}}, &\text{for all }0\leq i\leq n\text{ with }i\neq k.
\end{array}
\end{displaymath}
\end{itemize}
To fill in the horn, we must construct a theory
\[ I^{\mathrm{R}}\in\NCThy{\mathcal{E},\dRham{\Delta_n};Q^{\mathrm{GF}}} \]
represented by a local interaction $I\in\intnoncommLI{\mathcal{E},\dRham{\Delta_n}}$ satisfying:
\begin{equation} \label{eqn_hornfill}
\begin{array}{ll}
d_i I = I^i, &\text{for all }0\leq i\leq n\text{ with }i\neq k; \\
I_{[0]} = I^{\mathrm{Tree}}.
\end{array}
\end{equation}

To do so, we construct by induction a sequence of local interactions
\[ I(\!(p)\!)\in\intnoncommLI{\mathcal{E},\dRham{\Delta_n}}/F_{p+1}\intnoncommLI{\mathcal{E},\dRham{\Delta_n}}, \quad p\geq 0 \]
satisfying the following conditions:
\begin{itemize}
\item
the initial interaction $I(\!(0)\!)$ is the interaction $I^{\mathrm{Tree}}$,
\item
the interaction $I(\!(p+1)\!)$ lifts $I(\!(p)\!)$ under the canonical projection map, for all $p\geq 0$;
\item
the level $p$ pretheory $I(\!(p)\!)^{\mathrm{R}}$ defined by the local interaction $I(\!(p)\!)$ is a theory; that is, it satisfies the quantum master equation;
\item
the equation
\begin{equation} \label{eqn_levelpface}
d_i I(\!(p)\!) = I^{i} \mod F_{p+1}
\end{equation}
holds for all $p\geq 0$ and all $0\leq i\leq n$ with $i\neq k$.
\end{itemize}
Using the completeness of the filtration, we may then find a local interaction $I$ that agrees with $I(\!(p)\!)$ modulo $F_{p+1}$ for all $p\geq 0$, and hence will satisfy the conditions required by \eqref{eqn_hornfill}.

To construct such a sequence, suppose that $I(\!(p)\!)$ has been constructed with the required properties. Since the interactions $I^i$ all define theories, it follows from Equation \eqref{eqn_levelpface} and Theorem \ref{thm_obsvanish}\eqref{itm_obsvanish1} that the obstruction cohomology class $d_i O_{p+1}$ associated to the theory defined by the interaction $d_i I(\!(p)\!)$ must vanish, where $O_{p+1}$ is the obstruction that is associated to the interaction $I(\!(p)\!)$. However, all the face maps $d_i$ defined on the complex \eqref{eqn_obscomplex} are quasi-isomorphisms, and hence the obstruction $O_{p+1}$ itself must vanish. From this it follows that the theory defined by $I(\!(p)\!)$ has a lift to a level $(p+1)$ theory, defined by some local interaction
\[ J(\!(p+1)\!)\in\intnoncommLI{\mathcal{E},\dRham{\Delta_n}}/F_{p+2}\intnoncommLI{\mathcal{E},\dRham{\Delta_n}} \]
that lifts $I(\!(p)\!)$.

Nonetheless, the gluing condition \eqref{eqn_levelpface} may not hold for this choice of interaction, so for $0\leq i\leq n$ with $i\neq k$, we set
\[ \omega_{p+1}^i := I^i - d_i J(\!(p+1)\!) \in F_{p+1}\intnoncommLP{\mathcal{E},\dRham{\Delta_{n-1}}}/F_{p+2}\intnoncommLP{\mathcal{E},\dRham{\Delta_{n-1}}}; \]
which must be a cocycle by Theorem \ref{thm_obsvanish}\eqref{itm_obsvanish2}. By the extension property of simplicial abelian groups, there is a cocycle
\[ \omega_{p+1}\in F_{p+1}\intnoncommLP{\mathcal{E},\dRham{\Delta_n}}/F_{p+2}\intnoncommLP{\mathcal{E},\dRham{\Delta_n}} \]
such that $d_i\omega_{p+1} = \omega_{p+1}^i$, for all $0\leq i\leq n$ with $i\neq k$. Now setting
\[ I(\!(p+1)\!) := J(\!(p+1)\!) + \omega_{p+1} \]
completes the inductive step.
\end{proof}

\section{Noncommutative Chern--Simons theory} \label{sec_NCCSTheory}

In Section \ref{sec_NCChernSimonsCME} we briefly described how to define an analogue of Chern-Simons theory in our framework of noncommutative geometry. This subject was also considered in \cite[\S 6]{NCRGF}. There it was shown that this noncommutative analogue yields $U(N)$ Chern-Simons theory at all ranks $N$ under the large $N$ correspondence defined by Diagram \eqref{dig_largeN}, at least as far as \emph{pretheories} are concerned; the issue of quantization in the BV-formalism was not considered there. We take up this problem in this section.

Since the complexification of the real Lie algebra of $N$-by-$N$ skew Hermitian matrices is $\mathfrak{gl}_N(\mathbb{C})$, Chern-Simons theory for the gauge group $U(N)$ is the same as that for the group $GL_N(\mathbb{C})$, see \cite[\S 6.1.2]{NCRGF}. We adopt the latter point of view for the remainder of this section, and hence work over the ground field $\mathbb{C}$.

In this section we will not consider families of theories or gauge-fixing conditions. Hence our commutative algebra will be the ground field $\mathcal{A}:=\mathbb{C}$.

\subsection{A noncommutative analogue of Chern-Simons theory}

The analogue of Chern-Simons theory that we will construct in this section will be very simple-minded. The noncommutative geometry will arise as a result of treating the Chern-Simons interactions through open string diagrams; that is, using ribbon graphs rather than ordinary graphs. This is in contrast to those approaches, cf. \cite{SussCS}, that deform the Chern-Simons interaction using a Moyal star product; which is what is more typically understood when the term ``noncommutative field theory'' is employed \cite{NoteonNCCSThy}. Instead, our treatment is more in keeping with the approach of \cite{CosTCFT} and Witten's treatment of Chern-Simons theory in \cite{WitCSstring}.

\subsubsection{Noncommutative Chern-Simons theory} For the remainder of the paper, we will work over a compact oriented Riemannian three-manifold $M$, although the results of this section remain true for any manifold of odd dimension at least three, providing that we are willing to exchange a $\mathbb{Z}$-grading for a $\mathbb{Z}/2\mathbb{Z}$-grading.

\begin{defi}
The free noncommutative Chern-Simons theory on $M$ is the theory defined by Example \ref{exm_NCCStheory} for the Frobenius algebra $\textgoth{A}:=\mathbb{C}$, and hence has space of fields
\[ \mathcal{E}^{\mathrm{NC}} := \Sigma\dRham{M,\mathbb{C}}. \]
As in Example \ref{exm_CStheory}, the operator $Q$ will be the exterior derivative. The Riemannian metric defines the Hodge adjoint of $Q$, which is a gauge-fixing operator that we denote by $Q^{\mathrm{GF}}$.

A local tree-level interaction $I^{\mathrm{NC}}\in\intnoncommLItree{\mathcal{E}^{\mathrm{NC}}}$ satisfying the classical master equation \eqref{eqn_CMEinteraction} may be defined by Equation \eqref{eqn_CSinteraction}, where of course now the matrices are just scalars. Applying the maps from Diagram \eqref{dig_largeN} to this interaction defines local interactions
\[ I^{U(N)} := \sigma_{\gamma,\nu}\Morita[\mat{N}{\mathbb{C}}](I^{\mathrm{NC}}) \in \intcommLI{\mathcal{E}^{U(N)}}, \]
where $\mathcal{E}^{U(N)}:=\Sigma\dRham{M,\mathbb{C}}\otimes\mathfrak{gl}_N(\mathbb{C})$ is the space of fields for $U(N)$ Chern-Simons theory. Using the formula from Example \ref{exm_matrixOTFT}, we see that these are just the standard Chern-Simons interactions for the gauge group $U(N)$.
\end{defi}

\subsubsection{Large $N$ correspondence}

Since the interactions $I^{\mathrm{NC}}$ and $I^{U(N)}$ are related through the maps of diagram \eqref{dig_largeN}, the same will be true for the pretheories which they define through the formula of Definition \ref{def_intrenormalized};
\begin{equation} \label{eqn_CScorrespondence}
(I^{U(N)})^{\mathrm{R}}[L] = \sigma_{\gamma,\nu}\Morita[\mat{N}{\mathbb{C}}]\bigl((I^{\mathrm{NC}})^{\mathrm{R}}[L]\bigr), \quad\text{for all }N\geq 1;
\end{equation}
see Proposition 6.1 of \cite{NCRGF}. (For the sake of clarity, we emphasize that these pretheories define interactions in $\intcomm{\mathcal{E}^{U(N)}}$ and $\intnoncomm{\mathcal{E}^{\mathrm{NC}}}$; that is, they are not just tree-level pretheories.)

From this correspondence, we may draw conclusions about our noncommutative Chern-Simons theory using the vanishing criteria from Section \ref{sec_vanishcriteria}. One such result, which is Proposition 6.3 of \cite{NCRGF}, describes the vanishing of the counterterms.

\begin{prop}
If the Riemannian metric on $M$ is flat, then the counterterms
\[ (I^{\mathrm{NC}})^{\mathrm{CT}}(\varepsilon)\in\intnoncomm{\mathcal{E}^{\mathrm{NC}}} \]
for our noncommutative Chern-Simons theory vanish modulo constants; that is, for all $\varepsilon>0$,
\[ (I^{\mathrm{NC}})^{\mathrm{CT}}(\varepsilon)\in\mathbb{C}[[\gamma,\nu]]. \]
\end{prop}

\subsection{Theories modulo constants}

We will use the correspondence \eqref{eqn_CScorrespondence} and Costello's results from \cite{CosBVrenormalization} to deduce that our noncommutative theory is quantizable. One curious feature of Costello's results is that they only hold modulo constants. Here we will explain what that means.

\begin{defi}
If $\mathcal{E}$ is a free BV-theory then we say that a functional $I\in\intnoncomm{\mathcal{E}}$ is \emph{constant} if
\[ I_{ijk} = 0, \quad\text{for all }k\geq 1. \]
We will denote the subspace of $\intnoncomm{\mathcal{E}}$ consisting of all constants simply by
\[ \intnoncommconstants\subset\mathbb{C}[[\gamma,\nu]], \]
since this is the space that we would recover from $\intnoncomm{\mathcal{E}}$ by setting $\mathcal{E}=0$. Likewise, $\intnoncommIconstants$ will denote the corresponding subspace of $\intnoncommI{\mathcal{E}}$.
\end{defi}

\subsubsection{The quantum master equation modulo constants}

Since the constants form a differential graded ideal of the differential graded Lie algebra that arises from Proposition \ref{prop_dglanoncomm}, the quantum master equation makes sense modulo constants.

\begin{defi}
If $\mathcal{E}$ is a free BV-theory with a gauge-fixing operator $Q^{\mathrm{GF}}$, then we will say that an interaction $I\in\intnoncommI{\mathcal{E}}/\intnoncommIconstants$ satisfies the scale $L$ quantum master equation \emph{modulo constants} if
\[ (Q+\Delta_L)I + \frac{1}{2}\{I,I\}_L \in \mathbb{C}[[\gamma,\nu]]. \]
\end{defi}

\subsubsection{The renormalization group flow modulo constants}

Since the only connected graphs that have a vertex of valency zero are corollas with no legs, it follows that the renormalization group flow is well-defined modulo constants and we hence have a well-defined map
\[ W(-,P):\intnoncommI{\mathcal{E}}/\intnoncommIconstants\longrightarrow\intnoncommI{\mathcal{E}}/\intnoncommIconstants, \]
defined for any free BV-theory $\mathcal{E}$ and propagator $P\in\mathcal{E}\cotimes\mathcal{E}$.

\begin{defi}
Let $\mathcal{E}$ be a free BV-theory with a gauge-fixing operator $Q^{\mathrm{GF}}$. We will say that a family of interactions
\[ I[L]\in\intnoncommI{\mathcal{E}}/\intnoncommIconstants, \quad L>0 \]
is a theory in the noncommutative BV formalism \emph{modulo constants} if:
\begin{itemize}
\item
the renormalization group equation \eqref{eqn_RGequation} holds,
\item
each $I_{ijkl}[L]$ is asymptotically local as $L\to 0$, and
\item
for all $L>0$, the interaction $I[L]$ satisfies the scale $L$ quantum master equation modulo constants.
\end{itemize}
\end{defi}

\subsubsection{The powerseries associated to a theory modulo constants} \label{sec_powerseries}

If $I[L]$ is a theory modulo constants, then Equation \eqref{eqn_QMEflow} of Theorem \ref{thm_QMEflow} and Equation \eqref{eqn_intgrpactzero} of Theorem \ref{thm_intgrpact} together imply that the powerseries
\[ (Q+\Delta_L)I[L] + \frac{1}{2}\{I[L],I[L]\}_L \in \mathbb{C}[[\gamma,\nu]] \]
does not depend on the length scale parameter $L$. Hence, we may canonically associate a powerseries to any theory modulo constants.

\begin{rem}
The notion of a theory modulo constants has an obvious commutative analogue for families of interactions taken from $\intcomm{\mathcal{E}}$. In this case, the powerseries assigned to such a theory will belong to the ring $\mathbb{C}[[\hbar]]$. From Proposition \ref{prop_dglanoncommtocomm} and \ref{prop_dglanoncommOTFT}, along with the formula from Example \ref{exm_matrixOTFT}, it then follows that the powerseries $p_{U(N)}(\hbar)$ for those commutative theories that are obtained from a noncommutative theory modulo constants through the large $N$ correspondence described by Diagram \eqref{dig_largeN}, may be calculated for all $N$ from the powerseries $p^{\mathrm{NC}}(\gamma,\nu)$ for the noncommutative theory by making the substitutions $\gamma=\hbar^2$ and $\nu=N\hbar$;
\[ p_{U(N)}(\hbar) = \hbar^{-1}p^{\mathrm{NC}}(\hbar^2,N\hbar), \quad\text{for all }N\geq 1. \]
\end{rem}

\subsection{Quantization of Chern-Simons in a flat metric}

We now deduce, as a consequence of our large $N$ correspondence, that our noncommutative Chern-Simons interaction $I^{\mathrm{NC}}$ defines a noncommutative theory modulo constants when the Riemannian metric on the manifold $M$ is flat.

\subsubsection{Results from $U(N)$ Chern-Simons theory}

We will deduce the above promised result from the corresponding result for $U(N)$ Chern-Simons theory, which was proved (for a general gauge group) in Theorem 15.1.3 of \cite[\S 15]{CosBVrenormalization}.

\begin{theorem} \label{thm_quantizeU(N)}
For a compact oriented flat Riemannian three-manifold $M$, the $U(N)$ Chern-Simons interaction $I^{U(N)}\in\intcommdelim{(\mathcal{E}^{U(N)})}$ defines a (commutative) theory modulo constants $(I^{U(N)})^{\mathrm{R}}$; that is, the interactions $(I^{U(N)})^{\mathrm{R}}[L]$ satisfy the scale $L$ quantum master equation modulo constants, for all $L>0$.
\end{theorem}

\subsubsection{Application of the large $N$ correspondence}

Since all of the $U(N)$ theories are quantizable, the same will be true of our noncommutative theory.

\begin{theorem} \label{thm_quantizeNC}
For a compact oriented flat Riemannian three-manifold $M$, the noncommutative Chern-Simons interaction $I^{\mathrm{NC}}\in\intnoncommdelim{(\mathcal{E}^{\mathrm{NC}})}$ defines a (noncommutative) theory modulo constants $(I^{\mathrm{NC}})^{\mathrm{R}}$.
\end{theorem}

\begin{proof}
We need only show that the interactions $(I^{\mathrm{NC}})^{\mathrm{R}}[L]\in\intnoncommdelim{(\mathcal{E}^{\mathrm{NC}})}$ satisfy the scale $L$ quantum master equation modulo constants. This follows from the large $N$ correspondence determined by Equation \eqref{eqn_CScorrespondence}, as a simple corollary of Theorem \ref{thm_quantizeU(N)} and Corollary~ \ref{cor_vanishingmodconstants}.
\end{proof}

\subsubsection{Generalization to curved spacetimes}

We emphasize that stronger results than Theorem \ref{thm_quantizeNC} are possible. In \cite[\S 15]{CosBVrenormalization} Costello uses Theorem \ref{thm_quantizeU(N)} to prove that a canonical quantization for Chern-Simons theory on a curved spacetime exists, for a general gauge group. To do so however requires him to prove that theories in the BV formalism form a sheaf over the manifold $M$.

This invites us to employ the same approach to our noncommutative Chern-Simons theory, and prove that it has a canonical quantization on any curved spacetime. However, to introduce the extra technology required to carry this out would significantly lengthen the current article, and so we will postpone such a treatment to a future paper.

\appendix
\section{Topological vector spaces}

In this appendix we recall some basic definitions and facts about topological vector spaces that we use throughout the paper. Here we will be very brief, and refer the reader desiring additional details to \cite[App. A]{NCRGF}, or the original cited sources, such as \cite{Treves}.

Consider the space $\mathcal{E}:=\Gamma(M,E)$ of smooth sections of some vector bundle $E$ over a smooth manifold $M$.

\begin{defi} \label{def_Cinftopology}
For any open neighborhood $U$ of the zero section in the $i$th jet bundle $J^i(E)$ of $E$, and any compact subset $K$ of $M$, consider the subset of $\mathcal{E}$ defined by
\[ \mathscr{G}^i(K,U):=\{\gamma\in\mathcal{E}:\gamma^i(K)\subset U\}, \]
where $\gamma^i$ denotes the prolongation of $\gamma$ to $J^i(E)$. Allowing $K$, $U$ and $i$ to vary generates a basis of neighborhoods of zero for the $C^{\infty}$-topology on $\mathcal{E}$.
\end{defi}

With the $C^{\infty}$-topology, $\mathcal{E}$ is a nuclear Fr\'echet space. The completed projective tensor product of two such spaces may again be identified as the space of sections of the external tensor product.

\begin{prop} \label{prop_tensorsections}
There is a canonical isomorphism of topological vector spaces
\begin{displaymath}
\begin{array}{ccc}
\Gamma(M,E)\cotimes\Gamma(N,F) & \cong & \Gamma(M\times N,E\boxtimes F) \\
(\xi,\eta) & \longmapsto & \left[(x,y)\mapsto\xi(x)\otimes\eta(y)\right]
\end{array}
\end{displaymath}
\end{prop}

Hence, the resulting space will again be a nuclear Fr\'echet space---a fact which is occasionally useful.

Any nuclear Fr\'echet space is a Montel space, and hence reflexive; see Corollary 3 of Proposition 50.2 and the corollary of Proposition 36.9 in \cite{Treves}.

Given locally convex Hausdorff topological vector spaces $\mathcal{U}$ and $\mathcal{V}$, the two main topologies on $\Hom_{\gf}(\mathcal{U},\mathcal{V})$ are:
\begin{itemize}
\item
the \emph{weak} topology of pointwise convergence, and
\item
the \emph{strong} topology of uniform convergence on bounded sets.
\end{itemize}
The Banach-Steinhaus Theorem \cite[Ch. 33]{Treves} often allows us to conclude convergence in the latter from the former.

\begin{lemma} \label{lem_compositionconverge}
Let $\mathcal{U}$ and $\mathcal{V}$ be nuclear Fr\'echet spaces and $\mathcal{W}$ be a locally convex Hausdorff topological vector space and suppose that as $t\to 0$, we have weakly converging one-parameter families of linear operators;
\[ \phi_t:\mathcal{U}\to\mathcal{V},\quad\psi_t:\mathcal{V}\to\mathcal{W}; \qquad t\in (0,1). \]
Then the one-parameter family $\psi_t\circ\phi_t:\mathcal{U}\to\mathcal{W}$ converges strongly as $t\to 0$.
\end{lemma}

Another property we make use of follows from Proposition 50.7 of \cite{Treves}.

\begin{prop} \label{prop_tensordual}
Given any two nuclear Fr\'echet spaces $\mathcal{U}$ and $\mathcal{V}$,
\[ (\mathcal{U}\cotimes\mathcal{V})^{\dag} = \mathcal{U}^\dag\cotimes\mathcal{V}^{\dag}; \]
\end{prop}

Finally, we make use of the following fact throughout the paper.

\begin{prop}
If $\mathcal{U}$ is a nuclear Fr\'echet space and $\mathcal{V}$ is a complete locally convex Hausdorff topological vector space then
\[ \mathcal{U}^{\dag}\cotimes\mathcal{V} = \Hom_{\gf}(\mathcal{U},\mathcal{V}). \]
\end{prop}

We emphasize that when we apply this identification throughout the text, we make use of the Koszul sign rule $(f\otimes v)[u]:=(-1)^{|v||u|}f(u)v$.


\begin{thebibliography}{ABC00}
%
\bibitem{MAGOO}
O. Aharony, S. Gubser, J. Maldacena, H. Ooguri, Y. Oz; \emph{Large $N$ field theories, string theory and gravity.} Phys. Rep. 323 (2000), no.3-4, 183--386.
%
\bibitem{AtiyahTFT}
M. Atiyah, \emph{Topological quantum field theories.} Inst. Hautes \'Etudes Sci. Publ. Math. (1988), no.68, 175--186.
%
\bibitem{AuKoGVLargeN}
D. Auckly, S. Koshkin; \emph{Introduction to the Gopakumar-Vafa large $N$ duality.} The interaction of finite-type and Gromov-Witten invariants (BIRS 2003), 195--456. Geom. Topol. Monogr., 8 Geometry \& Topology Publications, Coventry, 2006.
%
\bibitem{baran}
S. Barannikov, \emph{Modular operads and Batalin-Vilkovisky geometry.} Int Math Res Notices (2007) Vol. 2007, article ID rnm075.
%
\bibitem{BaViGauge}
I. Batalin, G. Vilkovisky; \emph{Gauge algebra and quantization.} Phys. Lett. B 102 (1981), no. 1, 27--31.
%
\bibitem{BerGetVer}
N. Berline, E. Getzler, M. Vergne; \emph{Heat kernels and Dirac operators.} Grundlehren Text Ed., Springer-Verlag, Berlin, 2004.
%
\bibitem{BriGOVconjecture}
A. Brini, \emph{On the Gopakumar-Ooguri-Vafa correspondence for Clifford-Klein 3-manifolds.} Topological recursion and its influence in analysis, geometry, and topology, 59--81. Proc. Sympos. Pure Math., 100, American Mathematical Society, Providence, RI, 2018.
%
\bibitem{ChLaOTFT}
J. Chuang, A. Lazarev; \emph{Dual Feynman transform for modular operads.} Commun. Number Theory Phys. 1 (2007), no.4, 605--649.
%
\bibitem{CiVoCSStrTop}
K. Cieliebak, E. Volkov; \emph{Chern-Simons theory and string topology.} \texttt{arXiv:2312.05922}.
%
\bibitem{CosTCFT}
K. Costello, \emph{Topological conformal field theories and gauge theories.} Geom. Topol. 11 (2007), 1539--1579.
%
\bibitem{CosBVrenormalization}
K. Costello, \emph{Renormalisation and the Batalin-Vilkovisky formalism.} \texttt{arXiv:0706.1533}.
%
\bibitem{CosEffThy}
K. Costello, \emph{Renormalization and effective field theory.} Math. Surveys Monogr. 170, American Mathematical Society, Providence, RI, 2011.
%
\bibitem{GerPreLie}
M. Gerstenhaber, \emph{The cohomology structure of an associative ring.} Ann. of Math. (2) Vol. 78, No. 2, 1963, pp. 267--288.
%
\bibitem{GerDeform}
M. Gerstenhaber, \emph{On the deformation of rings and algebras.} Ann. of Math. (2) 79 (1964), 59--103.
%
\bibitem{GetJonCyclic}
E. Getzler, J. Jones; \emph{$A_{\infty}$-algebras and the cyclic bar complex.} Illinois J. Math. 34 (1990), no. 2, 256--283.
%
\bibitem{GetKap}
E. Getzler, M. Kapranov; \emph{Modular operads.} Compositio Math. 110 (1998), no.1, 65--126.
%
\bibitem{GiGqHaZeLQT}
G. Ginot, O. Gwilliam, A. Hamilton, M. Zeinalian; \emph{Large $N$ phenomena and quantization of the Loday-Quillen-Tsygan theorem.} Adv. Math. 409 (2022).
%
\bibitem{GinzNC}
V. Ginzburg, \emph{Non-commutative symplectic geometry, quiver varieties, and operads.} Math. Res. Lett. 8 (2001), no. 3, 377--400.
%
\bibitem{GopVafa}
R. Gopakumar, C. Vafa; \emph{On the gauge theory/geometry correspondence.} Adv. Theor. Math. Phys. 3 (1999), no.5, 1415--1443.
%
\bibitem{HtopNC}
A. Hamilton, A. Lazarev; \emph{Homotopy algebras and noncommutative geometry.} \texttt{arXiv:math/0410621}.
%
\bibitem{HamCompact}
A. Hamilton, \emph{Noncommutative geometry and compactifications of the moduli space of curves.} J. Noncommut. Geom. 4 (2010), no. 2, 157--188.
%
\bibitem{HamQME}
A. Hamilton; \emph{Classes on compactifications of the moduli space of curves through solutions to the quantum master equation.} Lett. Math. Phys. 89 (2009), no. 2, 115--130.
%
\bibitem{HamNCBV}
A. Hamilton, \emph{Classes on the moduli space of Riemann surfaces through a noncommutative Batalin-Vilkovisky formalism.} Adv. Math. 243 (2013), 67--101.
%
\bibitem{NCRGF}
A. Hamilton, \emph{Noncommutative effective field theories and the large $N$ correspondence.} \texttt{arXiv:2505.13678}.
%
\bibitem{HarOrbicells}
J. Harer, \emph{The virtual cohomological dimension of the mapping class group of an orientable surface.}  Invent. Math.  84  (1986),  no. 1, 157--176.
%
\bibitem{KadScaling}
L. Kadanoff, \emph{Scaling laws for ising models near $T_c$.} Phys. Phys. Fiz. 2 (1966), no. 6, 263--272.
%
\bibitem{HiroCyclic}
H. Kajiura, \emph{Noncommutative homotopy algebras associated with open strings.} Rev. Math. Phys. 19 (2007), no. 1, 1--99.
%
\bibitem{KontAiry}
M. Kontsevich, \emph{Intersection theory on the moduli space of curves and the matrix Airy function.} Comm. Math. Phys. 147 (1992), no. 1, 1--23.
%
\bibitem{KontVass}
M. Kontsevich, \emph{Vassiliev's knot invariants.} I. M. Gelfand Seminar, 137--150. Adv. Soviet Math., 16, Part 2, American Mathematical Society, Providence, RI, 1993.
%
\bibitem{KontSympGeom}
M. Kontsevich, \emph{Formal noncommutative symplectic geometry.} The Gelfand Mathematical Seminars, 1990-1992, pp. 173--187, Birkh\"auser Boston, Boston, MA, 1993.
%
\bibitem{KontFeyn}
M. Kontsevich, \emph{Feynman Diagrams and Low-Dimensional Topology.} First European Congress of Mathematics, Vol. 2 (Paris, 1992), pp. 97--121, Progr. Math., Vol. 120, Birkh\"auser Basel, 1994.
%
\bibitem{LodayQuillen}
J-L. Loday, D. Quillen; \emph{Cyclic homology and the Lie algebra homology of matrices.} Comment. Math. Helv. 59, 565--591 (1984).
%
\bibitem{LooCompact}
E. Looijenga, \emph{Cellular decompositions of compactified moduli spaces of pointed curves.} The moduli space of curves (Texel Island, 1994), 369--400, Progr. Math., 129, Birkh\"auser Boston, Boston, MA, 1995.
%
\bibitem{OpAlgTopPhys}
M. Markl, S. Shnider, J. Stasheff; \emph{Operads in algebra, topology and physics.} Mathematical Surveys and Monographs, 96. American Mathematical Society, Providence, RI, 2002.
%
\bibitem{MaySimplicial}
P. May, \emph{Simplicial objects in algebraic topology.} Chicago Lectures in Math. University of Chicago Press, Chicago, IL, 1992.
%
\bibitem{Mondello}
G. Mondello, \emph{Combinatorial classes on $\overline{\mathcal{M}}_{g,n}$ are tautological.} Int. Math. Res. Not. 2004,  no. 44, 2329--2390.
%
\bibitem{movshevcobracket}
M. Movshev, \emph{Fukaya category with curves of higher genus.} \texttt{arXiv:math.SG/9911123}.
%
\bibitem{NijRicDeform}
A. Nijenhuis, R. Richardson; \emph{Cohomology and deformations in graded Lie algebras.} Bull. Amer. Math. Soc. 72 (1966), 1--29.
%
\bibitem{PenOrbicells}
R. C. Penner, \emph{The decorated Teichm\"uller space of punctured surfaces.} Comm. Math. Phys. 113  (1987), no. 2, 299--339.
%
\bibitem{PolRenormLag}
J. Polchinski, \emph{Renormalization and effective lagrangians.} Nuclear Physics B, Volume 231, Issue 2, 1984, 269--295.
%
\bibitem{SchwarzBV}
A. Schwarz, \emph{Geometry of Batalin-Vilkovisky quantization.} Comm. Math. Phys. 155 (1993), no. 2, 249--260.
%
\bibitem{NoteonNCCSThy}
M. Sheikh-Jabbari, \emph{A note on noncommutative Chern-Simons theories.} Phys. Lett. B 510 (2001), no. 1-4, 247--254.
%
\bibitem{StaBracket}
J. Stasheff, \emph{The intrinsic Bracket on the Deformation Complex of an Associative Algebra.} J. Pure Appl. Algebra Vol. 89, 1993, No. 1--2, pp. 231--235.
%
\bibitem{SussCS}
L. Susskind, \emph{The quantum Hall fluid and non-commutative Chern Simons theory.} \texttt{hep-th/0101029}, 2001.
%
\bibitem{tHooftplanar}
G. 't~Hooft, \emph{A planar diagram theory for strong interactions.} Nuclear Phys. B, Volume 72, Issue 3, 1974, 461--473.
%
\bibitem{Treves}
F. Tr\`eves, \emph{Topological vector spaces, distributions and kernels.} Dover Publications, Inc., Mineola, NY, 2006.
%
\bibitem{Tsygan}
B. Tsygan, \emph{Homology of matrix Lie algebras over rings and the Hochschild homology.} Uspekhi Mat. Nauk 38 (1983), no. 2(230), 217--218.
%
\bibitem{WilRGcrit}
K. Wilson, \emph{Renormalization group and critical phenomena, I.} Physical Review B 4 (1971), no. 9, 3174--3183.
%
\bibitem{WilRenormScalar}
K. Wilson, \emph{Renormalization of a scalar field theory in strong coupling.} Physical Review D 6 (1972), no. 2, 419--426.
%
\bibitem{WitCSstring}
E. Witten, \emph{Chern-Simons gauge theory as a string theory.} The Floer memorial volume, 637--678. Progr. Math., 133 Birkh\"auser Verlag, Basel, 1995.
%
\bibitem{ZwiClosedstring}
B. Zwiebach, \emph{Closed string field theory: quantum action and the Batalin-Vilkovisky master equation.} Nuclear Phys. B 390 (1993), no. 1, 33--152.
%
\end{thebibliography}
\end{document}